\documentclass[journal]{IEEEtran}
\usepackage{amsmath,amsfonts}
\usepackage{algorithmic}
\usepackage{algorithm}
\usepackage{array}
\usepackage[caption=false,font=normalsize,labelfont=sf,textfont=sf]{subfig}
\usepackage{textcomp}
\usepackage{stfloats}
\usepackage{url}
\usepackage{verbatim}
\usepackage{graphicx}
\usepackage{cite}
\usepackage{amsthm}
\usepackage{todonotes}
\usepackage{multirow}
\usepackage[caption=false,font=normalsize,
labelfont=sf,textfont=sf]{subfig}
\newtheorem{theorem}{Theorem}
\newtheorem{lemma}[theorem]{Lemma}
\newtheorem{proposition}[theorem]{Proposition}

\theoremstyle{definition}
\newtheorem{remark}{Remark}
\newcommand{\sgn}{\mathop{\mathrm{sgn}}}

\begin{document}

\title{A New Approach to Motion Planning in 3D for a Dubins Vehicle: Special Case on a Sphere
\\
\thanks{Deepak Prakash Kumar is with the Department of Electrical Engineering and Computer Science, University of California, Irvine, CA 92697, USA (e-mail: {\tt\small deepakprakash1997@gmail.com}).

Swaroop Darbha is with the Department of Mechanical Engineering, Texas A\&M University, College Station, TX 77843, USA (e-mail: {\tt\small dswaroop@tamu.edu}).

Satyanarayana Gupta Manyam is with the DCS Corporation, 4027 Col Glenn Hwy, Dayton, OH 45431, USA (e-mail: {\tt\small msngupta@gmail.com}).

David Casbeer is with the Control Science Center, Air Force Research Laboratory, Wright-Patterson Air Force Base, OH 45433 USA (e-mail:
{\tt\small david.casbeer@us.af.mil}).

DISTRIBUTION STATEMENT A. Approved for public release. Distribution is unlimited. AFRL-2024-2505; Cleared 05/08/2024.
}
}

\author{Deepak Prakash Kumar, \IEEEmembership{Member, IEEE}, Swaroop Darbha, \IEEEmembership{Fellow, IEEE}, Satyanarayana Gupta Manyam, \IEEEmembership{Senior Member, IEEE}, David W. Casbeer, \IEEEmembership{Senior Member, IEEE}}

\markboth{IEEE Transactions on Robotics}%
{Kumar \MakeLowercase{\textit{et al.}}: A New Approach to Motion Planning in 3D for a Dubins Vehicle: Special Case on a Sphere}


\maketitle

\bstctlcite{IEEEexample:BSTcontrol}

\begin{abstract}
In this article, a new model for 3D motion planning, applicable to aerial vehicles, is proposed to connect an initial and final configuration subject to pitch rate and yaw rate constraints. 
The motion planning problem for a curvature-constrained vehicle over the surface of a sphere is identified as an intermediary problem to be solved, and it is the focus of this paper. In this article, the optimal path candidates for a vehicle with a minimum turning radius $r$ moving over a unit sphere are derived using a phase portrait approach. We show that the optimal path is $CGC$ or concatenations of $C$ segments through simple proofs, where $C = L, R$ denotes a turn of radius $r$ and $G$ denotes a great circular arc. We generalize the previous result of optimal paths being $CGC$ and $CCC$ paths for $r \in \left(0, \frac{1}{2} \right]\bigcup\{\frac{1}{\sqrt{2}}\}$ to  
 $r \leq \frac{\sqrt{3}}{2}$ to account for vehicles with a larger $r$. 
We show that the optimal path is $CGC, CCCC,$ for $r \leq \frac{1}{\sqrt{2}},$ and $CGC, CC_\pi C, CCCCC$ for $r \leq \frac{\sqrt{3}}{2}.$
Additionally, we analytically construct all candidate paths and provide the code in a publicly accessible repository.
\end{abstract}


\section{Introduction}

\IEEEPARstart{D}{ue} to the substantially increasing civilian and military applications of curvature-constrained vehicles, such as fixed-wing unmanned aerial vehicles, path (or motion) planning for such vehicles has tremendous applications. Path planning involves identifying a sufficient list of candidate optimal paths for the vehicle to travel from one position and orientation, referred to as configuration, to another at the cheapest cost. Curvature-constrained vehicles represent a broad class of vehicles that have a constraint on the rate of change of their heading angle and need to maintain a minimum airspeed. A popular method for capturing the constraints of such a vehicle is by modeling the vehicle as a Dubins vehicle, wherein the vehicle is assumed to travel at a constant (unit) speed and have a minimum turning radius $r$. The Dubins model is a kinematic model used to generate the path for the vehicle to connect an initial and a final configuration in 2D plane \cite{Dubins}; the obtained path is then tracked by the vehicle using a controller \cite{small_unmanned_aircraft, tracking_dubins_path_uav}.

The Dubins model has been extensively explored in the literature for the path planning of such vehicles moving at a constant altitude (or a plane). In \cite{Dubins}, the path planning problem for a Dubins vehicle moving on a plane, referred to as the classical Markov-Dubins problem, was solved. The author showed that the optimal path to travel from one configuration to another is of type $CSC, CCC,$ or a degenerate path\footnotemark of the same. Here, $C = L\; \text{or} \;R$ denotes a left or right turn arc of minimum turning radius, and $S$ denotes a straight line segment. A variation of the path planning problem, wherein the vehicle moves forward or backward, known as the Reeds-Shepp vehicle, was studied in \cite{Reeds_Shepp}.
\footnotetext{A degenerate path denotes a subpath of the considered path. For example, a degenerate path of $CSC$ and $CCC$ paths is of type $C,$ $S,$ $CS,$ $SC,$ and $CC$.}

However, in both \cite{Dubins} and \cite{Reeds_Shepp}, the results were obtained without utilizing Pontryagin's Maximum Principle (PMP) \cite{PMP}. The same problems were later studied using PMP to simplify the proofs \cite{Shortest_path_synthesis_Boissonat, sussman_geometric_examples}. Furthermore, phase portraits have also been employed in recent studies to systematically characterize the optimal paths by partitioning the analysis into a fewer cases \cite{monroy, weighted_Markov_Dubins}. 

While many variants exist in the literature considering the vehicle to move on a plane \cite{sinistral/dextral, dubins_circle, weighted_Markov_Dubins}, the path planning problem from one configuration to another in 3D or other surfaces is less explored. A survey of the same follows.


\subsection{Planning in 3D}

Motion planning in 3D is an active area of interest due to its extensive applications for aerial vehicles, underwater vehicles, and robots \cite{3D_underwater, motion_planning_two_3D_Dubins_vehicles}. Similar to the motion planning over a plane, simple kinematic models are considered to generate paths in 3D. It should be noted that while specifying the heading angle suffices to describe the orientation of the vehicle in 2D, one needs to specify the heading angle and the plane containing the vehicle for the 3D problem to uniquely define its orientation. Using the generated path in 3D, lower-level controllers are used to guide the vehicle to travel over the generated paths using the kinematic model, such as demonstrated in \cite{path_generation_tracking_3D}. 
In addition, since the Dubins trajectories can be computed efficiently due to the existence of analytical solutions, they can be used for planning in complex environments as well. For instance, in \cite{rick_lind}, Rapidly-exploring Random Tree (RRT) was used to construct a 3D path for an aerial vehicle in the presence of static obstacles. Here, the planar Dubins problem solution and its free terminal variant \cite{bui_free_terminal}, wherein the final heading angle is free, were used to connect the projected configurations on a plane, and the elevation was calculated to attain the desired altitude. In particular, the free terminal variant of the 2D Dubins problem was used for the tree expansion steps (to connect intermediary locations), with a final 2D Dubins curve to connect the last node in the tree to the final configuration. A similar combination of RRT with the planar Dubins solution was utilized in \cite{path_planning_3D_Dubins_saripalli}, wherein motion planning in 3D in the presence of static and dynamic obstacles was performed. 
The authors demonstrated the application of this combination to obtain a feasible solution through experiments for different static and dynamic obstacle scenarios using an ARDrone. 




The earliest analytical exploration of the general 3D Dubins problem to the best of our knowledge traces to the work by Sussmann in \cite{sussman_3D}. In \cite{sussman_3D}, the author addressed a curvature-constrained 3D Dubins problem, wherein the goal was to connect given initial and final locations and heading directions. The author showed that the optimal path is either a helicoidal arc, or a path of the form $CSC, CCC,$ or a degenerate path. However, unlike the 2D problem, infinitely many $CSC$ paths exist in 3D since the plane containing the $C$ segments can be arbitrarily chosen to connect the initial and final locations and heading directions; hence, efficient path construction for the $CSC$ path has been an active area of interest.

In \cite{optimal_geometrical_path_in_3D}, the authors constructed a $CSC$ path through a geometrical and numerical approach for locations that are spaced sufficiently far apart to connect the given initial and final configurations. The numerical approach for constructing the $CSC$ path was improved upon in \cite{optimal_path_planning_aerial_vehicle_3D}, wherein the authors formulated a nonlinear optimization problem to construct a feasible solution. The proposed formulation, wherein the control input at discrete time steps serves as decision variables, was solved using a multiple-shooting method.
The numerical generation of a $CSC$ path was improved upon in \cite{reparametrization_3D_Dubins}, wherein the authors parametrized $CSC$ paths using two parameters, and numerically optimized over the parameters. The authors showed an improvement by showing additional valid $CSC$ paths being generated through their approach, particularly in cases wherein the initial and final locations are close. The $CSC$ path construction was alternately viewed in \cite{analytic_solution_3D} as an inverse kinematics problem for a robotic manipulator with five degrees of freedom. The authors derived analytical solutions for the path's parameters to reduce the computation time for constructing the path.



An alternate approach for the 3D problem was considered in \cite{time_optimal_paths_Dubins_airplane}, wherein the authors extend the Dubins car model~\cite{Dubins} to an airplane model in 3D. The airplane was modeled with two controllable inputs: one for the yaw rate, and another for the rate of change of the altitude.
The authors showed that the optimal path contains turns of minimum turning radius (corresponding to the bounded yaw rate), straight line segments, or a Dubins path with a certain length.
Furthermore, the authors generate feasible solutions depending on three modes of the vehicle depending on the altitude difference: low altitude, medium altitude, and high altitude. Here, additional turn segments are introduced to allow the vehicle to attain the desired final altitude, depending on the mode. The proposed model in \cite{time_optimal_paths_Dubins_airplane} was modified in \cite{Dubins_airplane_fixed_wing_UAVs} to more accurately represent the kinematics of the vehicle by considering constraints on the pitch angle of the vehicle. However, similar segments were used in \cite{Dubins_airplane_fixed_wing_UAVs} to connect a location and heading angle to another in 3D. Contrary to \cite{time_optimal_paths_Dubins_airplane}, the authors used helical segments for medium and high altitude modes for the vehicle to aid in the vehicle attaining the desired altitude. Furthermore, the authors demonstrated the benefit of such kinematics-based path construction for the guidance of UAVs by using a six-degree-of-freedom model derived in \cite{small_unmanned_aircraft}, and utilizing a vector-field-based guidance law, based on \cite{goncalves}, to track the path.

In recent studies, the 3D motion planning problem has also been studied by constructing the Dubins path on a plane and varying the turning radius to ensure that the desired altitude can be obtained. In \cite{minimal_3D_Dubins_path_bounded_curvature_pitch}, the authors consider a model with two constraints for the vehicle: a curvature constraint, and a minimum and maximum pitch angle constraint. The authors decouple the $CSC$ path construction to connect the given initial and final configuration into a horizontal and vertical component. Finally, through iterative optimization, the path is combined into a single feasible path. The proposed decoupling approach was improved upon in \cite{finding_3D_dubins_paths_pitch_angle_nonlinear_optimization}, wherein the decoupled approach to construct the solution in \cite{minimal_3D_Dubins_path_bounded_curvature_pitch} was used as an initial solution for solving a non-linear optimization problem. To this end, the final path is constructed using small segments such that the boundary conditions corresponding to the locations, heading vectors, and pitch angles are satisfied.



From the reviewed studies, it can be observed that the complete vehicle configuration, which includes the position and 3D orientation, which can be described by a rotation matrix, has not been considered for aerial vehicles. In previous studies, the orientation of the vehicle has predominantly been represented using a heading angle and/or a pitch angle. However, these two angles do not uniquely capture the orientation of the vehicle since the plane containing the vehicle is not uniquely specified. The study in \cite{towards_finding_shortest_paths_3D_rigid_bodies} considers the complete configuration for a general robot, wherein the authors employ a numerical search technique to find two-segment and three-segment paths to connect the initial and final configurations. However, the connection between the model considered in this paper and the kinematics of an aerial vehicle is not evident. To this end, we propose an alternate approach in which the 3D motion planning problem can be viewed.

\subsection{Proposed model and connection to motion planning on sphere}

Unlike motion planning in 2D, wherein the configuration can be specified using a location $(x, y) \in \mathbb{R}^2$ and a heading angle $\psi \in [0, 2\pi),$ the configuration in 3D is specified using a location vector $\mathbf{X},$ and three vectors $\mathbf{T}, \mathbf{Y},$ and $\mathbf{U}$ for its orientation. Here, $\mathbf{T}$, which is the tangent vector, represents the longitudinal direction of the vehicle, $\mathbf{Y},$ which is a tangent-normal vector, represents the lateral direction of the vehicle (chosen to be the direction to the left of the vehicle), and $\mathbf{U} := \mathbf{T} \times \mathbf{Y}$ is the surface-normal vector. Noting that the vectors $\mathbf{T}, \mathbf{Y},$ and $\mathbf{U}$ are unit and are mutually orthogonal, the orientation of the vehicle is specified by a rotation matrix $\mathbf{R} = \begin{bmatrix}
    \mathbf{T} & \mathbf{Y} & \mathbf{U}
\end{bmatrix}$. In this regard, the goal of motion planning in 3D is to connect a given initial configuration, specified by location $\mathbf{X}_i$ and orientation $\mathbf{R}_i,$ and a final configuration, specified by $\mathbf{X}_f$ and $\mathbf{R}_f$ with the least cost. In this article, the cost of interest is the distance or time taken to travel from one configuration to another. A depiction of the initial and final configurations of the vehicle is shown in Fig.~\ref{fig: osculating_spheres}.

\begin{figure}[htb!]
    \centering
    \includegraphics[width = \linewidth]{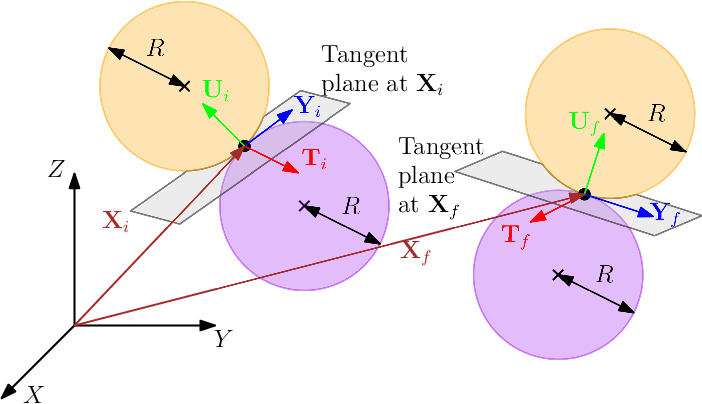}
    \caption{Depiction of osculating spheres at initial and final configuration corresponding to the pitch rate constraints}
    \label{fig: osculating_spheres}
\end{figure}

To address this path planning problem, an appropriate mathematical model is necessary to capture the motion constraints. To this end, we consider two motion constraints for the vehicle: a bounded pitch rate, and a bounded yaw rate. The bounded pitch rate and bounded yaw rate yield motion primitives for the vehicle, as shown in Fig.~\ref{fig: optimal_segments}. In this figure, pure pitch motion of the vehicle with the yaw rate set to zero yields great circular arcs on a sphere above the vehicle (along $\mathbf{U}_0$) and below the vehicle (along $-\mathbf{U}_0$), whereas pure yaw motion, wherein the pitch rate is set to zero, yields great circular arcs corresponding to left and right motion, denoted by $L_p$ and $R_p,$ respectively, on spheres whose centers lie on the $\mathbf{Y}_0$ axis. Furthermore, concurrent pitch and yaw motions with the values at the bounds yield circular arcs denoted by $L_{si}$ and $R_{si}$ corresponding to ascent motion and $L_{so}$ and $R_{so}$ for the descent motion. These segments also lie on the spheres corresponding to the pitch motion (whose centers are on the $\mathbf{U}_0$ axis) and the sphere corresponding to the yaw motion (whose centers are on the $\mathbf{Y}_0$ axis).

\begin{figure}[htb!]
    \centering
    \subfloat[Segments on spheres corresponding to max. pitch rate]{\includegraphics[width = 0.435\linewidth]{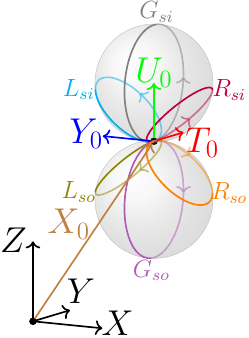}\label{subfig: sphere_segments}}
    \hfil
    \subfloat[Segments on spheres corresponding to max. yaw rate (and straight line segment)]{\includegraphics[width = 0.51\linewidth]{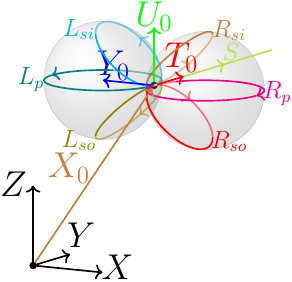}\label{subfig: planar_segments}}
    \caption{Visualization of segments}
    \label{fig: optimal_segments}
\end{figure}

Hence, it can be observed that the optimal path connecting the initial and final configurations can contain portions of its path on a sphere. However, for the path connecting the initial and final configurations to be optimal, the portion of the path lying on the sphere must be optimal as well. To this end, it is imperative to completely characterize the optimal path for a Dubins vehicle moving on the surface of a sphere.

\begin{remark}
Formally, a rotation minimizing frame, proposed by Bishop, can be used to describe the 3D model \cite{Bishop1975}. By the fundamental theorem of space curves, two scalar functions, curvature $\kappa(s)$ and torsion $\tau(s)$, uniquely determine the curve in terms of arc length $s$~\cite{MathWorldFTSC}. In the Bishop frame context, two ``Bishop curvatures" $(k_1(s), k_2(s))$~\cite{Bishop1975,FaroukiSpringerRMF} are used, and ``roll'' (rotation about $\mathbf{T}$) is set to zero by construction. Using the Bishop frame, the 3D motion planning problem can be posed as
\begin{align*}
    \min_{\mathbf{X}(s)} L &= \int_0^L ds,\\
    &\text{subject to}\\
    \frac{d\mathbf{X}}{ds}(s)=\mathbf{T}(s),& \quad
    \frac{d\mathbf{T}}{ds}=k_1(s)\mathbf{Y}(s) + k_2(s)\mathbf{U}(s),\\
    \frac{d\mathbf{Y}}{ds}(s) = -k_1(s)\mathbf{T}(s), &\quad
    \frac{d\mathbf{U}}{ds}(s) = -k_2(s)\mathbf{T}(s),
\end{align*}
with $|k_1|\le M_1,\ |k_2|\le M_2$. Specifically, curvature $k_1$ governs yaw motion, while $k_2$ governs pitch motion. The motion primitives generated by setting $k_1 = \pm M_1$ and $k_2 = 0$ correspond to the $L_p$ and $R_p$ segments (see Fig.~\ref{fig: optimal_segments}). Setting $k_2 = \pm M_2$ with $k_1 = 0$ yields the $G_{si}$ and $G_{so}$ segments. When both $k_1$ and $k_2$ simultaneously achieve their maximum absolute values, one of $L_{si}$, $R_{si}$, $L_{so}$, or $R_{so}$ is obtained.\footnotemark
\end{remark}

\footnotetext{For a more detailed discussion regarding the model, refer to Appendix~\ref{appsubsect: 3D_model_details}.}

\begin{remark}
Allowing additional ``roll control'' (rotation about $\mathbf{T}$) affects the vehicle's orientation but not the geometric path; the curve's geometry is completely captured by the curvature vector in the $\mathbf{Y}$-$\mathbf{U}$ plane, while roll is an extra attitude degree of freedom. This two-control curvature model is well suited to platforms and planning layers where the path geometry is primary and roll is either unavailable, negligible, or managed by a lower-level controller (as often seen in ground vehicles, underwater vehicles, and fixed-wing aircraft).
\end{remark}

\subsection{Planning on a sphere}

Motion planning for a Dubins vehicle on a sphere entails determining the optimal path to travel from one configuration to another on a sphere.\footnotemark\, Motion planning on a sphere has been reasonably explored in the literature.

\footnotetext{A different notation is used for the frame for the sphere problem compared to the 3D problem since the chosen notation is more intuitive for the sphere problem. In particular, the center of the sphere is chosen to be the origin, and hence, $\mathbf{X}$ denotes the location on the vehicle on the sphere. The choice of $\mathbf{T}$ automatically fixes $\mathbf{N},$ which is defined to be $\mathbf{X} \times \mathbf{T},$ unlike the case for the 3D problem. More details regarding the $\mathbf{X}, \mathbf{T},$ and $\mathbf{N}$ vectors are provided in the following section.}


In \cite{monroy}, path planning on a Riemmanian surface for a curvature-constrained vehicle was considered. In particular, the author showed that the same candidate paths as the Dubins result on a plane were obtained for a sphere, wherein the vehicle's minimum turning radius is $r = \frac{1}{\sqrt{2}}$. The motion planning problem on a sphere was recently studied in \cite{3D_Dubins_sphere}, wherein $r$ was considered to be a parameter for the vehicle. The authors showed that for normal controls\footnotemark, the optimal path is of type $CGC, CCC,$ or a degenerate path of the same for $r \leq \frac{1}{2}$. Here, $C = L, R$ denotes a tight left or right turn, and $G$ denotes a great circular arc. The same model was utilized in \cite{free_terminal_sphere} and \cite{generalization} for the motion planning problem wherein the vehicle must only attain a desired final location. The authors initially showed that the optimal path is of type $CC, CG,$ or a degenerate path of the same for $r \leq \frac{1}{2}$ in \cite{free_terminal_sphere}, and later showed that the optimal path types remain the same for $r \leq \frac{\sqrt{3}}{2}$ in \cite{generalization}. In \cite{time_optimal_synthesis_SO3} and \cite{time_optimal_control_satellite}, the authors considered time-optimal motion planning for a generic model on the Lie group $SO (3),$ which is the group of special orthogonal matrices, considering a single control input and two control input system, respectively. In particular, in \cite{time_optimal_synthesis_SO3}, the authors consider a system given by
\begin{align*}
    \dot{x} = x \left(f + u g \right),
\end{align*}
where $x \in SO(3),$ and $f$ and $g$ are elements of the Lie algebra of $SO (3),$ which is the set of skew-symmetric matrices. In this model, the authors assume that $f$ and $g$ are perpendicular to each other and $\|f \| = \cos{\alpha}, \|g \| = \sin{\alpha}.$ Here, the notion of perpendicularity and norm of skew-symmetric matrices are obtained by identifying each skew-symmetric matrix with a vector in $\mathbb{R}^3$ and utilizing the dot product of vectors in $\mathbb{R}^3$ to define the norm. The authors then show that for $\alpha \in \left(0, \frac{\pi}{4} \right),$ the number of segments in an optimal path is bounded by a function of $\alpha$. The generic model considered by the authors is also applicable to Dubins path planning since the models proposed in \cite{monroy, 3D_Dubins_sphere} can fit into the generic framework. However, fitting the models in \cite{monroy, 3D_Dubins_sphere} in the generic model restricts the radius of turn of the vehicle to be $\frac{1}{\sqrt{2}}$. Furthermore, the angle $\alpha$ reduces to $\frac{\pi}{4},$ for which the result in \cite{time_optimal_synthesis_SO3} corresponding to the bound on the number of segments is not applicable. 
\footnotetext{Abnormal controls and normal controls refer to the two branches of solutions obtained from employing PMP, a first-order necessary condition, to the optimal control problem. In abnormal controls, the adjoint variable corresponding to the objective function in the Hamiltonian is zero, whereas for normal controls, it is non-zero.}

The motion planning problem for a Dubins vehicle on a sphere is also applicable for modeling the motion of a high-speed aircraft moving at a constant altitude over the surface of the Earth. This problem has also been approached using a spherical coordinates approach as opposed to a moving frame approach that has been discussed so far, wherein a bounded lateral force is considered to be the control input in the model using spherical coordinates. However, in \cite{kumar2024equivalencedubinspathsphere}, it was shown that the two models are equivalent, and hence, the moving frame approach can be considered without loss of generality.

From the surveyed papers on motion planning on a sphere to travel from one configuration to another, it can be observed that the optimal path candidates for abnormal controls, i.e., when the adjoint variable corresponding to the integrand in the Hamiltonian is zero, is not known for $r \neq \frac{1}{\sqrt{2}}$. Furthermore, for normal controls, the optimal path candidates for $r \not\in \left(0, \frac{1}{2} \right] \bigcup \{\frac{1}{\sqrt{2}}\}$ is not known. Moreover, since the Dubins vehicle on a sphere corresponds to a single control input system, a phase portrait approach can be employed to obtain the results, particularly in \cite{3D_Dubins_sphere}, through simpler proofs. In this regard, the main contributions of this article are
\begin{itemize}
    \item Obtain the candidate optimal paths for normal controls for $r \in (\frac{1}{2},\frac{\sqrt{3}}{2}].$ In particular, the candidate optimal paths for normal controls are shown to be of type $CGC, CCCC,$ or a degenerate path of the same for $r \leq \frac{1}{\sqrt{2}},$ and of type $CGC, CCCCC,$ or a degenerate path of the same for $r \leq \frac{\sqrt{3}}{2}.$ Moreover, in paths with concatenations of $C$ segments, the angle of the intermediary $C$ segments is greater than $\pi$ radians.
    \item Obtain the candidate optimal paths for abnormal controls. In particular, the candidate optimal paths are shown to be of type $CC$ for $r \leq \frac{1}{\sqrt{2}}$ and $CC_\pi C$ for $r \leq \frac{\sqrt{3}}{2}.$\footnotemark\, The implication of this contribution are as follows: 
    \begin{itemize}
        \item For $r \leq \frac{1}{\sqrt{2}},$ no additional path needs to be considered since $CCC$ is a candidate path for normal controls. \item Noting that no additional path from abnormal controls needs to be considered for $r \leq \frac{1}{\sqrt{2}},$ we hence complete the argument for the optimal path being $CGC, CCC,$ or a degenerate path for $r \leq \frac{1}{2}$, building on \cite{3D_Dubins_sphere}.
        \item For $r \leq \frac{\sqrt{3}}{2},$ a $CC_\pi C$ path needs to be considered. To the best of our knowledge, for motion planning problems for minimizing time traveled, this is the first time abnormal controls yields a unique optimal path candidate. We also provide a numerical example later in the paper to show the existence of configurations for which this path is optimal.
    \end{itemize}
\end{itemize}
\footnotetext{$C_\pi$ denotes a $C$ segment with an arc angle of $\pi$ radians. We also remark here that while a $CC$ path for abnormal controls is subsumed in the list of paths obtained for normal controls, the $CC_\pi C$ path is not; hence, the $CC_\pi C$ path must be additionally considered as a candidate path.}
In addition to the previous two major contributions, another contribution is employing a phase portrait to provide a simpler, alternate approach confirming the results in \cite{3D_Dubins_sphere}. While \cite{monroy} employed a phase portrait approach to obtain the results for $r = \frac{1}{\sqrt{2}},$ we generalize the approach to systematically obtain the optimal path candidates for $r \in \left(0, \frac{\sqrt{3}}{2} \right].$

A summary of the candidate optimal paths from this paper is given in Table~\ref{tab: theorem_candidate_paths}.
\newcommand{\rowstrut}{\rule[-1.5ex]{0pt}{4.5ex}}
\begin{table}[htb!]
    \centering
    \caption{Candidate optimal paths for changing radius on unit sphere up to $\frac{\sqrt{3}}{2}$}
    \label{tab: theorem_candidate_paths}
    \begin{tabular}{cccc}
    \hline
    \multicolumn{1}{|c|}{\textbf{\begin{tabular}[c]{@{}c@{}}Common\\ candidate\\ paths\end{tabular}}} & \multicolumn{1}{c|}{\textbf{\begin{tabular}[c]{@{}c@{}}Additional\\ paths\end{tabular}}} & \multicolumn{1}{c|}{$r$} & \multicolumn{1}{c|}{\textbf{Ref.}} \\ \hline
    \multicolumn{1}{|c|}{\multirow{5}{*}{\begin{tabular}[c]{@{}c@{}}$C,$ $G,$\\ $CG,$ $GC,$ \\ $CGC,$ $CC,$\\ $C C_{\pi + \beta} C^\dagger$\end{tabular}}} & \multicolumn{1}{c|}{$-$} & \multicolumn{1}{c|}{$\rowstrut\leq \frac{1}{2}$} & \multicolumn{1}{c|}{$^*$, \cite{3D_Dubins_sphere}} \\ \cline{2-4} 
    \multicolumn{1}{|c|}{} & \multicolumn{1}{c|}{$C C_{\pi + \beta} C_{\pi + \beta} C^\dagger$} & \multicolumn{1}{c|}{$\rowstrut< \frac{1}{\sqrt{2}}$} & \multicolumn{1}{c|}{$^*$} \\ \cline{2-4} 
    \multicolumn{1}{|c|}{} & \multicolumn{1}{c|}{$-$} & \multicolumn{1}{c|}{$\rowstrut= \frac{1}{\sqrt{2}}$} & \multicolumn{1}{c|}{\cite{monroy}$^\ddagger$} \\ \cline{2-4} 
    \multicolumn{1}{|c|}{} & \multicolumn{1}{c|}{\begin{tabular}[c]{@{}c@{}}$C C_\pi C,$ $C C_{\pi + \beta} C_{\pi + \beta} C^\dagger,$ \\ $C C_{\pi + \beta} C_{\pi + \beta} C_{\pi + \beta} C^\dagger$\end{tabular}} & \multicolumn{1}{c|}{$\in \left(\frac{1}{\sqrt{2}}, \frac{\sqrt{3}}{2} \right]$} & \multicolumn{1}{c|}{$^*$} \\ \hline
    \multicolumn{4}{c}{$^* - $ Current paper, $^\dagger - \beta \in \left(0, \pi \right)$} \\
    \multicolumn{4}{c}{$^\ddagger - $ $CC_{\pi + \beta}C_{\pi + \beta}$ subpath is shown to be non-optimal}
    \end{tabular}
\end{table}

The impact of our contribution is presenting the importance of a parameter of the vehicle (or aircraft), which in this case is $r,$ on the optimal path types. This result has not been observed to the best of our knowledge in motion planning. Such a result not only has a significant impact on motion planning for high-speed aircraft moving over the Earth \cite{kumar2024equivalencedubinspathsphere} but also has a significant impact on the 3D motion planning problem. 
Hence, this result indicates that the solution to the general 3D motion planning problem will depend on the vehicle parameters, an observation that has not been seen in the literature.

\begin{remark}[\textbf{Applications}]
Our models and results have two main applications of particular relevance to aerial navigation. The 3D motion-planning model with pitch and yaw rate constraints directly applies to fixed-wing unmanned aerial vehicles, while the spherical motion-planning model is pertinent for high-speed aerial vehicles flying at approximately constant altitude over the Earth. In both cases, these vehicles are curvature-constrained, making our model an appropriate kinematic abstraction. The proposed approach provides reference trajectories for tracking by lower-level flight controllers. Moreover, the candidate optimal paths serve as building blocks for more complex 3D tasks previously studied mainly in 2D, such as minimum-time convergence to a circular orbit for surveillance missions~\cite{dubins_shortest_circle} and computation of shortest paths avoiding no-fly zones and obstacles~\cite{dubins_gate}.
\end{remark}

\begin{remark}
    While we introduce a new model for the general 3D path planning problem, this paper does not propose a methodology to solve the 3D problem or construct its solutions. The focus of this work is exclusively on the spherical motion planning problem, which serves as an important subproblem in the broader 3D context.
\end{remark}

\emph{Paper Organization:} Section~\ref{sec:mathformulation} presents the problem formulation and control model for Dubins path planning on the sphere. Section~\ref{sect: optimal_paths} discusses the breakdown of the analysis into two cases to be considered to determine the candidate optimal paths. Sections~\ref{sect: Case_1} and \ref{sect: Case_2} consider the two cases, and the optimal paths are derived using geometry and phase portrait analysis. Section~\ref{sect: results} details numerical examples and implementation insights. Section~\ref{sec:conclusion} provides concluding remarks.


\section{Problem Formulation and Hamiltonian Construction} \label{sec:mathformulation}

In this section, the mathematical formulation of the path planning problem, the construction of the Hamiltonian function to employ PMP, and the optimal control actions will be derived. The sequence of steps taken to this end are as follows:
\begin{itemize}
    \item First, the problem formulation for motion planning on the surface of a unit sphere will be presented using the model in \cite{3D_Dubins_sphere}. However, a difficulty faced in this model is to show the non-triviality condition for abnormal controls when the adjoint variable corresponding to the control action is identically zero.
    \item Hence, the considered model is equivalently represented in the framework presented in \cite{monroy}, wherein the optimal control problem is posed on a Lie group. Using this equivalence, PMP on a Lie group is employed similarly to \cite{monroy}.
    \item Applying PMP, it will be shown that the optimal control actions for abnormal controls correspond to $\pm U_{max},$ whereas for normal controls, the optimal control actions will be shown to be $0, \pm U_{max}.$
\end{itemize}
\begin{remark}
The setup of the problem in a Lie group and the corresponding optimal control actions were derived in \cite{kumar2024equivalencedubinspathsphere}. We recap the main results in this section for continuity.
\end{remark}

\begin{figure}[htb!]
    \centering
    \includegraphics[width=0.5\linewidth]{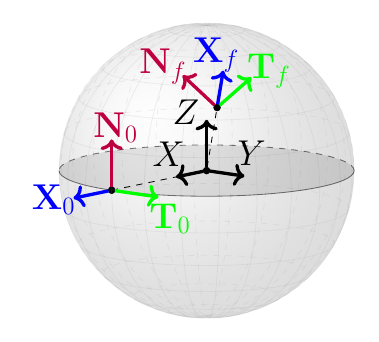}
    \caption{Initial and final configuration on sphere}
    \label{fig: ini_fin_config_sphere}
\end{figure}

The problem of determining the path of shortest length on a unit sphere connecting two configurations for a geodesic-curvature constrained Dubins vehicle can be written as the following variational problem \cite{3D_Dubins_sphere}:
\begin{align} \label{eq: original_obj_func}
    \min \int_0^L ds,
\end{align}
subject to
\begin{align} \label{eq: Sabban_frame_equations}
\begin{split}
    \frac{d \mathbf{X} (s)}{ds} &= \mathbf{T} (s), \quad \frac{d \mathbf{T} (s)}{ds} = -\mathbf{X} (s) + u_g (s) \mathbf{N} (s), \\
    \frac{d \mathbf{N} (s)}{ds} &= -u_g (s) \mathbf{T} (s),
\end{split}
\end{align}
and the boundary conditions
\begin{equation}
    {\mathbf R}(0) = I_3, \quad \quad {\mathbf R}(L) = R_f,
    \label{eq:optconstraints}
\end{equation}
with $u_g \in [-U_{max}, U_{max}],$ where $U_{max}$ is a design parameter\footnotemark\footnotetext{$U_{max}$ is a parameter that controls the rate of change of orientation of the vehicle. It will later be discussed that $U_{max}$ controls the minimum turning radius, which evaluates to be $r = \frac{1}{\sqrt{1 + U_{max}^2}}.$}. Here, $s$ denotes the arc length, $\mathbf{X}$ and $\mathbf{T}$ denote the location and tangent vector of the vehicle, and $\mathbf{N} := \mathbf{X} \times \mathbf{T}.$ These unit three vectors are mutually orthogonal and form a Sabban frame, as depicted in Fig.~\ref{fig: ini_fin_config_sphere}.\footnotemark\footnotetext{A different notation is used for the frame for the sphere problem compared to the 3D problem since the chosen notation is more intuitive for the sphere problem. In particular, the center of the sphere is chosen to be the origin, and hence, $\mathbf{X}$ denotes the location on the vehicle on the sphere. The choice of $\mathbf{T}$ automatically fixes $\mathbf{N},$ which is defined to be $\mathbf{X} \times \mathbf{T},$ unlike the case for the 3D problem. Hence, the orientation of the vehicle is given by $\mathbf{R} = \begin{bmatrix}
    \mathbf{X} & \mathbf{T} & \mathbf{N}
\end{bmatrix}$.}

\begin{remark}[\textbf{Model}]
Before proceeding with the analysis, we provide additional details on the model used. As in the classical 2D Dubins path planning, our model is kinematic; specifically, we neglect the vehicle's inertia. The simplest way to capture motion in 3D is via a rotation-minimizing frame, as introduced by Bishop (see Section~I.B). When restricted to motion on a sphere (by setting $k_1 = u_g,$ $k_2 = -1,$ and $\mathbf{U}_0 = \mathbf{X}_0$), this frame reduces to the Sabban frame model~\cite{3D_Dubins_sphere}, which is equivalent to the model proposed by Monroy~\cite{monroy}. We also note that in~\cite{kumar2024equivalencedubinspathsphere}, a model for a high-speed vehicle moving over the Earth as a particle subject to bounded lateral force was considered, and it was shown to be equivalent to the Sabban frame model.
\end{remark}


Noting the similarity in the model setup to Monroy \cite{monroy}, with $\epsilon$ set to $1$ in the model in \cite{monroy}, the approach of Monroy is adopted in this study, wherein the equations in \eqref{eq: Sabban_frame_equations} can be assembled into the equation
\begin{align} \label{eq: Lie_group_evolution}
    \frac{dg}{ds} = {\overrightarrow l}_1 \left(g (s) \right) - u_g (s) \overrightarrow{L}_{12} \left(g (s) \right),
\end{align}
where $g$ is an element of the Lie group $SO (3)$, where 
\begin{align*}
    g (s) = \begin{pmatrix}
        \textbf{X} (s) & \textbf{T} (s) & \textbf{N} (s)
    \end{pmatrix},
\end{align*}
and $s \rightarrow g (s)$ is a curve on the Lie group. In \eqref{eq: Lie_group_evolution}, $\overrightarrow{l}_1$ and $\overrightarrow{L}_{12}$ are left-invariant vector fields on the Lie group, whose value at the identity of the Lie group is given by
\begin{align*}
    l_1 = \begin{pmatrix}
        0 & -1 & 0 \\
        1 & 0 & 0 \\
        0 & 0 & 0
    \end{pmatrix}, \quad L_{12} = \begin{pmatrix}
        0 & 0 & 0 \\
        0 & 0 & 1 \\
        0 & -1 & 0
    \end{pmatrix}.
\end{align*}
A trajectory on the Lie group consists of a pair $\left(\hat{g} (s), u_g (s) \right),$ where $u_g$ is a measurable bounded function (which, in our case, is the geodesic curvature) with values in $[-U_{max}, U_{max}]$ and $\hat{g} (s)$ satisfies \eqref{eq: Lie_group_evolution}. Each solution $\left(\hat{g} (s), u_g (s) \right)$ projects then on the base manifold $SO (3)$, from which the corresponding trajectory of the vehicle is obtained to be along the curve $\hat{g} (s) e_1$. Here, $e_1$ denotes the unit vector $(1, 0, 0)^T;$ hence, $\hat{g} (s) e_1$ corresponds to the instantaneous location of the vehicle. The curve obtained has curvature $u_g$ and satisfies the boundary conditions in \eqref{eq:optconstraints} \cite{monroy}.

\begin{remark}
The left-invariant vector field $\overrightarrow{l}_1 \left(g (s) \right)$ and $\overrightarrow{L}_{12} \left(g (s) \right)$ can be expressed in terms of its value at the identity as
\begin{align*}
    \overrightarrow{l}_1 \left(g (s) \right) = g (s) l_1 = \begin{pmatrix}
        \textbf{T} (s) & -\textbf{X} (s) & \textbf{0}
    \end{pmatrix}, \\
    \overrightarrow{L}_{12} \left(g (s) \right) = g (s) L_{12} = \begin{pmatrix}
        \textbf{0} & -\textbf{N} (s) & \textbf{T} (s)
    \end{pmatrix}.
\end{align*}
\end{remark}


Similar to Monroy \cite{monroy}, PMP is applied for the symplectic formalism, wherein PMP is applied on the cotangent bundle of the Lie group ($T^* (G)$), which yields \cite{monroy}
\begin{align*}
    H (\zeta) = \zeta_0 + h_1 (\zeta) - u_g H_{12} (\zeta),
\end{align*}
where $\zeta$ is 
the integral curve on $T^* (G)$ \cite{monroy}. 
Here, $h_1$ and $H_{12}$ are functions on $T^* (G)$. In the above formulation, $\zeta_0$ does not depend on $u_g$ and can be taken as an arbitrary parameter. Hence, it can be normalized to $0$ or $-1$ \cite{monroy}. Therefore, the Hamiltonian can be rewritten as
\begin{align} \label{eq: Hamiltonian}
    H_{u_g, \lambda} (\zeta) = -\lambda + h_1 (\zeta) - u_g H_{12} (\zeta),
\end{align}
where $\lambda = \{0, 1\}.$ 



For this model, since the optimal control must pointwise maximize $H$ using PMP, the optimal control action $\kappa$ and the Hamiltonian when $H_{12} \neq 0$ are obtained as
\begin{align}
    \kappa (s) &:= -U_{max} \sgn(H_{12} (\zeta (s))), \\
    H_{\kappa, \lambda} (\zeta (s)) &= -\lambda + h_1 (\zeta (s)) - \kappa (s) H_{12} (\zeta (s)), \label{eq: Hamiltonian_rewritten}
\end{align}
where sgn denotes the signum function, which yields the sign of $H_{12}$ when $H_{12} \neq 0$ and is zero when $H_{12} = 0$.
\footnotetext{It should be noted that $\kappa$ is used to denote the optimal control action, whereas $u_g$ is used to denote any feasible control action. Both $\kappa$ and $u_g$ are measurable bounded functions with values in $[-U_{max}, U_{max}]$.}

Noting that the Hamiltonian expression is the same as given in \cite{monroy}, except for a change in the sign of the optimal curvature $\kappa$, the evolution of the functions $h_1, h_2,$ and $H_{12}$ are given by \cite{monroy}
\begin{align} \label{eq: adjoint_equations}
\begin{split}
    \dot{h}_1 (\zeta (s)) &= -\kappa (s) h_2 (\zeta (s)), \,\, \dot{H}_{12} (\zeta (s)) = -h_2 (\zeta (s)), \\
    \dot{h}_2 (\zeta (s)) &= H_{12} (\zeta (s)) + \kappa (s) h_1 (\zeta (s)).
\end{split}
\end{align}
Here, $h_2$ is the function corresponding to the Hamiltonian vector field $\overrightarrow{l}_2$, whose value at the identity of the Lie group is given by $l_2 = [L_{12}, l_1] = L_{12} l_1 - l_1 L_{12},$
since $\{l_1, l_2, L_{12}\}$ forms a basis for the Lie algebra of the Lie group $G$.\footnotemark 

\footnotetext{It should be noted here that $-H_{12}, h_2, h_1$ play the role of $A, B,$ and $C,$ respectively, in \cite{3D_Dubins_sphere}.}

\begin{figure*}[htb!]
    \centering
    \includegraphics[width = 0.65\linewidth]{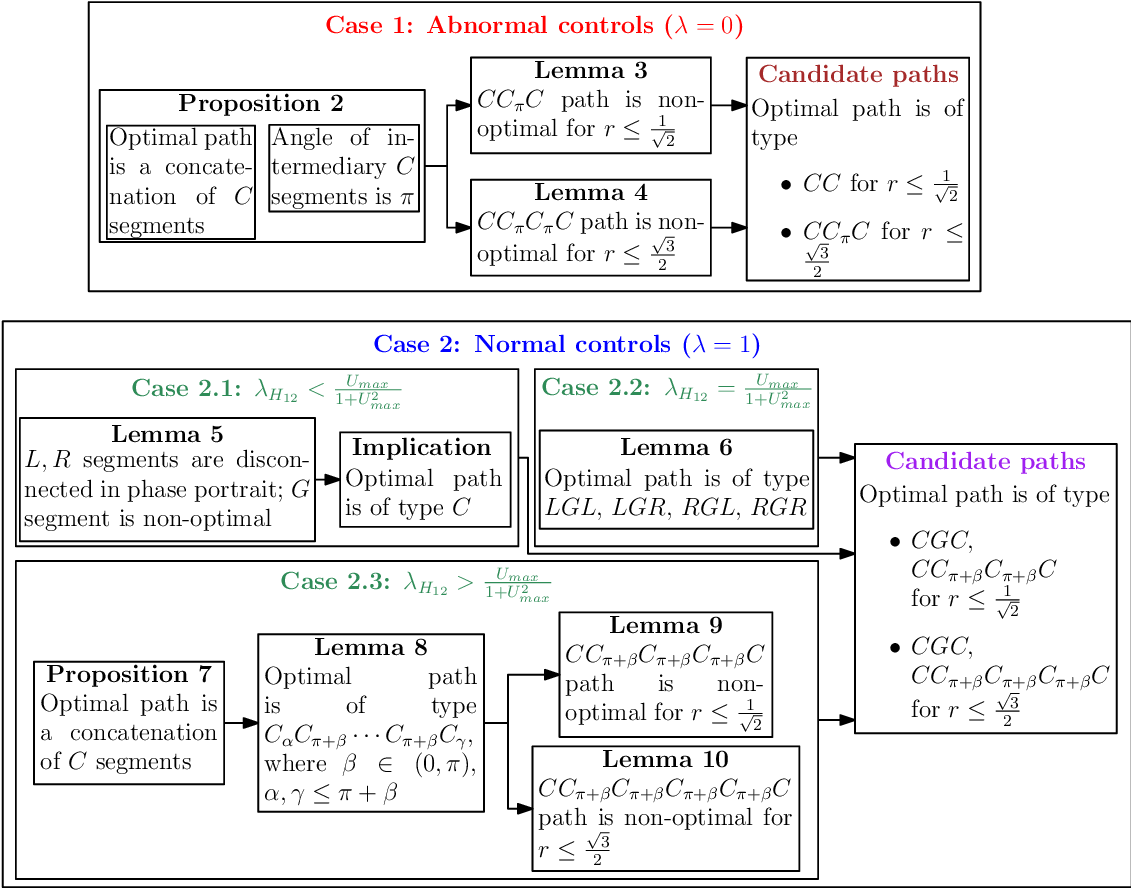}
    \caption{Overview of cases and results. (In this figure, whenever we refer to the optimal path being a particular type, the optimal path can be a degenerate path of the same as well. For example, for $\lambda = 1$ and $r \leq \frac{1}{\sqrt{2}},$ the optimal path can be a degenerate path of $CC_{\pi + \beta} C_{\pi + \beta} C,$ which includes $C, CC,$ $CC_{\pi + \beta}C$.)}
    \label{fig: overview_cases_results}
\end{figure*}

When $H_{12} \equiv 0,$ the following result is obtained (refer to \cite{kumar2024equivalencedubinspathsphere}).
\begin{lemma} \label{lemma: optimal_control_actions}
    If $H_{12} \equiv 0,$ then $\lambda$ cannot be zero; further, for $\lambda = 1,$ $\kappa \equiv 0.$
\end{lemma}

From this lemma, it follows that the optimal control actions are as follows:
\begin{align} \label{eq: optimal_control_inputs}
    \kappa (s) \equiv
    \begin{cases}
        - U_{max}, & H_{12} (\zeta (s)) > 0, \lambda \in \{0, 1 \} \\
        U_{max}, & H_{12} (\zeta (s)) < 0, \lambda \in \{0, 1\} \\
        0, & H_{12} (\zeta (s)) \equiv 0, \lambda = 1.
    \end{cases}
\end{align}
Henceforth, $H_{12} (\zeta (s))$ will be denoted as $H_{12} (s)$ for brevity.

\begin{remark}
The optimal control action $\kappa$ is such that $H \equiv 0$ in \eqref{eq: Hamiltonian_rewritten}, which follows as a consequence of utilizing PMP for the considered problem \cite{monroy}. For this class of time-optimal problems for autonomous systems (i.e., systems without explicit time dependence)~\cite{PMP_lecture_notes}, this condition is observed.
\end{remark}

\section{Characterization of Optimal Path} \label{sect: optimal_paths}

Using the obtained control actions, the optimal path is a concatenation of segments corresponding to $\kappa = -U_{max}, 0,$ and $U_{max}.$ Here, $\kappa (s) \equiv 0$ corresponds to an arc of a great circle, and $\kappa (s) = \pm U_{max}$ correspond to an arc of a small circular arc of radius $r = \frac{1}{\sqrt{1 + U_{max}^2}}$ \cite{3D_Dubins_sphere}. Henceforth, a great circular arc will be denoted as $G$, whereas an arc of a circle of radius $r$ corresponding to $\kappa = U_{max}$ and $\kappa = -U_{max}$ will be denoted as $L$ and $R,$ respectively. The three identified segments are depicted in Fig.~\ref{fig: turns_sphere}.

\begin{figure}[htb!]
    \centering
    \includegraphics[width = 0.5\linewidth]{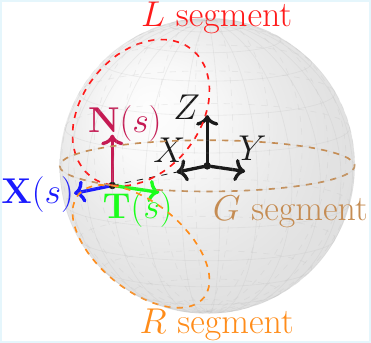}
    \caption{Turns on a sphere \cite{generalization}}
    \label{fig: turns_sphere}
\end{figure}

From the identified control inputs in \eqref{eq: optimal_control_inputs}, the optimal curvature $\kappa$ can be observed to be dependent on the scalar function $H_{12} (s).$ Therefore, the candidate optimal paths can be obtained by constructing the phase portrait of $H_{12} (s).$ To this end, an equation relating $H_{12} (s)$ and $\frac{d H_{12} (s)}{ds}$ is first obtained. 
We can first observe that $J := h_1^2 + h_2^2 + H_{12}^2$ is a constant over a trajectory. This can be shown by differentiating $J$ with respect to $s$ and using the expressions for the derivatives of $h_1, h_2,$ and $H_{12}$ given in \eqref{eq: adjoint_equations} \cite{monroy}. Noting that $\frac{dH_{12} (s)}{ds} = -h_2 (s)$ from \eqref{eq: adjoint_equations}, $\kappa (s) = -U_{max} \sgn{(H_{12} (s))}$ from \eqref{eq: optimal_control_inputs}, and $h_1 (s) = \lambda + \kappa (s) H_{12} (s)$ from the Hamiltonian $H$ given in \eqref{eq: Hamiltonian_rewritten} since $H \equiv 0,$ it follows that (using a similar procedure given in \cite{monroy})
\begin{align*}
    \left(\frac{d H_{12} (s)}{ds} \right)^2 &= J - h_1^2 (s) - H_{12}^2 (s) \\
    &= J - \left(\lambda - U_{max} |H_{12} (s)| \right)^2 - H_{12}^2 (s).
\end{align*}
The above equation can be rewritten as
\begin{align} \label{eq: phase_portrait}
\begin{split}
    f &:= \left(|H_{12} (s)| - \frac{\lambda U_{max}}{1 + U_{max}^2} \right)^2 + \left(\frac{1}{\sqrt{1 + U_{max}^2}} \frac{d H_{12} (s)}{ds} \right)^2 \\
    &= \underbrace{\frac{J}{1 + U_{max}^2} - \frac{\lambda^2}{\left(1 + U_{max}^2 \right)^2}}_{\lambda_{H_{12}}^2}.
\end{split}
\end{align}
Here, since $J$ and $\lambda$ are constant for a trajectory, $\lambda_{H_{12}}$ is a constant.

Having derived the closed-form evolution of $H_{12},$ it is desired to obtain the candidate optimal paths. For this purpose, the analysis to obtain the candidate optimal paths can be broken down into two cases depending on the value of $\lambda:$
\begin{itemize}
    \item Case 1: $\lambda = 0.$ In this case, the solutions obtained are abnormal since they are independent of the integrand (which equals $1$).
    \item Case 2: $\lambda > 0$. In this case, the solutions obtained are normal. Without loss of generality, $\lambda$ can be set to one.
\end{itemize}
Three subcases are solved for the case of $\lambda = 1$; these are
\begin{itemize}
    \item Case 2.1: $\lambda_{H_{12}} < \frac{U_{max}}{1 + U_{max}^2},$
    \item Case 2.2: $\lambda_{H_{12}} = \frac{U_{max}}{1 + U_{max}^2},$ and
    \item Case 2.3: $\lambda_{H_{12}} > \frac{U_{max}}{1 + U_{max}^2}.$
\end{itemize}

The overview of the cases considered and the results obtained from each of the cases is provided in Fig.~\ref{fig: overview_cases_results}. 
In Section~\ref{sect: Case_1}, candidate optimal paths for Case~1 are derived, whereas Section~\ref{sect: Case_2} addresses Case~2.

\begin{remark}
From \eqref{eq: phase_portrait}, $\lambda_{H_{12}}^2 \geq 0$ since it is the sum of squares of real-valued expressions, and is constant as $J$ and $\lambda$ are constant. Thus, $\lambda_{H_{12}}$ is a real constant. We may set $\lambda_{H_{12}} \geq 0$ without loss of generality; the justification for the cases $\lambda=0$ and $\lambda=1$ is provided in the corresponding sections.
\end{remark}

\section{Candidate optimal paths for Case 1: $\lambda = 0$} \label{sect: Case_1}

Before proceeding, the phase portrait of $H_{12}$ for the case of $\lambda =0$ is described.
If $\lambda = 0,$ \eqref{eq: phase_portrait} reduces to
\begin{align} \label{eq: phase_portrait_e_0}
    f := H_{12}^2 (s) + \frac{1}{1 + U_{max}^2} \left(\frac{dH_{12} (s)}{ds} \right)^2 = \lambda_{H_{12}}^2.
\end{align}
The second-order ordinary differential equation in \eqref{eq: phase_portrait_e_0} is equivalent to the characteristic equation for a mass-spring system with force, for which the analytical solution is given by
\begin{align}
    H_{12} (s) &= \lambda_{H_{12}} \sin{\left(\sqrt{1 + U_{max}^2} s - \phi_{H_{12}} \right)}, \label{eq: solution_H12_lambda_0} \\
    \frac{dH_{12} (s)}{ds} &= \lambda_{H_{12}} \sqrt{1 + U_{max}^2} \cos{\left(\sqrt{1 + U_{max}^2} s - \phi_{H_{12}} \right)}, \label{eq: solution_dH12ds_lambda_0}
\end{align}
where $\phi_{H_{12}}$ is the phase angle.


It should be noted that since $H_{12} (s) \not\equiv 0$ from Lemma~\ref{lemma: optimal_control_actions} for $\lambda = 0$, $\lambda_{H_{12}} \neq 0$. Moreover, due to continuity of $H_{12} (s)$ and $\frac{d H_{12} (s)}{ds},$ the 
parameter $\phi_{H_{12}}$ for the two solutions (for $H_{12} < 0$ and $H_{12} > 0$) is equal or a multiple of $2 \pi$ apart.

\begin{remark}
If $\lambda_{H_{12}} < 0$, this can be equivalently reparameterized with $\lambda_{H_{12}} > 0$ by offsetting the phase $\phi_{H_{12}}$ by $\pi$ radians. Hence, $\lambda_{H_{12}} >0$ can be considered without loss of generality.
\end{remark}

Having obtained the closed-form solution for $H_{12} (s)$ and $\frac{dH_{12} (s)}{ds},$ it is desired to characterize the inflection point, i.e., when $H_{12} (s) = 0.$ 
Since $H_{12} (s)$ and $\frac{dH_{12} (s)}{ds}$ are continuous functions of $s$ from \eqref{eq: solution_H12_lambda_0} and \eqref{eq: solution_dH12ds_lambda_0}, when $H_{12} \rightarrow 0,$  $\lim_{H_{12} (s) \rightarrow 0} \frac{dH_{12} (s)}{ds} = \pm \lambda_{H_{12}} \sqrt{1 + U_{max}^2}.$ Since $\lambda_{H_{12}} > 0$, $H_{12} = 0$ occurs only transiently, i.e., $H_{12} \not\equiv 0$ (as can also be observed from the phase portrait in Fig.~\ref{fig: phase_portrait_H12_lambda_0}). Thus, the points $\left(H_{12}(s), \frac{dH_{12}(s)}{ds}\right) = \left(0, \pm \lambda_{H_{12}} \sqrt{1 + U_{max}^2}\right)$ correspond to inflection points between $L$ and $R$ segments.

Using the definition of the function $f$ in \eqref{eq: phase_portrait_e_0} and the previously discussed observations, the phase portrait obtained for $\lambda = 0$ is shown in Fig.~\ref{fig: phase_portrait_H12_lambda_0}.

\begin{figure}[htb!]
    \centering
    \includegraphics[width = 0.7\linewidth]{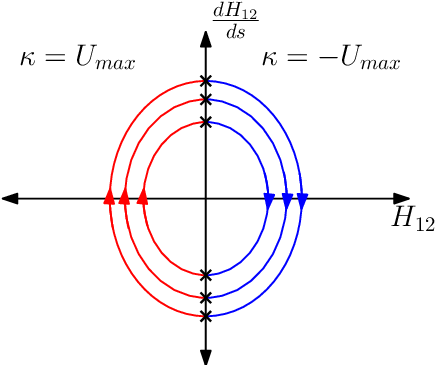}
    \caption{Phase portrait of $H_{12}$ for $\lambda = 0$}
    \label{fig: phase_portrait_H12_lambda_0}
\end{figure}

\begin{remark}
A non-trivial path is defined as a path in which all segments have a non-zero length.
\end{remark}

\begin{remark}
    A $C$ segment is said to be completely traversed if $H_{12} (s_1) = 0, H_{12} (s_2) = 0,$ and $H_{12} (s) > 0$ for  $C = R$ and $H_{12} (s) < 0$ for $C = L.$ Here, $s_1$ and $s_2$ denote the arc length corresponding to the start and end of the $C$ segment, respectively, and $s \in \left(s_1, s_2 \right).$ For example, the phase portrait in Fig.~\ref{fig: phase_portrait_H12_lambda_0} corresponds to a complete $L$ segment traversal if $H_{12}$ starts at a cross-mark with $\frac{dH_{12}}{ds} < 0$, follows the red half-ellipse, and ends at the corresponding cross-mark where $\frac{dH_{12}}{ds} > 0$.
\end{remark}

Given the phase portrait, the following proposition characterizing the optimal path is provided.

\begin{proposition}
    For $\lambda = 0,$ the optimal path is a concatenation of $C$ segments. Further, if a $C$ segment is traversed completely, then the arc angle of such a $C$ segment is $\pi$ radians.
\end{proposition}
\begin{proof}
    From the phase portrait shown in Fig.~\ref{fig: phase_portrait_H12_lambda_0}, it is immediate that the optimal path is a concatenation of $C$ segments. Consider an $L$ segment that is completely traversed, which corresponds to $\kappa = U_{max}$. Letting $s_1$ denote the arc length of the $L$ segment, the condition for $H_{12}$ and $\frac{d H_{12}}{ds}$ before and after traversing the $L$ segment are given by $\left(H_{12} (0), \frac{dH_{12} (0)}{ds} \right) = (0, -\lambda_{H_{12}} \sqrt{1 + U_{max}^2})$ and $\left(H_{12} (s_1), \frac{dH_{12} (s_1)}{ds} \right) = (0, \lambda_{H_{12}} \sqrt{1 + U_{max}^2}),$ respectively (as observed from Fig.~\ref{fig: phase_portrait_H12_lambda_0}). Using the expressions for $H_{12} (s)$ and $\frac{d H_{12} (s)}{ds}$ given in \eqref{eq: solution_H12_lambda_0} and \eqref{eq: solution_dH12ds_lambda_0}, it follows that
    \begin{align*}
        \sin{(-\phi_{H_{12}})} &= \sin{\left(\sqrt{1 + U_{max}^2} s_1 - \phi_{H_{12}} \right)} = 0, \\
        \cos{\left(- \phi_{H_{12}} \right)} &= -\cos{\left(\sqrt{1 + U_{max}^2} s_1 - \phi_{H_{12}} \right)} = -1.
    \end{align*}

    It should be recalled from Section~\ref{sect: optimal_paths} that if $\kappa = \pm U_{max},$ the segment corresponds to a circular arc of radius $r = \frac{1}{\sqrt{1 + U_{max}^2}}$. Hence, the arc length $s$ is related to the angle of the $L$ segment $\phi_L$ by $s_1 = \phi_L r = \phi_L \frac{1}{\sqrt{1 + U_{max}^2}}.$ Since $\cos{\left(-\phi_{H_{12}} \right)} = -1,$ $\cos{\left(\phi_L - \phi_{H_{12}} \right)} = 1,$ it follows that $\phi_L$ is an odd multiple of $\pi.$ Therefore, $\phi_L = \pi$ since $\phi_L \in [0, 2 \pi)$ radians.
\end{proof}


From the above proposition, it follows that the optimal path is of type $C_{\alpha} C_\pi C_\pi \cdots C_{\gamma},$ where $0 \leq \alpha, \gamma \leq \pi$. However, the number of concatenations is unknown. Hence, we claim that
\begin{itemize}
    \item A $C C_\pi C$ path is non-optimal for $r \leq \frac{1}{\sqrt{2}},$ which is proved in Lemma~\ref{lemma: non-optimality_CCpiC}. 
    \item A $C C_\pi C_\pi C$ path is non-optimal for $r \leq \frac{\sqrt{3}}{2},$ which is proved in Lemma~\ref{lemma: nonoptimality_CCpiCpiC}.
\end{itemize}

We now state and prove non-optimality of a $C C_\pi C$ path for $r \leq \frac{1}{\sqrt{2}}$ in the following lemma.

\begin{lemma} \label{lemma: non-optimality_CCpiC}
A non-trivial $CC_\pi C$ path is not optimal for $r \leq \frac{1}{\sqrt{2}}$.
\end{lemma}
\begin{proof}
    Consider a non-trivial optimal $L_\alpha R_\pi L_\gamma$ path, wherein $\alpha, \gamma > 0.$ Consider an angle $\delta$ such that $0 < \delta < \min{\left(\alpha, \gamma \right)},$ and $\delta$ is close to zero. A non-trivial $L_\alpha R_\pi L_\gamma$ path can be shown to be non-optimal if there exists an alternate path of a lower path length than the $L_\delta R_\pi L_\delta$ subpath of the $L_\alpha R_\pi L_\gamma$ path. For this purpose, it is claimed that there exists a $G_{\phi_1 (\delta)} R_{\pi + \phi_2 (\delta)} G_{\phi_1(\delta)}$ path such that $\phi_1(\delta) \geq 0, \phi_2 (\delta) \leq 0$ for $r \leq 1/\sqrt{2}, 0<\delta <<1$, which is of a lower path length than the $L_\delta R_\pi L_\delta$ path. 
    Hence, the alternate path must satisfy
    \begin{align} \label{eq: LRL_path_matrix_equation}
    \begin{split}
        &\mathbf{R}_L (r, \delta) \mathbf{R}_R (r, \pi) \mathbf{R}_L (r, \delta) \\
        &= \mathbf{R}_G (\phi_1 (\delta)) \mathbf{R}_R (r, \pi + \phi_2 (\delta)) \mathbf{R}_G (\phi_1 (\delta)),
    \end{split}
    \end{align}    
    where $\mathbf{R}_L, \mathbf{R}_R, \mathbf{R}_G$ denote the rotation matrices corresponding to the $L, R,$ and $G$ segments, respectively. The above equation ensures that the initial and alternate paths connect the same initial and final configurations.

    A depiction of the initial $L_\delta R_\pi L_\delta$ and the proposed alternate $GRG$ path is demonstrated for $r = 0.5$ and $\delta = 30^\circ$ in Fig.~\ref{fig: GRG_path_shortcut_LRL}.\footnotemark\footnotetext{The closed-form expressions for the arc angles of the $GRG$ path were computed analytically by solving \eqref{eq: LRL_path_matrix_equation}. A discussion regarding path construction is provided in Section~\ref{sect: results}.}\, From this figure, it can be immediately observed that the alternate $GRG$ path is shorter than the $L_\delta R_\pi L_\delta$ path due to spherical convexity.\footnotemark
    \footnotetext{An analytical proof is provided in Lemma~\ref{lemma: non-optimality_CCpiC} as opposed to utilizing the spherical convexity argument to account for the non-optimality argument holding for up to $r = \frac{1}{\sqrt{2}}.$ The dependence on $r$ could not be rigorously obtained through the geometric approach.}
    
    \begin{figure}[htb!]
        \centering
        \includegraphics[width = 0.8\linewidth]{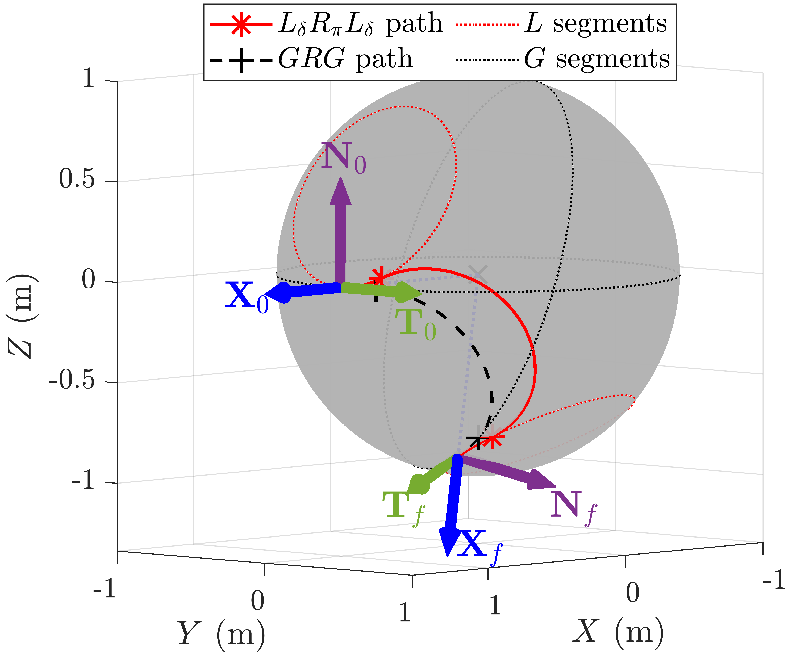}
        \caption{A $GRG$ path connecting the same configurations as the $L_\delta R_\pi L_\delta$ path for $r = 0.5,$ $\delta = 30^\circ$. Animations of a vehicle moving along the initial and alternate paths are available on our GitHub page.}
        \label{fig: GRG_path_shortcut_LRL}
    \end{figure}
    
    It should be noted here that when $\delta = 0$, it is desired for $\phi_1$ and $\phi_2$ to be zero for \eqref{eq: LRL_path_matrix_equation} to be satisfied. In the vicinity of $\delta = 0,$ $\phi_1$ and $\phi_2$ can be approximated using Taylor's series expansion as $\phi_1 (\delta) = a_1 \delta + \frac{1}{2} b_1 \delta^2,$ $\phi_2 (\delta) = a_2 \delta + \frac{1}{2} b_2 \delta^2.$
    
    The $LRL$ path considered will not be optimal if the difference in the path length between the $LRL$ path and the alternate $GRG$ path is greater than zero (refer to Fig.~\ref{fig: GRG_path_shortcut_LRL}). The path difference is given by
    \begin{align} \label{eq: path_length_difference}
        \Delta (\delta) &= 2 r \delta - 2 \phi_1 (\delta) - r \phi_2 (\delta).
    \end{align}
    Noting that $\Delta (\delta = 0) = 0,$ it suffices to show that $\frac{d \Delta (\delta)}{d \delta} |_{\delta = 0} > 0.$ The expression for $\frac{d \Delta (\delta)}{d \delta} |_{\delta = 0}$ is given by
    \begin{align} \label{eq: derivative_cost_difference_GRG_path}
        \frac{d \Delta (\delta)}{d \delta} \bigg{|}_{\delta = 0} = 2 r - 2 a_1 - r a_2.
    \end{align}
    To prove the above claim, $a_1$ and $a_2$ are to be calculated. Hence, it is desired to differentiate \eqref{eq: LRL_path_matrix_equation}, for which the expressions for the derivative of $\mathbf{R}_L$ and $\mathbf{R}_R$ are desired to be obtained.
    
    Noting that the Sabban frame equations in \eqref{eq: Sabban_frame_equations} can be assembled as 
    \begin{align*}
        \mathbf{R}' (s) = \mathbf{R} (s) \underbrace{\begin{pmatrix}
            0 & -1 & 0 \\
            1 & 0 & -u_g \\
            0 & u_g & 0
        \end{pmatrix}}_{\Omega}, 
    \end{align*}
    the solution to the above equation when $u_g$ is constant is given by $\mathbf{R} (s) = \mathbf{R} (s_i) \left(e^{\Omega \Delta s} \right)$ \cite{free_terminal_sphere}. Hence, corresponding to $u_g = 0, U_{max},$ and $-U_{max},$ $\mathbf{R}_S (s) := e^{\Omega_S \Delta s}$ is the rotation matrix of segment $S = G, L, R,$ respectively. Noting that the arc length of the segment $\Delta s$ is related to the arc angle of the segment $\phi$ by $\Delta s = \phi$ for a $G$ segment and $\Delta s = r \phi$ for $L$ and $R$ segments, the rotation matrix expression can be alternately written as $\mathbf{R}_S (\phi) := e^{\hat{\Omega}_S \phi},$ where $\hat{\Omega}_G = \Omega_G, \hat{\Omega}_L = r \Omega_L,$ and $\hat{\Omega}_R = r \Omega_R.$ Hence, it follows that $\frac{d \mathbf{R}_S}{d \phi} = \mathbf{R}_S \hat{\Omega}_S.$
    
    Therefore, differentiating \eqref{eq: LRL_path_matrix_equation}, using the expression for the derivative of the rotation matrices, and evaluating the equation for $\delta = 0,$ the equation can be simplified as\footnotemark\footnotetext{The reader can refer to Appendix~\ref{appsubsect: calculations} for the intermediary steps to derive \eqref{eq: first_order_equation_GRG_path} from \eqref{eq: LRL_path_matrix_equation}.}
    \begin{align} \label{eq: first_order_equation_GRG_path}
    \begin{split}
        &\hat{\Omega}_L \mathbf{R}_R (r, \pi) + \mathbf{R}_R (r, \pi) \hat{\Omega}_L \\
        &= a_1 \left(\hat{\Omega}_G \mathbf{R}_R (r, \pi) + \mathbf{R}_R (r, \pi) \hat{\Omega}_G \right) + a_2 \mathbf{R}_R (r, \pi) \hat{\Omega}_R.
    \end{split}
    \end{align}
    Using the closed form expression for $\mathbf{R}_R$ (derived in \cite{free_terminal_sphere}, and summarized in Appendix~\ref{appsubsect: rotation_matrices}) and the skew-symmetric matrices, the above equation can be expanded as
    \begin{align*}
        &\left(2 r^2 - 1 \right) \begin{pmatrix}
            0 & 2 r & 0 \\
            -2 r & 0 & -2 \sqrt{1 - r^2} \\
            0 & 2 \sqrt{1 - r^2} & 0
        \end{pmatrix} \\
        &= (2 r a_1 + a_2) \begin{pmatrix}
            0 & r & 0 \\
            -r & 0 & - \sqrt{1 - r^2} \\
            0 & \sqrt{1 - r^2} & 0
        \end{pmatrix}.
    \end{align*}
    The matrices in the above equation are non-zero since $0 < r \leq \frac{1}{\sqrt{2}}$. Hence, from the above equation, it follows that 
    \begin{align} \label{eq: constraint_a1_a2_GRG_path}
        2 r a_1 + a_2 = 2 (2 r^2 - 1).
    \end{align}
    Since $a_1$ and $a_2$ cannot be explicitly obtained from the first-order approximation comparison, \eqref{eq: LRL_path_matrix_equation} is differentiated twice with respect to $\delta$ and 
    evaluated at $\delta = 0$. Expanding the equation using the expression for $\mathbf{R}_R$ (given in Appendix~\ref{appsubsect: rotation_matrices}) and the skew-symmetric matrices, the equation obtained is given by
    \begin{align} \label{eq: LRL_path_matrix_equation_second_derivative_expanded}
    \begin{split}
        &\begin{pmatrix}
            -4 r^2 + 8 r^4 & 0 & 0 \\
            0 & 4 \left(1 - 2 r^2 \right)^2 & 0 \\
            0 & 0 & 4 (1 - 2 r^2) (1 - r^2)
        \end{pmatrix} \\
        &= \begin{pmatrix}
            \zeta_{11} & \zeta_{12} & \zeta_{13} \\
            -\zeta_{12} & \zeta_{22} & \zeta_{23} \\
            \zeta_{13} & -\zeta_{23} & \zeta_{33}
        \end{pmatrix},
    \end{split}
    \end{align}
    where
    \begin{align*}
        \zeta_{11} &= 4 r^2 a_1^2 + 4 r a_1 a_2 + r^2 a_2^2, \,\, \zeta_{12} = (2 r b_1 + b_2) r, \\
        \zeta_{13} &= \sqrt{1 - r^2} \left(2 r a_1^2 + 2 a_1 a_2 + r a_2^2 \right), \,\, \zeta_{22} = \left(2 r a_1 + a_2 \right)^2, \\
        \zeta_{23} &= -(2 r b_1 + b_2) \sqrt{1 - r^2}, \,\, \zeta_{33} = \left(1 - r^2 \right) a_2^2.
    \end{align*}
    From the above equation, it follows that $2 r b_1 + b_2 = 0$ since $\zeta_{12}$ and $\zeta_{23}$ are zero, and $a_2^2 = 4 (1 - 2 r^2)$ (comparing the terms in the third row and third column). From the two possible solutions, consider $a_2 = -2 \sqrt{1 - 2 r^2},$ which is less than or equal to zero for $r \leq \frac{1}{\sqrt{2}}$. Using \eqref{eq: constraint_a1_a2_GRG_path}, it follows that
    \begin{align*}
        a_1 = \frac{\sqrt{1 - 2 r^2} \left(1 - \sqrt{1 - 2 r^2} \right)}{r} \geq 0,
    \end{align*}
    for $0 < r \leq \frac{1}{\sqrt{2}}.$ 

    \begin{remark}
        Using the obtained expressions for $a_1$ and $a_2,$ the following equations can be easily verified:
        \begin{align*}
            \zeta_{11} &= -4 r^2 + 8 r ^4, \quad
            \zeta_{13} = 0, \quad
            \zeta_{22} = 4 \left(1 - 2 r^2 \right)^2.
        \end{align*}
        Hence, \eqref{eq: LRL_path_matrix_equation_second_derivative_expanded} is satisfied for the chosen solution for $a_1$ and $a_2,$ with $b_1$ and $b_2$ satisfying $2 r b_1 + b_2 = 0.$
    \end{remark}
    \begin{remark}
        The solution chosen for $a_2$ ensures that $a_1$ is non-negative for $r \leq \frac{1}{\sqrt{2}}$, and in particular, positive for $r < \frac{1}{\sqrt{2}}$. As $a_1$ equals $\frac{d\phi_1 (\delta)}{d \delta}|_{\delta = 0}$, and $\phi_1 (\delta)|_{\delta = 0} = 0,$ it is desired for $\phi_1 (\delta) \geq 0$ for small non-negative perturbation in $\delta$. 
        This is necessary for the feasibility of the alternate $GRG$ path, as $\phi_1$, representing the arc angle of the $G$ segment, must be non-negative. Choosing the other solution for $a_2$ would result in $a_1 \leq 0$, and thus $\phi_1$ would be negative near $\delta = 0$ for $r < \frac{1}{\sqrt{2}}$, rendering the path infeasible.
    \end{remark}

    Substituting the obtained expressions for $a_1$ and $a_2$ in \eqref{eq: derivative_cost_difference_GRG_path}, $\frac{d \Delta (\delta)}{d \delta}|_{\delta = 0}$ is simplified as
    \begin{align*}
        \frac{d \Delta (\delta)}{d \delta}|_{\delta = 0} = \frac{2}{r} \left(1 - \sqrt{1 - 2 r^2} \right) \left(1 - r^2 \right) > 0,
    \end{align*}
    for $0 < r \leq \frac{1}{\sqrt{2}}$. Therefore, the constructed alternate $GRG$ path is shorter than the considered $L_\delta R_\pi L_\delta$ path. Hence, the $L_\delta R_\pi L_\delta$ path is non-optimal. Using a similar proof for a $R_\delta L_\pi R_\delta$ path, it can be concluded that a non-trivial $CC_\pi C$ path is not an optimal path for $r \leq \frac{1}{\sqrt{2}}$.
\end{proof}

Using the previous result, it follows that the optimal path contains at most two $C$ segments for $r \leq \frac{1}{\sqrt{2}}$ for $\lambda = 0.$ However, the maximum number of concatenations in an optimal path for $r > \frac{1}{\sqrt{2}}$ is not known. To address this, we claim that the maximum number of concatenations for $r \leq \frac{\sqrt{3}}{2}$ is three through the following lemma, which is the second claimed result for abnormal controls (refer to Fig.~\ref{fig: overview_cases_results}). 
We proceed to formally state and prove the non-optimality of a $CC_\pi C_\pi C$ path for $r \leq \frac{\sqrt{3}}{2}$ in the following lemma.

\begin{lemma} \label{lemma: nonoptimality_CCpiCpiC}
    A non-trivial $CC_\pi C_\pi C$ path is not optimal for $r \leq \frac{\sqrt{3}}{2}$.
\end{lemma}
\begin{proof}
    Consider a non-trivial $L_\alpha R_\pi L_\pi R_\gamma$ path, where $\alpha, \gamma > 0.$ Consider an angle $\delta$ such that $0 < \delta < \min(\alpha, \gamma),$ and $\delta$ is close to zero. It is sufficient to show that this path is non-optimal for $r \in \left(\frac{1}{\sqrt{2}}, \frac{\sqrt{3}}{2} \right]$, since for $r \leq \frac{1}{\sqrt{2}}$, the considered path is non-optimal using Lemma~\ref{lemma: non-optimality_CCpiC}. It is claimed that there exists an alternate $R_{\pi + \phi_1 (\delta)} G_{\phi_2 (\delta)} L_{\pi + \phi_1 (\delta)}$ path such that $\phi_1 (\delta) \leq 0, \phi_2 (\delta) \geq 0$ for $0 < \delta << 1$ that can connect the same configurations as the $L_\delta R_\pi L_\pi R_\delta$ path at a lower path length for $r \in \left(\frac{1}{\sqrt{2}}, \frac{\sqrt{3}}{2} \right]$. Hence, the alternate path must satisfy
    \begin{align} \label{eq: LRLR_path_matrix_equation}
    \begin{split}
        &\mathbf{R}_L (r, \delta) \mathbf{R}_R (r, \pi) \mathbf{R}_L (r, \pi) \mathbf{R}_R (r, \delta) \\
        &= \mathbf{R}_R (r, \pi + \phi_1 (\delta)) \mathbf{R}_G (\phi_2 (\delta)) \mathbf{R}_L (r, \pi + \phi_1 (\delta)).
    \end{split}
    \end{align}

    A depiction of the initial $L_\delta R_\pi L_\pi R_\delta$ path and an alternate $RGL$ path is illustrated for $r = 0.865$ and $\delta = 20^\circ$ in Fig.~\ref{fig: RGL_path_shortcut_LRLR}.\footnotemark\footnotetext{In this figure, closed-form expressions for the arc angles for the $RGL$ path were obtained by solving \eqref{eq: LRLR_path_matrix_equation}. A discussion regarding path construction is provided in Section~\ref{sect: results}.}
    \begin{figure}[htb!]
        \centering
        \includegraphics[width = 0.8\linewidth]{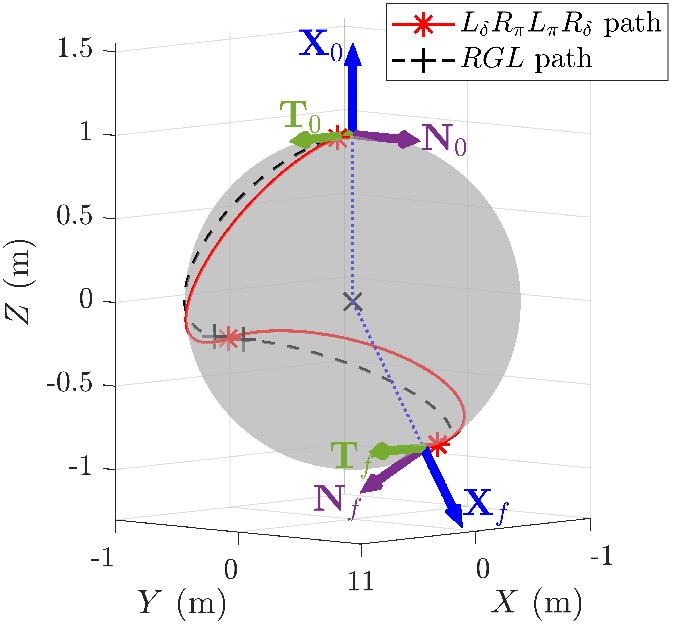}    \caption{An $RGL$ path connecting the same configurations as the $L_\delta R_\pi L_\pi R_\delta$ path for $r = 0.865,$ $\delta = 20^\circ$. Animations of a vehicle moving along the initial and alternate paths are available on our GitHub page.}
        \label{fig: RGL_path_shortcut_LRLR}
    \end{figure}
    
    It should be noted here that when $\delta = 0$, it is desired for $\phi_1$ and $\phi_2$ to be zero for \eqref{eq: LRLR_path_matrix_equation} to be satisfied. In the vicinity of $\delta = 0,$ $\phi_1$ and $\phi_2$ can be approximated using Taylor's series expansion as $\phi_1 (\delta) = a_1 \delta + \frac{1}{2} b_1 \delta^2,$ $\phi_2 (\delta) = a_2 \delta + \frac{1}{2} b_2 \delta^2.$
    
    The considered $LRLR$ path will not be optimal if the difference in the path length between the $LRLR$ path and the alternate $RGL$ path is greater than zero (refer to Fig.~\ref{fig: RGL_path_shortcut_LRLR}). The difference in the path length is given by 
    \begin{align} \label{eq: cost_difference_RGL_path}
        \Delta (\delta) &= 
        2 r \delta - 2 r \phi_1 (\delta) - \phi_2 (\delta).
    \end{align}
    Noting that $\Delta (\delta = 0) = 0,$ it suffices to show that $\frac{d \Delta (\delta)}{d \delta} |_{\delta = 0} > 0.$ The expression for $\frac{d \Delta (\delta)}{d \delta} |_{\delta = 0}$ is given by
    \begin{align} \label{eq: derivative_cost_difference_RGL_path}
        \frac{d \Delta (\delta)}{d \delta} \bigg{|}_{\delta = 0}
        = 2 r - 2 r a_1 - a_2.
    \end{align}    
    
    To prove the above claim, $a_1$ and $a_2$ are to be calculated. For this purpose, \eqref{eq: LRLR_path_matrix_equation} is differentiated and evaluated for $\delta = 0.$ Furthermore, the expression for the derivative of the rotation matrices derived in the proof of Lemma~\ref{lemma: non-optimality_CCpiC}, given by $\frac{d \mathbf{R}_S}{d \delta} = \mathbf{R}_S \hat{\Omega}_S,$ where $S \in \{L, R, G\},$ is used.
    Using the closed-form expression for the rotation matrices (given in Appendix~\ref{appsubsect: rotation_matrices}) and the skew-symmetric matrices corresponding to the segments (introduced in the proof of Lemma~\ref{lemma: non-optimality_CCpiC}), 
    the equation obtained is given by
    \begin{align*}
        &2 r \left(4 r^2 - 3 \right) \begin{pmatrix}
            0 & 1 - 2 r^2 & 0 \\
            -\left(1 - 2 r^2 \right) & 0 & -2 r \sqrt{1 - r^2}\\
            0 & -2 r \sqrt{1 - r^2} & 0
        \end{pmatrix} \\
        &= \left(2 r a_1 + a_2 \right) \begin{pmatrix}
            0 & \left(1 - 2 r^2 \right) & 0 \\
            \left(2 r^2 - 1 \right) & 0 & - 2 r \sqrt{1 - r^2} \\
            0 & - 2 r \sqrt{1 - r^2} & 0
        \end{pmatrix}.
    \end{align*}
    In the above equation, the matrices are non-zero since $r > \frac{1}{\sqrt{2}}$, which implies $1 - 2 r^2 > 0.$ Hence, from the above equation, it follows that
    \begin{align} \label{eq: relation_a1_a2_CCpiCpiC_nonoptimality}
        2 r a_1 + a_2 = 2 r \left(4 r^2 - 3 \right).
    \end{align}
    Substituting the above equation in \eqref{eq: derivative_cost_difference_RGL_path}, $\frac{d \Delta (\delta)}{d \delta} \big{|}_{\delta = 0}$ is obtained as
    \begin{align*}
        \frac{d \Delta (\delta)}{d \delta}|_{\delta = 0} = 2 r - 2 r \left(4 r^2 - 3 \right) = 8 r \left(1 - r^2 \right),
    \end{align*}
    which is greater than zero for all $0 < r < 1.$ However, this does not imply that the initial $L_\delta R_\pi L_\pi R_\delta$ can be replaced by an $RGL$ path for all $r$, since existence of a real solution for $a_1$ and $a_2$ should be guaranteed. Hence, \eqref{eq: LRLR_path_matrix_equation} is differentiated twice, evaluated at $\delta = 0,$ and simplified using the expressions for the rotation matrices (given in Appendix~\ref{appsubsect: rotation_matrices}) and skew-symmetric matrices (given in the proof of Lemma~\ref{lemma: non-optimality_CCpiC}) to obtain
    \begin{align} \label{eq: second_order_equation_RGL_path}
    \begin{split}
        &\begin{pmatrix}
            \epsilon_{11} & 0 & \epsilon_{13} \\
            0 & -4 r^2 \left(3 - 4 r^2 \right)^2 & 0 \\
            -\epsilon_{13} & 0 & \epsilon_{33}
        \end{pmatrix} \\
        &= \left(2 r b_1 + b_2 \right) \begin{pmatrix}
            0 & \left(1 - 2 r^2 \right) & 0 \\
            \left(2 r^2 - 1 \right) & 0 & - 2 r \sqrt{1 - r^2} \\
            0 & - 2 r \sqrt{1 - r^2} & 0
        \end{pmatrix} \\
        & \quad\, + \begin{pmatrix}
            \gamma_{11} & 0 & \gamma_{13} \\
            0 & - \left(2 r a_1 + a_2 \right)^2 & 0 \\
            -\gamma_{13} & 0 & \gamma_{33}
        \end{pmatrix},
    \end{split}
    \end{align}
    where 
    \begin{align*}
        \epsilon_{11} &= -4r^2 \left(8 r^4 - 10 r^2 + 3 \right), \,\, \epsilon_{13} = 2 r \left(4 r^2 - 3 \right) \sqrt{1 - r^2}, \\
        \epsilon_{33} &= 8 r^2 \left(4 r^2 - 3 \right) \left(r^2 - 1 \right), \\
        \gamma_{11} &= \left(1 - 2 r^2 \right) \left(4 r \left(r a_1^2 + a_1 a_2 \right) - \left(1 - 2 r^2 \right) a_2^2 \right), \\
        \gamma_{13} &= 2 \sqrt{1 - r^2} \left(\left(4 r^2 - 1 \right) \left(r a_1^2 + a_1 a_2 \right) + r \left(2 r^2 - 1 \right) a_2^2 \right), \\
        \gamma_{33} &= 4 r \left(1 - r^2 \right) \left(2 r a_1^2 + 2 a_1 a_2 + r a_2^2 \right).
    \end{align*}
    From \eqref{eq: second_order_equation_RGL_path}, it follows that $2 r b_1 + b_2 = 0.$ Further, noting that $\gamma_{33}$ must equal $\epsilon_{33},$ it follows that
    \begin{align*}
        2 r a_1^2 + 2 a_1 a_2 + r a_2^2 = 2 r \left(3 - 4 r^2 \right).
    \end{align*}
    Substituting for $a_1$ from \eqref{eq: relation_a1_a2_CCpiCpiC_nonoptimality} in the above equation and simplifying, the equation for $a_2$ is obtained as
    \begin{align*}
        \frac{2 r^2 - 1}{2 r} a_2^2 = 4 r \left(3 - 4 r^2 \right) \left(2 r^2 -1 \right).
    \end{align*}
    Since $r > \frac{1}{\sqrt{2}}$ is considered, $2 r^2 - 1 \neq 0.$ Therefore, $a_2^2 = 8 r^2 \left(3 - 4 r^2 \right).$ For $\frac{1}{\sqrt{2}} < r \leq \frac{\sqrt{3}}{2},$ a solution for $a_2$ can be obtained. The positive solution is chosen, since it is desired for $\frac{d \phi_2 (\delta)}{d \delta}\big{|}_{\delta = 0} = a_2 \geq 0$ to be positive. This is because $\phi_2 (\delta)|_{\delta = 0} = 0,$ and for a path to be feasible, all angles must be non-negative. Therefore, the solution for $a_2$, and the corresponding solution for $a_1$ obtained using \eqref{eq: relation_a1_a2_CCpiCpiC_nonoptimality}, are chosen to be $a_2 = 2 \sqrt{2} r \sqrt{3 - 4 r^2}, a_1 = 4 r^2 - 3 - \sqrt{2} \sqrt{3 - 4 r^2}.$
    It follows that $a_1 \leq 0$ from the obtained expression for $\frac{1}{\sqrt{2}} < r \leq \frac{\sqrt{3}}{2}$. 
    \begin{remark}
        Using the obtained expressions for $a_1$ and $a_2,$ the following equations can be easily verified:
        \begin{align*}
            \gamma_{11} &= \epsilon_{11}, \,\,
            \gamma_{13} = \epsilon_{13}, \,\,
            4 r^2 \left(3 - 4 r^2 \right)^2 = \left(2 r a_1 + a_2 \right)^2.
        \end{align*}
        Hence, \eqref{eq: second_order_equation_RGL_path} is satisfied for the chosen solution for $a_1$ and $a_2,$ with $b_1$ and $b_2$ satisfying $2 r b_1 + b_2 = 0.$
    \end{remark}
    
    Since $\frac{d \Delta (\delta)}{d \delta}|_{\delta = 0}$ was shown to be greater than zero for all $r$, and real solutions for $a_1$ and $a_2$ were obtained for $\frac{1}{\sqrt{2}} < r \leq \frac{\sqrt{3}}{2},$ it follows that the constructed alternate $RGL$ path connects the same initial and final configurations as the $L_\delta R_\pi L_\pi R_\delta$ path at a lower path length. Hence, the $L_\delta R_\pi L_\pi R_\delta$ path is non-optimal for $r \leq \frac{\sqrt{3}}{2}$ (using Lemma~\ref{lemma: non-optimality_CCpiC} for $r \leq \frac{1}{\sqrt{2}}$). Using a similar proof for the $R_\delta L_\pi R_\pi L_\delta$ path, it follows that a $C C_\pi C_\pi C$ path is non-optimal for $r \leq \frac{\sqrt{3}}{2}$.
\end{proof}

\section{Candidate optimal paths for Case 2: $\lambda = 1$} \label{sect: Case_2}

Having characterized the optimal path for $\lambda = 0$ (abnormal controls) in the previous section, we now consider the case $\lambda = 1$, corresponding to normal controls (see Fig.~\ref{fig: overview_cases_results}). This is the second case needed to complete the characterization of the optimal path. Recall that for $\lambda = 1$, we identified three subcases depending on the value of $\lambda_{H_{12}}$ (refer to Section~\ref{sect: optimal_paths}).

For $\lambda = 1$, the evolution of $H_{12} (s)$ given in \eqref{eq: phase_portrait} reduces to
\begin{align} \label{eq: ellipse_equations} 
\begin{split}
    f &:= \left(|H_{12} (s)| - \frac{U_{max}}{1 + U_{max}^2} \right)^2 + \left(\frac{1}{\sqrt{1 + U_{max}^2}} \frac{d H_{12} (s)}{ds} \right)^2 \\
    &= \lambda_{H_{12}}^2.
\end{split}
\end{align}

The solutions for $H_{12} (s)$ and $\frac{dH_{12} (s)}{ds}$ 
can be obtained using the expression of $\kappa$ from \eqref{eq: optimal_control_inputs} as
\begin{align}
    H_{12} (s) &= \begin{cases}
        \lambda_{H_{12}} \sin{\left(\frac{s}{r} - \phi_{H_{12}}' \right)} + \frac{U_{max}}{1 + U_{max}^2}, & H_{12} (s) \geq 0 \\
        \lambda_{H_{12}} \sin{\left(\frac{s}{r} - \phi_{H_{12}} \right)} - \frac{U_{max}}{1 + U_{max}^2}, & H_{12} (s) < 0
    \end{cases}, \label{eq: parameterized_expression_H12} \\
    \frac{dH_{12} (s)}{ds} &= \begin{cases}
        \frac{\lambda_{H_{12}}}{r} \cos{\left(\frac{s}{r} - \phi_{H_{12}}' \right)}, & H_{12} (s) \geq 0 \\
        \frac{\lambda_{H_{12}}}{r} \cos{\left(\frac{s}{r} - \phi_{H_{12}} \right)}, & H_{12} (s) < 0
    \end{cases}, \label{eq: parameterized_expression_dH12_ds}
\end{align}
where the relationship between $\phi_{H_{12}}$ and $\phi_{H_{12}}'$ can be obtained leveraging continuity of $H_{12}$ and $\frac{dH_{12}}{ds}$ at $H_{12} = 0.$ In the above equations, $r = \frac{1}{\sqrt{1 + U_{max}^2}}$ was used.\footnotemark\footnotetext{If $\lambda_{H_{12}} \leq 0$, $\lambda_{H_{12}}$ can be equivalently reparameterized such that $\lambda_{H_{12}} \geq 0$ by offsetting the phase $\phi_{H_{12}}'$ and $\phi_{H_{12}}$ by $\pi$ radians. Hence, $\lambda_{H_{12}} \geq 0$ can be considered without loss of generality.}

\begin{remark}
    The derived relationship in~\eqref{eq: ellipse_equations} between $H_{12}(s)$ and $\frac{d H_{12}(s)}{ds}$ is independent of the value of $\lambda_{H_{12}}$, and is therefore applicable to Cases~2.1, 2.2, and 2.3 discussed in subsequent subsections.
\end{remark}

We now derive the candidate optimal paths for the three subcases, depending on the value of $\lambda_{H_{12}}$ (see Fig.~\ref{fig: overview_cases_results}), in the following subsections.

\subsection{Case 2.1: $\lambda = 1,$ $\lambda_{H_{12}} < \frac{U_{max}}{1 + U_{max}^2}$}

In the following, we address the first subcase, namely $\lambda_{H_{12}} < U_{\max}/\left(1 + U_{\max}^2 \right)$.
In this case, the optimal path is claimed to be of type $C$. To this end, the following lemma is crucial to construct the phase portrait of $H_{12},$ using which this result is obtained.
\begin{lemma} \label{lemma: case_lambda_1_lambdaH12_less_than}
    For $\lambda = 1, \lambda_{H_{12}} < U_{max}/\left(1 + U_{max}^2\right),$ if $H_{12} (0) < 0,$ then $H_{12} (s) < 0 \,\forall\, s > 0.$ Similarly, if $H_{12} (0) > 0,$ then $H_{12} (s) > 0 \, \forall \, s$. Furthermore, $H_{12} (s) \not\equiv 0.$ 
\end{lemma}
\begin{proof}
    Suppose $H_{12} (0) < 0.$ If $H_{12} (s) = 0$ for some $s,$ then from \eqref{eq: ellipse_equations}, 
    \begin{align*}
        \frac{U_{max}^2}{\left(1 + U_{max}^2 \right)^2} &\leq \frac{U_{max}^2}{\left(1 + U_{max}^2 \right)^2} + \left(\frac{1}{\sqrt{1 + U_{max}^2}} \frac{dH_{12} (s)}{ds} \right)^2 \\
        & = \lambda_{H_{12}}^2 < \frac{U_{max}^2}{\left(1 + U_{max}^2 \right)^2},
    \end{align*}
    which is a contradiction. Similarly, it can be proved that if $H_{12} (0) > 0,$ $H_{12} (s) > 0 \, \forall \, s$. 
    Furthermore, from \eqref{eq: ellipse_equations}, it follows immediately that if $H_{12} \equiv 0,$ then $\lambda_{H_{12}} = U_{max}/\left(1 + U_{max}^2\right),$ which leads to a contradiction. Hence, the lemma is proved.
\end{proof}

The phase portrait obtained for $\lambda = 1, \lambda_{H_{12}} < U_{max}/\left(1 + U_{max}^2 \right)$ using \eqref{eq: ellipse_equations} and the previous observations is shown in Fig.~\ref{fig: phase_portrait_H12_lambda_1_lambdaH12_less_than}. Using Lemma~\ref{lemma: case_lambda_1_lambdaH12_less_than} and Fig.~\ref{fig: phase_portrait_H12_lambda_1_lambdaH12_less_than}, it can be observed that the candidate optimal path is $L, R,$ for $\lambda = 1, \lambda_{H_{12}} < U_{max}/\left(1 + U_{max}^2 \right)$.

\begin{figure}[htb!]
    \centering
    \includegraphics[width = 0.8\linewidth]{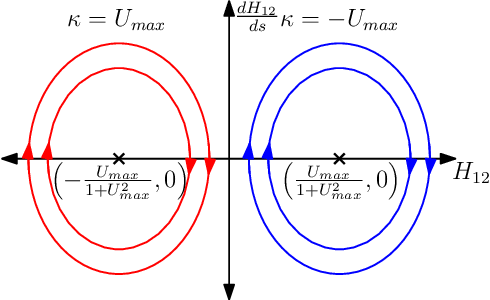}
    \caption{Phase portrait of $H_{12}$ for $\lambda = 1$ and $\lambda_{H_{12}} < \frac{U_{max}}{1 + U_{max}^2}$}
    \label{fig: phase_portrait_H12_lambda_1_lambdaH12_less_than}
\end{figure}

\subsection{Case 2.2: $\lambda = 1,$ $\lambda_{H_{12}} = \frac{U_{max}}{1 + U_{max}^2}$}

The equation corresponding to $H_{12} (s) < 0$ and $H_{12} (s) \geq 0$, given in \eqref{eq: ellipse_equations}, are ellipses, whose origin lies on the $H_{12}-$axis. The intersection of the ellipse corresponding to $H_{12} (s) < 0$ with the $H_{12}-$axis are at $\pm \lambda_{H_{12}} - \left(U_{max}/\left(1 + U_{max}^2 \right)\right),$ whereas for the ellipse corresponding to $H_{12} (s) \geq 0,$ the intersection points are at $\pm \lambda_{H_{12}} + \left(U_{max}/\left(1 + U_{max}^2 \right)\right).$ For $\lambda_{H_{12}} = U_{max}/\left(1 + U_{max}^2 \right),$ the intersection points of the ellipse corresponding to $H_{12} (s) < 0$ are $0$ and $-2U_{max}/\left(1 + U_{max}^2 \right),$ whereas for the ellipse corresponding to $H_{12} (s) \geq 0,$ the intersection points are $0$ and $2U_{max}/\left(1 + U_{max}^2 \right)$. Using these observations and \eqref{eq: ellipse_equations}, the phase portrait of $H_{12}$ is obtained as shown in Fig.~\ref{fig: phase_portrait_H12_lambda_1_lambdaH12_equal}. It should be noted that in this case, $H_{12}$ can be identically zero since \eqref{eq: ellipse_equations} will be satisfied. Hence, a $G$ segment can be part of the optimal path.

The candidate paths for this case can be obtained using the following lemma.

\begin{figure}[htb!]
    \centering
    \includegraphics[width = 0.8\linewidth]{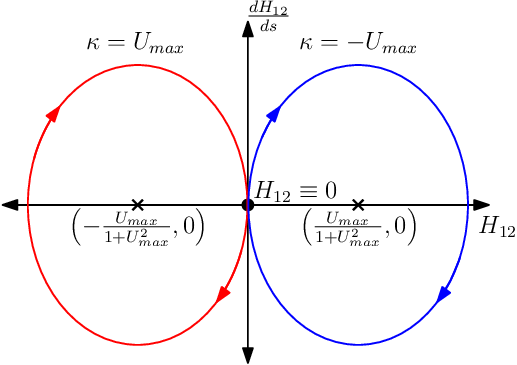}
    \caption{Phase portrait of $H_{12}$ for $\lambda = 1$ and $\lambda_{H_{12}} = \frac{U_{max}}{1 + U_{max}^2}$}
    \label{fig: phase_portrait_H12_lambda_1_lambdaH12_equal}
\end{figure}

\begin{lemma}
    For $\lambda = 1,$ $\lambda_{H_{12}} = \frac{U_{max}}{1 + U_{max}^2},$ the optimal path is $LGL, LGR, RGL, RGR,$ or a degenerate path of one of the four paths.
\end{lemma}
\begin{proof}
    To prove the claim in the lemma, it suffices to show that non-trivial paths of type $GCG,$ $GCC,$ $CCG,$ or $CCC$ are non-optimal. Using non-optimality of such paths, from Fig.~\ref{fig: phase_portrait_H12_lambda_1_lambdaH12_equal}, the lemma immediately follows. For this purpose, it is claimed that in the $GCG,$ $GCC,$ $CCG,$ and $CCC$ paths, the angle of the middle $C$ segment is $2 \pi$. The argument for the same will rely on the observation from the phase portrait given in Fig.~\ref{fig: phase_portrait_H12_lambda_1_lambdaH12_equal} that the inflection point in an optimal path is $H_{12} (s) = 0, \frac{dH_{12} (s)}{ds} = 0.$ Suppose the middle $C$ segment is $L$, and the arc length corresponding to the inflection points of the path be denoted by $s_1$ and $s_2$. Hence, $H_{12} (s_1) = H_{12} (s_2) = 0$ and $\frac{dH_{12}}{ds} (s_1) = \frac{dH_{12}}{ds} (s_2) = 0.$ Further, $H_{12} (s) < 0$ for all $s \in (s_1, s_2).$ Using the parameterized expressions for $H_{12}$ and $\frac{dH_{12}}{ds}$ given in \eqref{eq: parameterized_expression_H12} and \eqref{eq: parameterized_expression_dH12_ds}, respectively, and noting that $H_{12} (s)$ and $\frac{dH_{12} (s)}{ds}$ are continuous,
    \begin{align*}
        \sin{\left(\frac{s_i}{r} - \phi_{H_{12}} \right)} &= \frac{U_{max}}{1 + U_{max}^2} \frac{1}{\lambda_{H_{12}}} = 1, \\
        \cos{\left(\frac{s_i}{r} - \phi_{H_{12}} \right)} &= 0,
    \end{align*}
    for $i = 1, 2.$ Hence, $\frac{1}{r} \left(s_2 - s_1 \right),$ which denotes the arc angle of the $L$ segment, is equal to $2 n \pi,$ where $n = 1, 2, \cdots$ ($n \neq 0$ since path is non-trivial). 
    Hence, $GLG,$ $GLR,$ $RLG,$ and $RLR$ paths are non-optimal. Using a similar proof, $GRG,$ $GRL,$ $LRG,$ and $LRL$ paths can be shown to be non-optimal. Hence, using Fig.~\ref{fig: phase_portrait_H12_lambda_1_lambdaH12_equal} and non-optimality of $GCG,$ $GCC,$ $CCG,$ and $CCC$ paths, the optimal path is of type $CGC$ or a degenerate path of the same.
\end{proof}

\subsection{Case 2.3: $\lambda = 1,$ $\lambda_{H_{12}} > \frac{U_{max}}{1 + U_{max}^2}$}

As the intersection points of the ellipse corresponding to $H_{12} (s) < 0$ from \eqref{eq: ellipse_equations} with the $H_{12}-$axis are at $\pm \lambda_{H_{12}} - U_{max}/\left(1 + U_{max}^2 \right),$ one of the intersection points lies to the right of the line $H_{12} = 0$ for $\lambda_{H_{12}} > U_{max}/\left(1 + U_{max}^2 \right)$. Therefore, the ellipse corresponding to $H_{12} (s) < 0$ intersects the line $H_{12} (s) = 0$ at two points. The coordinates for these points can be obtained from \eqref{eq: ellipse_equations} as $\left(0, \pm \sqrt{1 + U_{max}^2} \sqrt{\lambda_{H_{12}}^2 - U_{max}^2/\left(1 + U_{max}^2 \right)^2} \right)$. Similarly, the ellipse corresponding to $H_{12} (s) > 0$ from \eqref{eq: ellipse_equations} can be observed to intersect the $H_{12} = 0$ line at two points, whose coordinates are obtained as $\left(0, \pm \sqrt{1 + U_{max}^2} \sqrt{\lambda_{H_{12}}^2 - U_{max}^2/\left(1 + U_{max}^2 \right)^2} \right).$ Further, if $H_{12} \equiv 0,$ then \eqref{eq: ellipse_equations} is not satisfied; hence, a $G$ segment is not part of the optimal path. Using these observations, the phase portrait for $\lambda_{H_{12}} > U_{max}/\left(1 + U_{max}^2 \right)$ can be obtained as shown in Fig.~\ref{fig: phase_portrait_H12_lambda_1_lambdaH12_greater}.
Furthermore, the following proposition follows using these observations and Fig.~\ref{fig: phase_portrait_H12_lambda_1_lambdaH12_greater}.

\begin{figure}[htb!]
    \centering
    \includegraphics[width = 0.8\linewidth]{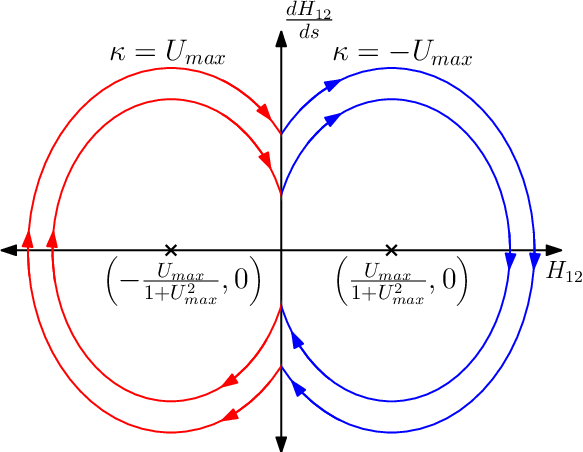}
    \caption{Phase portrait of $H_{12}$ for $\lambda = 1$ and $\lambda_{H_{12}} > \frac{U_{max}}{1 + U_{max}^2}$}
    \label{fig: phase_portrait_H12_lambda_1_lambdaH12_greater}
\end{figure}

\begin{proposition}
    For $\lambda = 1, \lambda_{H_{12}} > \frac{U_{max}}{1 + U_{max}^2},$ the optimal path is a concatenation of $C$ segments.
\end{proposition}

For the optimal path, which is a concatenation of $C$ segments, the following claim is made regarding the angle of the middle $C$ segments. It should be noted here that while the following result has been shown in \cite{3D_Dubins_sphere}, a simpler proof is shown here utilizing the phase portrait for deriving the result.
\begin{lemma} \label{lemma: lambda_1_CCC_middle_arc_angle}
    For $\lambda = 1$ and $\lambda_{H_{12}} > \frac{U_{max}}{1 + U_{max}^2},$ the angle of the middle arcs of a non-trivial optimal path with concatenation of $C$ segments is $\pi + \beta,$ where $\beta \in (0, \pi),$ i.e., the optimal path is of type $C_\alpha C_{\pi + \beta} C_{\pi + \beta} \cdots C_{\pi + \beta} C_\gamma,$ where $0 \leq \alpha, \gamma \leq \pi + \beta.$
\end{lemma}
\begin{proof}
    Consider an $L$ segment that is part of a non-trivial optimal path such that the considered segment is an intermediary segment in the path, i.e., not the first or the last segment of the path. Hence, this segment is completely traversed. Let arc lengths $s_1$ and $s_2$ represent the inflection points before and after the $L$ segment, respectively. Hence, from Fig.~\ref{fig: phase_portrait_H12_lambda_1_lambdaH12_greater} and the previously derived expressions for the inflection point, it follows that $H_{12} (s_1) = H_{12} (s_2) = 0,$ and $\frac{dH_{12} (s_1)}{ds} = - \sqrt{1 + U_{max}^2} \sqrt{\lambda_{H_{12}}^2 - U_{max}^2/\left(1 + U_{max}^2 \right)^2},$ $\frac{dH_{12} (s_2)}{ds} = \sqrt{1 + U_{max}^2} \sqrt{\lambda_{H_{12}}^2 - U_{max}^2/\left(1 + U_{max}^2 \right)^2}.$ Using the obtained values for $H_{12}$ and $\frac{dH_{12}}{ds}$ at the inflection points, and the closed-form expressions for $H_{12}$ and $\frac{dH_{12}}{ds}$ for the $L$ segment from \eqref{eq: parameterized_expression_H12} and \eqref{eq: parameterized_expression_dH12_ds} which are evaluated at the inflection points (due to continuity), 
    it follows that 
    \begin{align*}
        \sin{\left(\frac{s_i}{r} - \phi_{H_{12}} \right)} &= \frac{U_{max}}{1 + U_{max}^2} \frac{1}{\lambda_{H_{12}}}, \\
        \cos{\left(\frac{s_i}{r} - \phi_{H_{12}} \right)} &= \frac{(-1)^i}{\lambda_{H_{12}}} \sqrt{\lambda_{H_{12}}^2 - \frac{U_{max}^2}{\left(1 + U_{max}^2 \right)^2}}.
    \end{align*}
    Here, $i = 1, 2.$ Further, $\sqrt{1 + U_{max}^2} = \frac{1}{r}$ was used.
    The obtained conditions can be represented as shown in Fig.~\ref{fig: angle_L_segment_completely_traversed}. It should be noted that the expression for the arc angle of the $L$ segment, denoted by $\phi_L$, is desired to be obtained. The arc angle $\phi_L$ is related to $s_1$ and $s_2$ by $\phi_L = \left(s_2 - s_1\right)/r.$
    From Fig.~\ref{fig: angle_L_segment_completely_traversed}, it follows that $\phi_L = \pi + 2 \tan^{-1} \left(U_{max}/\sqrt{\lambda_{H_{12}}^2 \left(1 + U_{max}^2 \right)^2 - U_{max}^2} \right) = \pi + \beta,$ where $\beta \in (0, \pi)$ since $U_{max} > 0$, and $\lambda_{H_{12}}^2 \left(1 + U_{max}^2 \right)^2 - U_{max}^2 > 0$ since $\lambda_{H_{12}} > U_{max}/\left(1 + U_{max}^2 \right)$. Using a similar argument for an intermediary $R$ segment, it follows that the expression for the arc angle of the $R$ segment $\phi_R$ is the same as that of $\phi_L.$ Hence, the claim in the lemma follows.
    \begin{figure}[htb!]
        \centering
        \includegraphics[width = 0.7\linewidth]{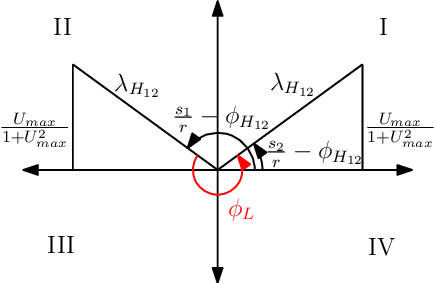}
        \caption{Angle of $L$ segment that is completely traversed for $\lambda_{H_{12}} > \frac{U_{max}}{1 + U_{max}^2}$}
        \label{fig: angle_L_segment_completely_traversed}
    \end{figure}
\end{proof}

Having provided a structure for the optimal path, an infinite list of candidate paths with concatenation of $C$ segments still exists. It is desired to obtain a finite list of candidate optimal paths. In \cite{3D_Dubins_sphere}, it was shown that a $C_\alpha C_{\pi + \beta} C_{\pi + \beta} C_\gamma$ path was non-optimal for $r \leq \frac{1}{2}$. While the authors conjectured that the candidate list of paths would contain more than three $C$ segments for $r > \frac{1}{2},$ the sufficient list of candidate paths for $r > \frac{1}{2}$ is unknown. To this end, the following claims are made:
\begin{itemize}
    \item A $C_\alpha C_{\pi + \beta} C_{\pi + \beta} C_{\pi + \beta} C_\gamma$ is non-optimal for $r \leq \frac{1}{\sqrt{2}}.$ Hence, for $\lambda_{H_{12}} > U_{max}/\left(1 + U_{max}^2 \right)$ and $r \leq \frac{1}{\sqrt{2}},$ the optimal path is of type $C_\alpha C_{\pi + \beta} C_{\pi + \beta} C_\gamma$ or a degenerate path of the same.
    \item A $C_\alpha C_{\pi + \beta} C_{\pi + \beta} C_{\pi + \beta} C_{\pi + \beta} C_\gamma$ is non-optimal for $r \leq \frac{\sqrt{3}}{2}.$ Hence, for $\lambda_{H_{12}} > U_{max}/\left(1 + U_{max}^2 \right)$ and $r \leq \frac{\sqrt{3}}{2}$, the optimal path is of type $C_\alpha C_{\pi + \beta} C_{\pi + \beta} C_{\pi + \beta} C_\gamma$ or a degenerate path of the same.
\end{itemize}
The proofs for these two claims are provided by Lemmas \ref{lemma: non-optimality_CCCCC} and \ref{lemma: claim2}, respectively.

\begin{lemma} \label{lemma: non-optimality_CCCCC}
    A non-trivial $CCCCC$ path is non-optimal for $r \leq \frac{1}{\sqrt{2}}$.
\end{lemma}
\begin{proof}
    For a non-trivial optimal $RLRLR$ path, all intermediary $C$ segments are of the same angle that is greater than $\pi$ from Lemma~\ref{lemma: lambda_1_CCC_middle_arc_angle}. Hence, consider a non-trivial optimal $R_{\alpha} L_{\pi + \beta} R_{\pi + \beta} L_{\pi + \beta} R_{\gamma}$ path, where $\beta \in (0, \pi), \alpha > 0, \gamma > 0$.     
    It is claimed that the considered path is non-optimal for $r \leq \frac{1}{\sqrt{2}}$ as the $L_{\pi} R_{\pi + \beta} L_{\pi}$ subpath can be replaced with an alternate $L_\phi R_{\pi - \beta} L_\phi$ path, where $\phi \leq \pi.$ It should be noted that since $\beta \in (0, \pi) \implies \pi - \beta < \pi + \beta,$ and $\phi \leq \pi$, showing that such an alternate path exists and connects the same initial and final configurations as the initial path is sufficient to show that the alternate path is of a lower length. 
    An illustration of the initial $L_{\pi} R_{\pi + \beta} L_{\pi}$ subpath and the alternative $L_\phi R_{\pi - \beta} L_\phi$ path is shown in Fig.~\ref{fig: Alternate_LRL_path_CCCCC_nonoptimality}.

    \begin{figure}[htb!]
        \centering
        \includegraphics[width = 0.8\linewidth]{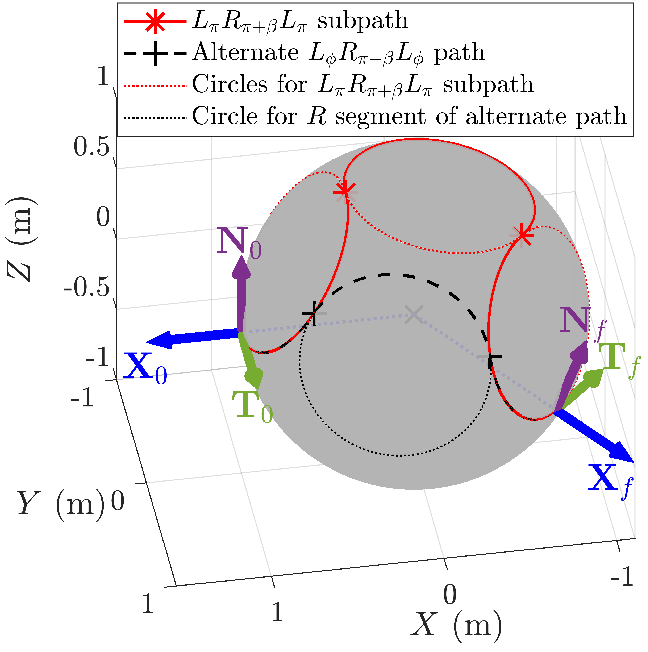}
        \caption{A $L_\phi R_{\pi - \beta} L_\phi$ path connecting the same configurations as the $L_{\pi} R_{\pi + \beta} L_{\pi}$ path for $r = 0.55, \beta = 40^\circ$. Animations of a vehicle moving along the initial and alternate paths are available on our GitHub page.}
        \label{fig: Alternate_LRL_path_CCCCC_nonoptimality}
    \end{figure}

    As it is claimed that the $L_\pi R_{\pi + \beta} L_\pi$ subpath can be replaced with an $L_\phi R_{\pi - \beta} L_\phi$ path, the net rotation matrix corresponding to the initial path and the alternate path should be equal. Hence,
    \begin{equation} \label{eq: matrix_equation_CCCCC_nonoptimality}
        \mathbf{R}_L (\phi) \mathbf{R}_R (\pi - \beta) \mathbf{R}_L (\phi) = \mathbf{R}_L (\pi) \mathbf{R}_R (\pi + \beta) \mathbf{R}_L (\pi).
    \end{equation}
    A solution for $\phi$ is obtained by comparing the entries in the second row and second column in the matrix in the LHS (denoted by $\alpha_{22}$) and RHS (denoted by $\beta_{22}$).
    Equating the expressions for $\alpha_{22}$ and $\beta_{22},$ the equation obtained is given by
    \begin{align} \label{eq: solving_phi_CCCCC_nonoptimality}
    \begin{split}
        &s^2_{\phi} \underbrace{\left(4 r^2 (r^2 - 1) + c_{\beta} \left(1 + (1 - 2 r^2)^2 \right) \right)}_{A} \\
        &+ s_{\phi} c_{\phi} \underbrace{\left(2 s_{\beta} (1 - 2 r^2) \right)}_{B} = 0,
    \end{split}
    \end{align}
    where $s_{\phi} := \sin{\phi}, c_{\phi} := \cos{\phi}.$
    A trivial solution to the above equation is $s_{\phi} = 0,$ which is discarded. Another solution to the above equation can be obtained by solving $A \sin{\phi} + B \cos{\phi} = 0,$
    which is the solution of interest. As $\beta \in (0, \pi)$ and since $r \leq \frac{1}{\sqrt{2}}$ is considered, $B \geq 0.$ It should be noted that when $r < \frac{1}{\sqrt{2}},$ $B > 0$. Hence, a solution for $\phi$ can always be obtained. For $r = \frac{1}{\sqrt{2}},$ $B = 0.$ However, the expression of $A$ for $r = \frac{1}{\sqrt{2}}$ is given by $A = 4 \left(\frac{1}{2} \right) \left(-\frac{1}{2} \right) + c_{\beta} = -1 + c_{\beta},$ which cannot equal zero as $0 < \beta < \pi.$ Hence, a solution for $\phi$ can always be obtained from the above equation, as $A^2 + B^2 \neq 0$ for $r \leq \frac{1}{\sqrt{2}}$.

    Let $s_{\theta} := \frac{B}{\sqrt{A^2 + B^2}}$ and $c_{\theta} := \frac{A}{\sqrt{A^2 + B^2}}$. Hence, the desired equation to be solved can be rewritten as $\sin{\left(\phi + \theta \right)} = 0$.
    As $B \geq 0,$ $s_{\theta} = \frac{B}{\sqrt{A^2 + B^2}} \geq 0.$ Hence, $\theta \in [0, \pi].$ Therefore, one of the solutions for the above equation is
    \begin{align} \label{eq: solution_phi_nonoptimality_CCCCC_path}
        \phi &= \pi - \theta,
    \end{align}
    which ensures that $\phi \in [0, \pi]$. Hence, if the alternate $L_\phi R_{\pi - \beta} L_\phi$ path connects the same initial and final configurations as the initial $L_\pi R_{\pi + \beta} L_\pi$ subpath, it is immediate that the $L_\pi R_{\pi + \beta} L_\pi$ subpath is non-optimal.

    To ensure that the constructed alternate path connects the same initial and final configurations as the initial path, all terms in the matrix equation in \eqref{eq: matrix_equation_CCCCC_nonoptimality} are compared using the obtained solution in Appendix~\ref{appsubsect: CCCCC_non_optimality_proof}. The constructed alternate path is shown to connect the same initial and final configurations as the initial path in Appendix~\ref{appsubsect: CCCCC_non_optimality_proof} since all terms in the matrix equation are shown to be equal using the obtained solution. Therefore, the $L_\pi R_{\pi + \beta} L_{\pi}$ subpath is non-optimal for $r \leq \frac{1}{\sqrt{2}},$ which implies that the considered $RLRLR$ path is non-optimal for $r \leq \frac{1}{\sqrt{2}}$. Using a similar proof for an $LRLRL$ path, it can be observed that a non-trivial $CCCCC$ path is not optimal for $r \leq \frac{1}{\sqrt{2}}.$
\end{proof}



The construction of an alternate path to show non-optimality of a $CCCCCC$ path in the lemma that follows is similar to the path construction used in the previous proof.

\begin{lemma} \label{lemma: claim2}
    A non-trivial $CCCCCC$ path is non-optimal for $r \leq \frac{\sqrt{3}}{2}$.
\end{lemma}
\begin{proof}
    For a non-trivial optimal $RLRLRL$ path, all intermediary $C$ segments are of the same angle that is greater than $\pi$ from Lemma~\ref{lemma: lambda_1_CCC_middle_arc_angle}. Hence, consider a non-trivial optimal $R_\alpha L_{\pi + \beta} R_{\pi + \beta} L_{\pi + \beta} R_{\pi + \beta} L_{\gamma}$ path, where $\beta \in (0, \pi), \alpha > 0, \gamma > 0.$ It should be noted that the considered path is non-optimal for $r \leq \frac{1}{\sqrt{2}}$, since the $L_{\pi + \beta} R_{\pi + \beta} L_{\pi + \beta}$ subpath is shown to be non-optimal using Lemma~\ref{lemma: non-optimality_CCCCC}. Hence, it is sufficient to show non-optimality for $r \in \left(\frac{1}{\sqrt{2}}, \frac{\sqrt{3}}{2} \right].$

    It is claimed that the considered path is non-optimal for $r \in \left(\frac{1}{\sqrt{2}}, \frac{\sqrt{3}}{2} \right]$ as the $L_\pi R_{\pi + \beta} L_{\pi + \beta} R_\pi$ subpath can be replaced with an alternate $L_\phi R_{\pi - \beta} L_{\pi - \beta} R_\phi$ path, where $\phi \leq \pi.$ Similar to the argument in the previous lemma, showing that the alternate path exists and connects the same initial and final configurations as the initial path is sufficient to show that the alternate path is of lower length. An illustration of the initial $L_{\pi} R_{\pi + \beta} L_{\pi + \beta} R_{\pi}$ subpath and the alternative $L_\phi R_{\pi - \beta} L_{\pi - \beta} R_\phi$ path is shown in Fig.~\ref{fig: Alternate_LRLR_path_CCCCCC_nonoptimality}.

    \begin{figure}[htb!]
        \centering
        \includegraphics[width = 0.8\linewidth]{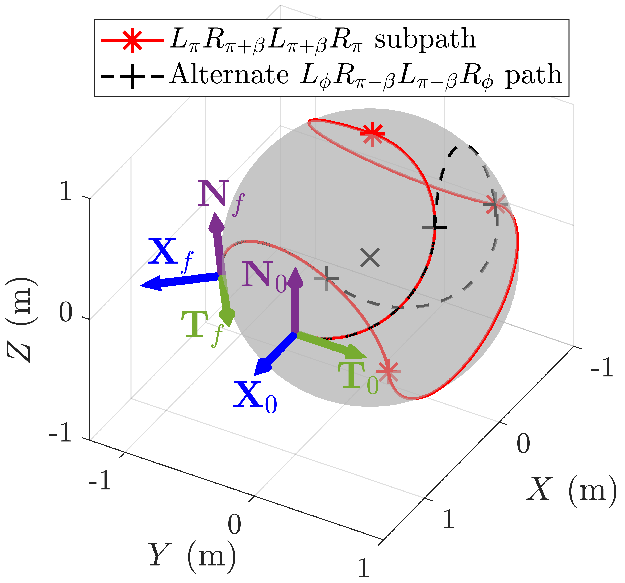}
        \caption{A $L_\phi R_{\pi - \beta} L_{\pi - \beta} R_\phi$ path connecting the same configurations as the $L_{\pi} R_{\pi + \beta} L_{\pi + \beta} R_{\pi}$ path for $r = 0.72, \beta = 40^\circ$. Animations of a vehicle moving along the initial and alternate paths are available on our GitHub page.}
        \label{fig: Alternate_LRLR_path_CCCCCC_nonoptimality}
    \end{figure}

    As it is claimed that the $L_\pi R_{\pi + \beta} L_{\pi + \beta} R_\pi$ subpath can be replaced with an $L_\phi R_{\pi - \beta} L_{\pi - \beta} R_\phi$ path, the net rotation matrix corresponding to the initial path and the alternate path should be equal. Hence,
    \begin{align} \label{eq: matrix_equation_CCCCCC_nonoptimality}
    \begin{split}
        &\mathbf{R}_L (\phi) \mathbf{R}_R (\pi - \beta) \mathbf{R}_L (\pi - \beta) \mathbf{R}_R (\phi) \\
        &= \mathbf{R}_L (\pi) \mathbf{R}_R (\pi + \beta) \mathbf{R}_L (\pi + \beta) \mathbf{R}_R (\pi).
    \end{split}
    \end{align}
    A solution for $\phi$ is obtained by comparing the entries $\eta_{22}$ and $\bar{\beta}_{22}$ in the LHS and RHS, respectively.
    It should be noted here that $\eta_{22}$ and $\bar{\beta}_{22}$ are terms in the second row and second column in the matrix in the LHS and RHS, respectively. 
    The equation obtained by setting $\eta_{22} = \bar{\beta}_{22}$ and simplifying is given by
    \begin{align} \label{eq: solving_phi_CCCCCC_nonoptimality}
        C \cos{2 \phi} + D \sin{2 \phi} = C,
    \end{align}
    where $r^2$ was canceled in the LHS and RHS. Here,
    \begin{align*}
        C &= 2 \left(3 r^2 - 2 \right) \left(r^2 - 1 \right) + 4 \left(2 r^2 - 1 \right) \left(r^2 - 1 \right) \cos{\beta} \\
        &\quad\, + \left(2 r^4 - 2 r^2 + 1 \right) \cos{2 \beta}, \\
        D &= 2 \sin{\beta} \left(2 \left(r^2 - 1 \right) + \left(2 r^2 - 1 \right) \cos{\beta} \right).
    \end{align*}
    It is claimed that \eqref{eq: solving_phi_CCCCCC_nonoptimality} can be used to solve for $\phi$ for $r \in \left(\frac{1}{\sqrt{2}}, \frac{\sqrt{3}}{2} \right]$. Noting that $\beta \in \left(0, \pi \right) \implies \sin{\beta} > 0$, and $2 r^2 - 1 > 0$ for $r$ in the chosen range, $D < 0$ for the considered range of $r$ since
    \begin{align*}
        2 \left(r^2 - 1 \right) + \left(2 r^2 - 1 \right) \cos{\beta} < 2 \left(r^2 - 1 \right) + \left(2 r^2 - 1 \right) \leq 0.
    \end{align*}
    Therefore, \eqref{eq: solving_phi_CCCCCC_nonoptimality} can be used to solve for $\phi$.

    Dividing both sides of \eqref{eq: solving_phi_CCCCCC_nonoptimality} by $\sqrt{C^2 + D^2},$ which is greater than zero since $D < 0,$ and defining an angle $\delta$ such that $c_{\delta} := C/\sqrt{C^2 + D^2}, s_{\delta} := D/\sqrt{C^2 + D^2},$ \eqref{eq: solving_phi_CCCCCC_nonoptimality} can be rewritten as
    \begin{align*}
        \cos{\left(2 \phi - \delta \right)} = \frac{C}{\sqrt{C^2 + D^2}}.
    \end{align*}
    A solution for $\phi$ given by 
    \begin{align} \label{eq: solution_phi_nonoptimality_CCCCCC_path}
        \phi = \delta - \pi,    
    \end{align}
    solves this equation, as $\cos{\left(2 \phi - \delta \right)} = \cos{\left(2 \delta - 2 \pi - \delta \right)} = \cos{\delta},$
    which, by definition of $\delta$, equals $C/\sqrt{C^2 + D^2}$. It should also be noted that since $D < 0,$ $\delta \in \left(\pi, 2 \pi \right)$ as $s_{\delta} < 0.$ Hence, for this choice of $\phi,$ $\phi \in \left(0, \pi \right).$ It is claimed that this solution of $\phi$ ensures that the alternate $L_\phi R_{\pi - \beta} L_{\pi - \beta} R_\phi$ path constructed connects the same initial and final configurations as the initial $L_\pi R_{\pi + \beta} L_{\pi + \beta} R_\pi$ path, i.e., \eqref{eq: matrix_equation_CCCCCC_nonoptimality} is satisfied.

    Using the solution obtained for $\phi,$ the other terms in the matrix equation in \eqref{eq: matrix_equation_CCCCCC_nonoptimality} are desired to be verified. The verification of the same is provided in Appendix~\ref{appsubsect: CCCCCC_non_optimality_proof}. Hence, the constructed alternate path connects the same initial and final configurations as the initial path. Further, the alternate path is shorter than the initial path as $0 < \phi < \pi$. Hence, the initially considered $RLRLRL$ path is non-optimal for $r \in \left(\frac{1}{\sqrt{2}}, \frac{\sqrt{3}}{2} \right].$ Using a similar proof for an $LRLRLR$ path, it can be observed that a non-trivial $CCCCCC$ path is not optimal for $r \leq \frac{\sqrt{3}}{2}.$
\end{proof}



Using the results shown in this section, which are summarized in Fig.~\ref{fig: overview_cases_results}, and the results in \cite{monroy, 3D_Dubins_sphere}, the theorem below follows.
\begin{theorem}
    The optimal path types for a Dubins vehicle on a sphere for turning radius $r$ up to $\frac{\sqrt{3}}{2}$ is given in Table~\ref{tab: theorem_candidate_paths}.
\end{theorem}

\begin{remark}
    For $r = \frac{1}{\sqrt{2}},$ a $CC_{\pi + \beta}C_{\pi + \beta}$ path is non-optimal for $\beta \in (0, \pi)$, as showed in \cite{monroy}. Hence, a $CC_{\pi + \beta}C_{\pi + \beta}C$ path is non-optimal.
\end{remark}

\section{Path Construction and Results} \label{sect: results}

Using the derived candidate optimal paths connecting a given initial and final configurations, it is now desired to construct the paths. Though the derived results corresponded to a unit sphere, we desire to derive the paths connecting a given initial configuration $\mathbf{R}_i = \begin{bmatrix}
    \mathbf{X}_i & \mathbf{T}_i & \mathbf{N}_i
\end{bmatrix}$ and final configuration $\mathbf{R}_f = \begin{bmatrix}
    \mathbf{X}_f & \mathbf{T}_f & \mathbf{N}_f
\end{bmatrix}$ on a sphere centered at the origin with radius $R.$ To plan the motion on the considered sphere, the sphere and the locations can be scaled to lie on a unit sphere. Hence, $\bar{\mathbf{R}}_i = \begin{bmatrix}
    \frac{\mathbf{X}_i}{R} & \mathbf{T}_i & \mathbf{N}_i
\end{bmatrix}$ and $\bar{\mathbf{R}}_f = \begin{bmatrix}
    \frac{\mathbf{X}_f}{R} & \mathbf{T}_f & \mathbf{N}_f
\end{bmatrix}.$ The scaled matrices $\bar{\mathbf{R}}_i$ and $\bar{\mathbf{R}}_f$ are rotation matrices. Furthermore, the minimum turning radius $r$ on the unit sphere will be the minimum turning radius on the initial sphere divided by $R.$

After scaling the configurations to lie on a unit sphere, the rotation matrix corresponding to each of the $L, R,$ and $G$ segments can be derived using the Sabban frame (summarized in Appendix~A). The arc angle for each segment in a path can then be obtained using these rotation matrices.\footnotemark\,  %
Using the implemented paths, we now provide numerical results regarding optimality of the identified path candidates.
\footnotetext{The detailed derivation for each path, along with the implementation for the path construction, are available at \cite{kumar2025generationpathsmotionplanning} and \protect\url{https://github.com/DeepakPrakashKumar/Motion-planning-on-sphere.git}, respectively.}

We remark here that for all numerical results presented in this section (and in our path construction code), all candidate paths, including both non-degenerate and degenerate variants, were considered.

\subsection{Optimality of $CC_\pi C$ and $CCCC$ paths}

An example instance where an $RL_\pi R$ path is optimal is shown in Fig.~\ref{subfig: RLpiR_path} for $r = 0.71$ (which is marginally larger than $\frac{1}{\sqrt{2}}$). Here, the arc angle of both the $R$ segments is $0.7$ radians. In this instance, the length of the $RL_\pi R$ path is $3.2245$ units. On the other hand, among the list of other paths that connect the same configuration, the shortest path is an $LRL$ path with length $6.6964.$ This in turn indicates the impact of consideration of a $CC_\pi C$ path for $r > \frac{1}{\sqrt{2}}.$ This example clarifies the role of candidate $CC_\pi C$ paths for $\frac{1}{\sqrt{2}} < r \leq \frac{\sqrt{3}}{2}$, a corollary of Lemma~\ref{lemma: nonoptimality_CCpiCpiC}.

On the other hand, an example instance in which an $RLRL$ path is optimal, wherein the arc angle of the middle $L$ and $R$ segments are equal and greater than $\pi$ radians, is shown in Fig.~\ref{subfig: RLRL_path}. In this instance, $r = 0.55,$ with the arc angle of the middle segments equaling $3.5458$ radians, and the arc angle of the first and last segments equaling $0.35$ radians. The length of the $RLRL$ path was $4.2853$ units, whereas the shortest path among the other paths connecting the same configurations was an $LRL$ path with $4.3643$ units.
This example supports the findings in Lemma~\ref{lemma: non-optimality_CCCCC}, which indicates that a $CCCC$ path can be optimal for $r > \frac{1}{2}$ (as a corollary).

\begin{figure}[htb!]
    \centering
    \subfloat[$RL_\pi R$ path for $r = 0.71$]{\includegraphics[width = 0.6\linewidth]{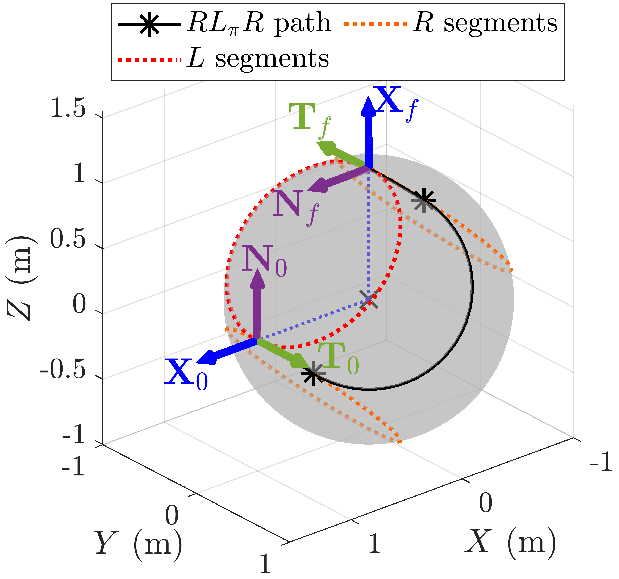}\label{subfig: RLpiR_path}}
    \hfil
    \subfloat[$RLRL$ path for $r = 0.55$]{\includegraphics[width = 0.6\linewidth]{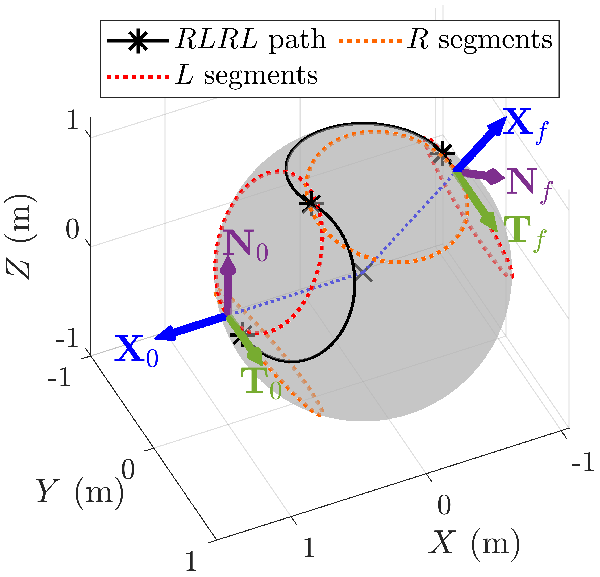}\label{subfig: RLRL_path}}
    \caption{Instances in which $CC_\pi C$ paths and $CCCC$ paths are optimal. Animations of these paths are available on our GitHub page.}
    \label{fig: numerical_results}
\end{figure}

\subsection{Optimality of $CGC$ and $CCC$ paths}

Example instances wherein $CGC$ and $CCC$ paths are optimal are shown in Fig.~\ref{fig: numerical_results_CGC_CCC}. Here, we consider $r = 0.4$ for illustration. For the $LGL, LGR, RGL, RGR$ paths, we pick $\phi_1 = 1.2$ radians, $\phi_2 = 0.6$ radians, and $\phi_3 = 1.4$ radians. For the $LRL$ and $RLR$ paths, we pick $\phi_1 = 1.5$ radians, $\phi_2 = \frac{3\pi}{2}$ radians, and $\phi_3 = 1.4$ radians. For each corresponding final configuration, the chosen path is indeed the shortest path among other candidate optimal paths.
\begin{figure*}[htb!]
    \centering
    \subfloat[$LGL$ path]{\includegraphics[width = 0.3\linewidth]{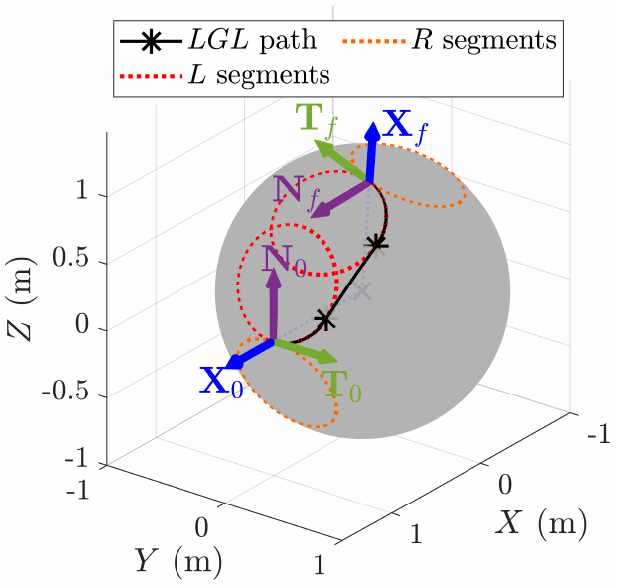}\label{subfig: LGL_path}}
    \hfil
    \subfloat[$LGR$ path]{\includegraphics[width = 0.32\linewidth]{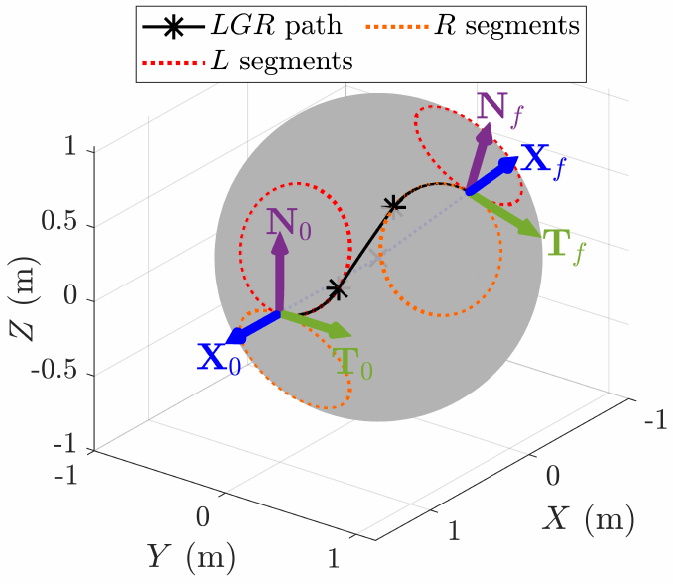}\label{subfig: LGR_path}}
    \hfil
    \subfloat[$LRL$ path]{\includegraphics[width = 0.34\linewidth]{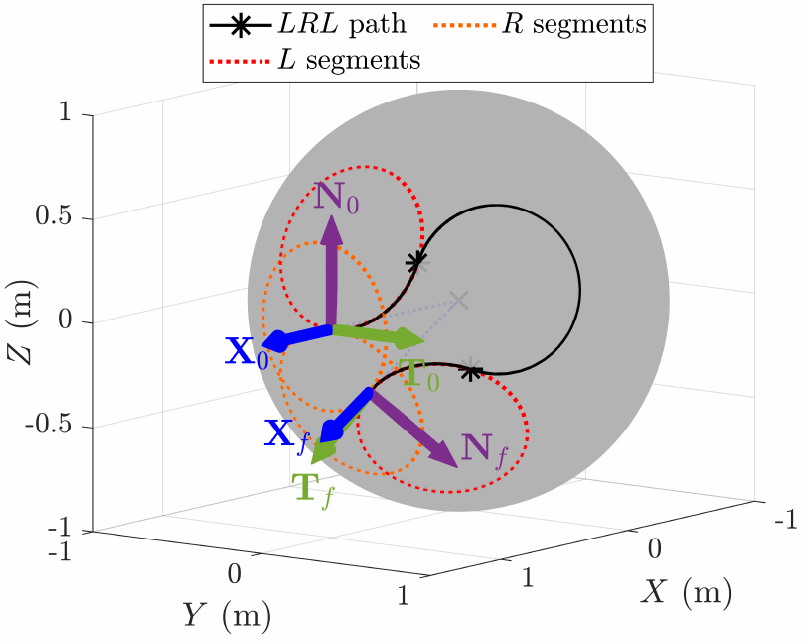}\label{subfig: LRL_path}}
    \caption{Instances in which $CGC$ and $CCC$ paths paths are optimal for $r = 0.4$. Instances for optimality of $RGR$, $RGL$, and $RLR$ paths can be obtained by reflecting the final configurations of the $LGL$, $LGR$, and $LRL$ cases, respectively, about the equator. Animations illustrating all six paths are available on our GitHub page.}
    \label{fig: numerical_results_CGC_CCC}
\end{figure*}

\subsection{Analysis of optimality of $CCCCC$ path}

Extensive numerical simulations were conducted to investigate the optimality of $CCCCC$ paths for $r \in (\frac{1}{\sqrt{2}}, \frac{\sqrt{3}}{2}]$. In our simulations, we discretized $r$ into 100 values uniformly sampled in $[\frac{1}{\sqrt{2}}+0.0001, \frac{\sqrt{3}}{2}]$, $\phi_1$ and $\phi_3$ into 20 values each uniformly sampled over $[0.0001, 0.2]$ radians, and $\phi_2$ into 100 values over $[\pi+0.0001, 2\pi-0.0001]$. This resulted in over 4 million parameter combinations. We remark here that we chose a smaller interval and finer discretization for $\phi_1$ and $\phi_3$ because, although a general $CCCCC$ path would allow $\phi_1, \phi_3 \in [0,2\pi]$, such subpaths are sufficient to capture the relevant cases for our analysis. If the original $CCCCC$ path is optimal, then the corresponding $CCCCC$ subpath with $\phi_1, \phi_3 \leq 0.2$ must also be optimal. No instance was found where a $CCCCC$ path was optimal, suggesting that such paths may not be a candidate optimal path for $r \in \left(\frac{1}{\sqrt{2}}, \frac{\sqrt{3}}{2} \right]$. Nevertheless, as a general proof of non-optimality for $CCCCC$ paths is still unknown, their inclusion in the sufficient candidate set remains valid.

\subsection{Varying final configuration and $r$ value}

Finally, we analyzed how often each path type is optimal over the space of final configurations. The final configuration on the sphere was parameterized by latitude $L \in [0,2\pi)$, longitude $l \in [-\pi/2,\pi/2]$, and heading angle $\psi \in [0,2\pi)$. Using the spherical-coordinate model in \cite{kumar2024equivalencedubinspathsphere}, $\mathbf{X}_f$, $\mathbf{T}_f$, and $\mathbf{N}_f$ are uniquely determined by $(l,L,\psi)$.\footnotemark\footnotetext{While the expression for $\mathbf{X}_f$ can easily be obtained as $\mathbf{X}_f = (\cos{l}\,\cos{L},\; \sin{l}\,\cos{L},\; \sin{L})^T$, the expression for $\mathbf{T}_f$ is given by
\[
    \mathbf{T}_f = \mathbf{R}_z(l)\, \mathbf{R}_y(-L - \pi/2)\, \mathbf{R}_z(\psi)\, (1,\; 0,\; 0)^T.
\]
Here, $\mathbf{R}_z$ and $\mathbf{R}_y$ denote standard elementary rotation matrices. The first two rotations align the coordinate axes (initially the columns of the identity matrix) with the $e_n$, $e_{lat}$, and $e_r$ axes as described in~\cite{kumar2024equivalencedubinspathsphere}. The final rotation by $\psi$ sets the desired orientation of the $x$ and $y$ axes to match the navigation frame. $\mathbf{N}_f = \mathbf{X}_f \times \mathbf{T}_f$ can finally be obtained.} Each of these three angles was discretized into $50$ values. The radius $r$ was discretized over three intervals, $[0.01,0.5]$, $[0.51,0.7]$, and $[0.71,0.86]$, with a step size of $0.01$ in each interval. The percentage of final configurations for which each path type is optimal is shown in Fig.~\ref{fig: percentage_configuration}. Degenerate two-segment and one-segment paths are grouped with the $CGC$ and $CCC$ classes.

From Fig.~\ref{fig: percentage_configuration}, we observe that for $r \leq 1/\sqrt{2}$, $CGC$ is predominantly optimal, whereas for larger $r$, the $LGR$, $RGL$, and $CCC$ paths dominate, highlighting the influence of $r$ on vehicle maneuverability. For $r > 1/2$, we also find configurations for which $CCCC$ paths are optimal. A few instances where the $CC_\pi C$ path is optimal arise for $r > 1/\sqrt{2}$. Finally, the $CCCCC$ path does not appear; this reinforces the hypothesis from the previous subsection that it is non-optimal.

\begin{figure}[htb!]
    \centering
    \includegraphics[width=\linewidth]{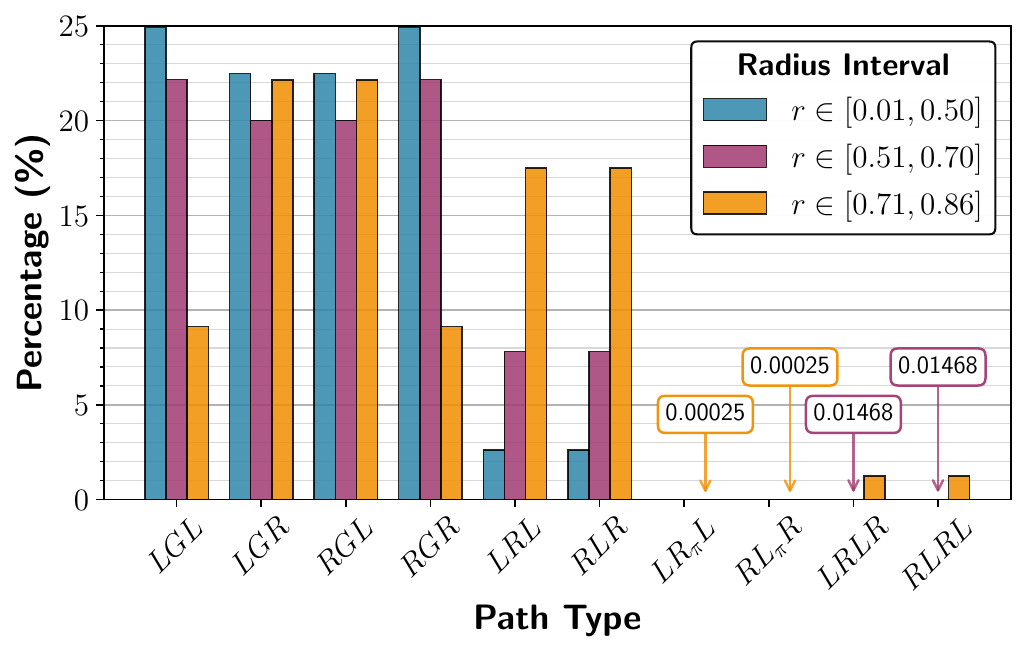}
    \caption{Percentage of instances in which each path type is optimal. Color coding indicate the corresponding radius intervals; for example, the value $0.00025$ corresponds to the $LR_\pi L$ path for $r \in [0.71,0.86]$.}
    \label{fig: percentage_configuration}
\end{figure}

Our analysis also revealed that, for the same final configuration, the optimal path type can change as the turning radius $r$ varies. For example, Fig.~\ref{fig: transition} shows a final configuration ($l = -90^\circ$, $L = 7.20^\circ,$ and $\psi = 230.40^\circ$) where the $RLR$ path is optimal for $r = 0.79$, but the $RLRL$ path becomes optimal for $r = 0.8$. This illustrates how the turning radius fundamentally influences the structure of the optimal path.

\begin{figure}[htb!]
    \centering
    \includegraphics[width=0.7\linewidth]{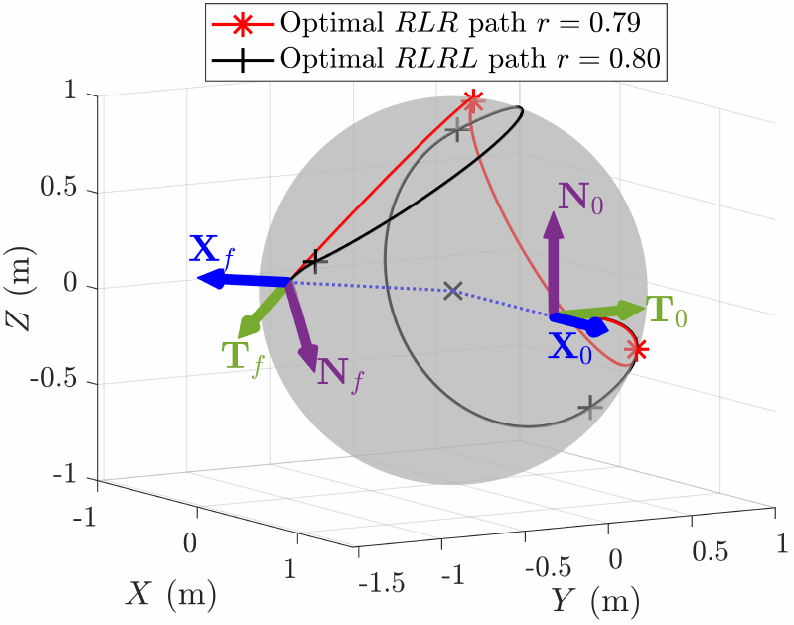}
    \caption{Transition from $CCC$ to $CCCC$ path with changing $r$. Animations of these paths are available on our GitHub page.}
    \label{fig: transition}
\end{figure}

\section{Conclusion}
\label{sec:conclusion}

In this article, a new model for motion planning in 3D was proposed, wherein the complete configuration description was considered. The motion constraints considered in this regard correspond to pitch rate and yaw rate constraints for the vehicle. As a step towards addressing this difficult problem, motion planning for a Dubins vehicle on a sphere was considered, which arises as an intermediary problem to be solved. The motion planning problem on a sphere is also a problem of independent interest, arising for planning the motion of high-speed aircraft moving at a constant altitude over the Earth. A model based on the Sabban frame was described to study the motion planning problem on a unit sphere, which has been established in \cite{3D_Dubins_sphere, monroy}. However, the results in \cite{3D_Dubins_sphere} were restricted to normal controls, i.e., when the coefficient of the adjoint variable corresponding to the integrand in the Hamiltonian is non-zero, and considered $r \leq \frac{1}{2}$. In \cite{monroy}, the results for the problem were derived for $r = \frac{1}{\sqrt{2}}$. Hence, existing literature has only considered $r \leq \frac{1}{2}$ and $r = \frac{1}{\sqrt{2}}$ for the sphere problem. In this article, the candidate optimal path results for abnormal controls and normal controls for $r \leq \frac{\sqrt{3}}{2}$ were derived. In particular, the optimal path is shown to be of type $CGC, CCCC,$ or a degenerate path of the same for $r \leq \frac{1}{\sqrt{2}}$ and $CGC, CC_\pi C, CCCCC,$ or a degenerate path of the same for $r \leq \frac{\sqrt{3}}{2}$. These results signify the importance of the curvature bound (or minimum turning radius), which is a parameter of the vehicle, on the optimal path, a result that has not been seen to the best of our knowledge in motion planning problems. Such a result indicates that the solution to the general 3D motion planning problem will be dependent on the yaw rate and pitch rate bounds, which deviates from the current literature for the 3D problem. 

\begin{remark}
The sufficient list of candidate optimal paths presented in this work are valid only for $r \leq \frac{\sqrt{3}}{2}$. For $r > \frac{\sqrt{3}}{2}$, identifying the finite list of candidate optimal paths remains an open problem. However, we hypothesize that as $r \to 1$, as vehicle maneuverability becomes increasingly restricted, the number of necessary path concatenations will tend to infinity. This phenomenon could arise as $r \rightarrow 1,$ where the $L$ and $R$ segments converge to the $G$ segment. Consequently, an increasing number of concatenations is required to reach the desired final configuration.
\end{remark}

\begin{remark}
    In practice, for $r > \frac{\sqrt{3}}{2}$, the candidate paths derived in this paper may be used as heuristic solutions to obtain feasible paths; however, there will be instances where a feasible solution cannot be obtained from this list. In such cases, practitioners can derive analytical expressions or numerically solve for paths with concatenated $C$ segments: $C_\alpha C_\pi C_\pi \cdots C_\pi C_\gamma$ for abnormal controls (with two unknown angles), and $C_\alpha C_{\pi + \beta} C_{\pi + \beta} \cdots C_{\pi + \beta} C_{\gamma}$ for normal controls (with three unknown angles). For analytically determining the expressions for these angles for paths with more $C$ segment concatenations than considered in this paper, practitioners may extend our path construction methodology as detailed in~\cite{kumar2025generationpathsmotionplanning}.
\end{remark}

\bibliographystyle{IEEEtran}
\bibliography{cite}

\appendix

\subsection{Expressions for rotation matrices} \label{appsubsect: rotation_matrices}

The expression for the rotation matrices $\mathbf{R}_L, \mathbf{R}_R$ and $\mathbf{R}_G,$ which were derived in \cite{free_terminal_sphere}, is given below for completeness:
\begin{align*} 
    \mathbf{R}_G (\phi) &= \begin{pmatrix}
        c_\phi & - s_\phi & 0 \\
        s_\phi & c_\phi & 0 \\
        0 & 0 & 1
    \end{pmatrix}, \\
    \mathbf{R}_L (r, \phi) &= \begin{pmatrix}
        \rho_{11} & - r s_\phi & \rho_{13} \\
        r s_\phi & c_\phi & - \rho_{23} \\
        \rho_{13} & \rho_{23} & \rho_{33}
    \end{pmatrix}, \\
    \mathbf{R}_R (r, \phi) &= \begin{pmatrix}
        \rho_{11} & - r s_\phi & -\rho_{13} \\
        r s_\phi & c_\phi & \rho_{23} \\
        -\rho_{13} & -\rho_{23} & \rho_{33}
    \end{pmatrix}.
\end{align*}
In the above equation, $c_\phi := \cos{\phi},$ $s_\phi := \sin{\phi},$ and
\begin{align*}
    \rho_{11} &= 1 - (1 - c_\phi) r^2, \quad \rho_{13} = (1 - c_\phi) r \sqrt{1 - r^2}, \\
    \rho_{23} &= s_\phi \sqrt{1 - r^2}, \quad \rho_{33} = c_\phi + (1 - c_\phi) r^2.
\end{align*}

\subsection{Intermediary calculations for deriving \eqref{eq: first_order_equation_GRG_path}} \label{appsubsect: calculations}

Differentiating \eqref{eq: LRL_path_matrix_equation} with respect to $\delta,$ and using product rule and chain rule for differentiation, the equation obtained is given by
\begin{align}
\begin{split}
    &\frac{d \mathbf{R}_L (r, \delta)}{d\delta} \mathbf{R}_R (r, \pi) \mathbf{R}_L (r, \delta) + \mathbf{R}_L (r, \delta) \mathbf{R}_R (r, \pi) \frac{d \mathbf{R}_L (r, \delta)}{d\delta} \\
    &= \frac{d\mathbf{R}_G (\phi_1 (\delta))}{d\phi_1} \frac{d\phi_1(\delta)}{d \delta} \mathbf{R}_R (r, \pi + \phi_2 (\delta)) \mathbf{R}_G (\phi_1 (\delta)) \\
    & \,\,\,\, + \mathbf{R}_G (\phi_1 (\delta)) \frac{d\mathbf{R}_R (r, \pi + \phi_2 (\delta))}{d\left(\pi + \phi_2 \right)} \frac{d \left(\pi + \phi_2 (\delta) \right)}{d\delta} \mathbf{R}_G (\phi_1 (\delta)) \\
    & \,\,\,\, + \mathbf{R}_G (\phi_1 (\delta)) \mathbf{R}_R (r, \pi + \phi_2 (\delta)) \frac{d\mathbf{R}_G (\phi_1 (\delta))}{d\phi_1} \frac{d\phi_1 (\delta)}{d\delta}.
\end{split}
\end{align}

Noting that $\frac{d \mathbf{R}_S (\phi)}{d \phi} = \mathbf{R}_S (\phi) \hat{\Omega}_S,$ where $S \in \{L, R, G\}$ and $\phi$ is the arc angle of the segment (which can be $\phi_1,$ $\phi_2,$ or $\delta$ in the above equation), the equation obtained is given by
\begin{align} \label{eq: LRL_path_matrix_equation_derivative}
\begin{split}
    &\mathbf{R}_L (r, \delta) \hat{\Omega}_L \mathbf{R}_R (r, \pi) \mathbf{R}_L (r, \delta) + \mathbf{R}_L (r, \delta) \mathbf{R}_R (r, \pi) \mathbf{R}_L (r, \delta) \hat{\Omega}_L \\
    &= \frac{d \phi_1 (\delta)}{d \delta} \bigg(\mathbf{R}_G (\phi_1 (\delta)) \hat{\Omega}_G \mathbf{R}_R (r, \pi + \phi_2 (\delta)) \mathbf{R}_G (\phi_1 (\delta))\\
    &\hspace{1.8cm}+ \mathbf{R}_G (\phi_1 (\delta)) \mathbf{R}_R (r, \pi + \phi_2 (\delta)) \mathbf{R}_G (\phi_1 (\delta)) \hat{\Omega}_G \bigg) \\
    & \quad\, + \frac{d \phi_2 (\delta)}{d \delta} \mathbf{R}_G (\phi_1 (\delta)) \mathbf{R}_R (r, \pi + \phi_2 (\delta)) \hat{\Omega}_R \mathbf{R}_G (\phi_1 (\delta)).
\end{split}
\end{align}
Note that $\frac{d\phi_1 (\delta)}{d\delta}$ and $\frac{d\phi_2 (\delta)}{d\delta}$ can be pulled out since it is a scalar term.

Evaluating the above equation for $\delta = 0,$ the equation can be simplified in terms of the Taylor coefficients $a_1$ and $a_2$ as obtained in \eqref{eq: first_order_equation_GRG_path}. This is because when $\delta = 0,$ $\phi_1 (\delta = 0) = \phi_2 (\delta = 0) = 0$; moreover, $\mathbf{R}_L (0),$ $\mathbf{R}_R (0),$ and $\mathbf{R}_G (0)$ reduce to the identity matrix, since the vehicle is not changing its configuration when it goes through such a rotation. Hence, we obtain the above equation.

\subsection{Computations for non-optimality of $RLRLR$ path for $r \leq \frac{1}{\sqrt{2}}$ in Lemma~\ref{lemma: non-optimality_CCCCC}} \label{appsubsect: CCCCC_non_optimality_proof}

The equation to be satisfied by the alternately constructed $LRL$ path to connect the same initial and final configurations as the considered $L_\pi R_{\pi + \beta} L_\pi$ subpath of the $RLRLR$ path is given in \eqref{eq: matrix_equation_CCCCC_nonoptimality}. The expression for the net rotation matrix corresponding to the $L_\pi R_{\pi + \beta} L_\pi$ is given by
\begin{align} \label{eq: CCCCC_nonoptimality_LRL_net_matrix}
\begin{split}
    \mathbf{R}_L (\pi) \mathbf{R}_R (\pi + \beta) \mathbf{R}_L (\pi) &= 
    \begin{pmatrix}
        \beta_{11} & \beta_{12} & \beta_{13} \\
        -\beta_{12} & \beta_{22} & \beta_{23} \\
        \beta_{13} & -\beta_{23} & \beta_{33}
    \end{pmatrix},
\end{split}
\end{align}
where
\begin{align} \label{appeq: beta_expressions}
\begin{split}
    \beta_{11} &= \left(1 - 4 r^2 \right)^2 \left(1 - r^2 \right) - r^2 \left(3 - 4 r^2 \right)^2 \cos{\beta}, \\
    \beta_{12} &= r \left(4 r^2 - 3 \right) \sin{\beta}, \quad\, \beta_{22} = -\cos{\beta}, \\
    \beta_{13} &= r \sqrt{1 - r^2} \left(4 r^2 - 1 \right) \left(4 r^2 - 3 \right) \left(1 + \cos{\beta} \right), \\
    \beta_{23} &= \sqrt{1 - r^2} \left(4 r^2 - 1 \right) \sin{\beta}, \\
    \beta_{33} &= r^2 \left(3 - 4 r^2 \right)^2 - \left(4 r^2 - 1 \right)^2 \left(1 - r^2 \right) \cos{\beta}.
\end{split}
\end{align}
The expression for the net rotation matrix corresponding to the $L_\phi R_{\pi - \beta} L_\phi$ path is given by 
\begin{align} \label{eq: CCCCC_nonoptimality_LRL_net_matrix_alt}
    &\mathbf{R}_L (\phi) \mathbf{R}_R (\pi - \beta) \mathbf{R}_L (\phi) = \begin{pmatrix}
        \alpha_{11} & \alpha_{12} & \alpha_{13} \\
        -\alpha_{12} & \alpha_{22} & \alpha_{23} \\
        \alpha_{13} & -\alpha_{23} & \alpha_{33}
    \end{pmatrix},
\end{align}
where
{\allowdisplaybreaks
\begin{align}
\begin{split} \label{appeq: alpha11}
    \alpha_{11} &= \left(2 r^2 - 1 \right)^2 \left(1 - r^2 \right) - 4 r^2 \left(1 - r^2 \right)^2 c_{\beta} \\
    & \quad\, + 4 r^2 \left(1 - r^2 \right) \left(1 - 2 r^2 \right) \left(1 + c_{\beta} \right) c_{\phi} \\
    & \quad\, + 4 r^2 \left(r^2 - 1 \right) s_{\beta} s_{\phi} + r^2 c_{\beta} s^2_{\phi} \\
    & \quad\, + \left(4 r^4 \left(1 - r^2 \right) - r^2 \left(2 r^2 - 1 \right)^2 c_{\beta} \right) c^2_{\phi} \\
    & \quad\, + 2 r^2 \left(1 - 2 r^2 \right) s_{\beta} s_{\phi} c_{\phi},
\end{split} \\
\begin{split} \label{appeq: alpha12}
    \alpha_{12} &= 2 r \left(1 - r^2 \right) \left(2 r^2 - 1 \right) \left(1 + c_{\beta} \right) s_{\phi} + 2 r \left(r^2 - 1 \right) s_{\beta} c_{\phi} \\
    & \quad\, + r \left(2 r^2 - 1 \right) s_{\beta} s^2_{\phi} - r s_{\beta} \left(2 r^2 - 1 \right) c^2_{\phi} \\
    & \quad\, + 2 \left(-2 r^3 \left(1 - r^2 \right) + r c_{\beta} \left(2 r^4 - 2 r^2 + 1 \right) \right) s_{\phi} c_{\phi},
\end{split} \\
\begin{split} \label{appeq: alpha13}
    \alpha_{13} &= \left(2 r^2 - 1 \right)^2 r \sqrt{1 - r^2} - 4 r^3 \left(1 - r^2 \right)^{\frac{3}{2}} c_{\beta} \\
    & \quad\, - 2 r \sqrt{1 - r^2} \left(2 r^2 - 1 \right)^2 \left(1 + c_{\beta} \right) c_{\phi} \\
    & \quad\, - 2 r \sqrt{1 - r^2} \left(2 r^2 - 1 \right) s_{\beta} s_{\phi} \\
    & \quad\, + 2 r \sqrt{1 - r^2} \left(2 r^2 - 1 \right) s_{\beta} s_{\phi} c_{\phi} - r \sqrt{1 - r^2} c_{\beta} s^2_{\phi} \\
    & \quad\, + \left(\left(2 r^2 - 1 \right)^2 r \sqrt{1 - r^2} c_{\beta} - 4 r^3 \left(1 - r^2 \right)^{\frac{3}{2}} \right) c^2_{\phi},
\end{split} \\
\begin{split} \label{appeq: alpha22}
    \alpha_{22} &= 4 r^2 s^2_{\phi} (r^2 - 1) + c_{\beta} \left(- c^2_{\phi} + (1 - 2 r^2)^2 s^2_{\phi} \right) \\
    & \quad\, + 2 s_{\beta} s_{\phi} c_{\phi} (1 - 2 r^2),
\end{split} \\
\begin{split} \label{appeq: alpha23}
    \alpha_{23} &= 2 r^2 \sqrt{1 - r^2} \left(\left(1 - 2 r^2 \right) \left(1 + c_{\beta} \right) s_{\phi} + s_{\beta} c_{\phi} \right) \\
    & \quad- \left(1 - 2 r^2 \right) \sqrt{1 - r^2} s_{\beta} \left(s^2_{\phi} - c^2_{\phi} \right) \\
    & \quad\, + \sqrt{1 - r^2} \Big(-4 r^2 \left(1 - r^2 \right) \\
    & \hspace{2.2cm} + \left(1 + \left(1 - 2 r^2 \right)^2 \right) c_{\beta} \Big) s_{\phi} c_{\phi},
\end{split} \\
\begin{split} \label{appeq: alpha33}
    \alpha_{33} &= r^2 \left(2 r^2 - 1 \right)^2 + 4 r^4 \left(r^2 - 1 \right) c_{\beta} \\
    & \quad\, + 4 r^2 \left(1 - r^2 \right) \left(2 r^2 - 1 \right) \left(1 + c_{\beta} \right) c_{\phi} \\
    & \quad\, + 4 r^2 \left(1 - r^2 \right) s_{\beta} s_{\phi} \\
    & \quad\, + \left(1 - r^2 \right) \left(4 r^2 \left(1 - r^2 \right) - \left(2 r^2 - 1 \right)^2 c_{\beta} \right) c^2_{\phi} \\
    & \quad\, + \left(1 - r^2 \right) c_{\beta} s^2_{\phi} + 2 s_{\beta} \left(1 - r^2 \right) \left(1 - 2 r^2 \right) s_{\phi} c_{\phi}.
\end{split}
\end{align}}
Now, consider the solution identified for $\phi$, which is given in ~\eqref{eq: solution_phi_nonoptimality_CCCCC_path}. For this solution, noting that $s_{\phi} = s_{\pi - \theta} = s_{\theta},$ $c_{\phi} = c_{\pi - \theta} = - c_{\theta},$ it is desired to shown that the two net rotation matrices given in \eqref{eq: CCCCC_nonoptimality_LRL_net_matrix} and \eqref{eq: CCCCC_nonoptimality_LRL_net_matrix_alt} are equal. Noting that the relation $\alpha_{22} = \beta_{22}$ was used to derive the expression for $\phi$, the following five relations must be shown using the considered solution for $\phi$ to prove that the two net rotation matrices given in \eqref{eq: CCCCC_nonoptimality_LRL_net_matrix} and \eqref{eq: CCCCC_nonoptimality_LRL_net_matrix_alt} are equal:
\begin{enumerate}
    \item $\alpha_{11} = \beta_{11}$
    \item $\alpha_{12} = \beta_{12}$
    \item $\alpha_{13} = \beta_{13}$
    \item $\alpha_{23} = \beta_{23}$
    \item $\alpha_{33} = \beta_{33}$
\end{enumerate}
To obtain the expression for $s_{\phi}$ and $c_{\phi}$, the definition of $s_{\theta}$ and $c_{\theta}$ in \eqref{eq: solving_phi_CCCCC_nonoptimality} in terms of $A$ and $B$, which are defined in the same equation, is desired to be used. First, the expression for $A^2 + B^2$ can be obtained as
\begin{align*}
    A^2 + B^2 &= 4 \left(1 - 2 r^2 \left(1 - r^2 \right) \left(1 + c_{\beta} \right) \right)^2.
\end{align*}
Since $\beta \in (0, \pi),$ $\left(1 - 2 r^2 \left(1 - r^2 \right) \left(1 + c_{\beta} \right) \right) > 1 - 4 r^2 (1 - r^2) = (2 r^2 - 1)^2 \geq 0.$ Hence, the expression for $\sqrt{A^2 + B^2}$ is given by $2 \left(1 - 2 r^2 \left(1 - r^2 \right) \left(1 + c_{\beta} \right) \right).$ Therefore, using the definition of $s_{\theta}$ and $c_{\theta}$ in \eqref{eq: solving_phi_CCCCC_nonoptimality},
\begin{align}
    s_{\phi} &= s_{\theta} = \frac{s_{\beta} \left(1 - 2 r^2 \right)}{d}, \label{eq: expression_sphi_alt_path_CCCCC} \\
    c_{\phi} &= - c_{\theta} = -\frac{4 r^2 (r^2 - 1) + c_{\beta} \left(1 + (1 - 2 r^2)^2 \right)}{2 d}. \label{eq: expression_cphi_alt_path_CCCCC}
\end{align}
Here, $d := 1 - 2 r^2 \left(1 - r^2 \right) \left(1 + c_{\beta} \right).$

\textbf{Claim 1:} $\alpha_{11} = \beta_{11}$ for $s_{\phi}$ and $c_{\phi}$ as given in \eqref{eq: expression_sphi_alt_path_CCCCC} and \eqref{eq: expression_cphi_alt_path_CCCCC}, respectively.

\begin{proof}
    Substituting the considered expressions for $s_{\phi}$ and $c_{\phi}$ in the expression for $\alpha_{11}$ in \eqref{appeq: alpha11} and simplifying,
    \begin{align} \label{eq: expression_alpha_11_CCCCC_path_soln}
    \begin{split}
        \alpha_{11} &= \left(2 r^2 - 1 \right)^2 \left(1 - r^2 \right) - 4 r^2 \left(1 - r^2 \right)^2 c_{\beta} \\
        & \quad\, + \frac{t_{2, \alpha_{11}}}{d} + \frac{t_{3, \alpha_{11}}}{d^2}
    \end{split}
    \end{align}
    where $t_{2, \alpha_{11}}$ is contributed to by $s_{\phi}$ and $c_{\phi}$ terms in \eqref{appeq: alpha11}, and $t_{3, \alpha_{11}}$ is contributed to by $s^2_{\phi},$ $c^2_{\phi},$ and $s_{\phi} c_{\phi}$ terms in \eqref{appeq: alpha11}. These terms can be obtained as
    \begin{align*}
        t_{2, \alpha_{11}} &= 4 r^2 \left(r^2 - 1 \right) \left(1 - 2 r^2 \right) \left(1 + c_{\beta} \right) d, \\
        t_{3, \alpha_{11}} &= \left(4 r^4 \left(1 - r^2 \right) - r^2 \left(2 r^2 - 1 \right)^2 c_{\beta} \right) d^2.
    \end{align*}
    Substituting the obtained expressions for $t_{2, \alpha_{11}}$ and $t_{3, \alpha_{11}}$ in \eqref{eq: expression_alpha_11_CCCCC_path_soln} and simplifying, it follows that
    \begin{align*}
        \alpha_{11} &= \left(1 - 4 r^2 \right)^2 \left(1 - r^2 \right) - r^2 \left(3 - 4 r^2 \right)^2 \cos{\beta} = \beta_{11}.
    \end{align*}
\end{proof}

\textbf{Claim 2:} $\alpha_{12} = \beta_{12}$ for $s_{\phi}$ and $c_{\phi}$ as given in \eqref{eq: expression_sphi_alt_path_CCCCC} and \eqref{eq: expression_cphi_alt_path_CCCCC}, respectively.

\begin{proof}
    Substituting the considered expression for $s_{\phi}$ and $c_{\phi}$ in the expression for $\alpha_{12}$ in \eqref{appeq: alpha12} and simplifying,
    \begin{align} \label{eq: expression_alpha_12_CCCCC_path_soln}
        \alpha_{12} &= \frac{t_{1, \alpha_{12}}}{d} + \frac{t_{2, \alpha_{12}}}{d^2},
    \end{align}
    where $t_{1, \alpha_{12}}$ is contributed to by $s_{\phi}$ and $c_{\phi}$ terms in \eqref{appeq: alpha12}, and $t_{2, \alpha_{12}}$ is contributed to by $s^2_{\phi},$ $c^2_{\phi},$ and $s_{\phi} c_{\phi}$ terms in \eqref{appeq: alpha12}. These terms can be obtained as
    \begin{align*}
        t_{1, \alpha_{12}} &= 2 r \left(r^2 - 1 \right) s_{\beta} d, \quad\,
        t_{2, \alpha_{12}} = r \left(2 r^2 - 1 \right) s_{\beta} d^2.
    \end{align*}
    Substituting the obtained expressions for $t_{1, \alpha_{12}}$ and $t_{2, \alpha_{12}}$ in \eqref{eq: expression_alpha_12_CCCCC_path_soln} and simplifying,
    \begin{align*}
        \alpha_{12} &= \left(4 r^3 - 3 r \right) s_{\beta} = \beta_{12}.
    \end{align*}
\end{proof}

\textbf{Claim 3:} $\alpha_{13} = \beta_{13}$ for $s_{\phi}$ and $c_{\phi}$ as given in \eqref{eq: expression_sphi_alt_path_CCCCC} and \eqref{eq: expression_cphi_alt_path_CCCCC}, respectively.

\begin{proof}
    Substituting the considered expression for $s_{\phi}$ and $c_{\phi}$ in the expression for $\alpha_{13}$ in \eqref{appeq: alpha13} and simplifying,
    \begin{align} \label{eq: expression_alpha_13_CCCCC_path_soln}
    \begin{split}
        \alpha_{13} &= \left(2 r^2 - 1 \right)^2 r \sqrt{1 - r^2} - 4 r^3 \left(1 - r^2 \right)^{\frac{3}{2}} c_{\beta} \\
        & \quad\, + \frac{t_{2, \alpha_{13}}}{d} + \frac{t_{3, \alpha_{13}}}{d^2},
    \end{split}
    \end{align}
    where $t_{2, \alpha_{13}}$ is contributed to by $s_{\phi}$ and $c_{\phi}$ terms in \eqref{appeq: alpha13}, and $t_{3, \alpha_{13}}$ is contributed to by $s^2_{\phi},$ $c^2_{\phi},$ and $s_{\phi} c_{\phi}$ terms in \eqref{appeq: alpha13}. These terms can be obtained as
    \begin{align*}
    t_{2, \alpha_{13}} &= 2 r \sqrt{1 - r^2} \left(2 r^2 - 1 \right)^2 \left(1 + c_{\beta} \right) d, \\
    t_{3, \alpha_{13}} &= r \sqrt{1 - r^2} \left(\left(2 r^2 - 1 \right)^2 c_{\beta} - 4 r^2 \left(1 - r^2 \right) \right) d^2.
    \end{align*}
    Substituting the obtained expressions for $t_{2, \alpha_{13}}$ and $t_{3, \alpha_{13}}$ in \eqref{eq: expression_alpha_13_CCCCC_path_soln} and simplifying, it follows that
    \begin{align*}
    \alpha_{13} &= r \sqrt{1 - r^2} \left(1 + c_{\beta} \right) \left(4 r^2 - 1 \right) \left(4 r^2 - 3 \right) = \beta_{13}.
    \end{align*}
\end{proof}

\textbf{Claim 4:} $\alpha_{23} = \beta_{23}$ for $s_{\phi}$ and $c_{\phi}$ as given in \eqref{eq: expression_sphi_alt_path_CCCCC} and \eqref{eq: expression_cphi_alt_path_CCCCC}, respectively.

\begin{proof}
    Substituting the considered expression for $s_{\phi}$ and $c_{\phi}$ in the expression for $\alpha_{23}$ in \eqref{appeq: alpha23} and simplifying,
    \begin{align} \label{eq: expression_alpha_23_CCCCC_path_soln}
        \alpha_{23} &= \frac{t_{1, \alpha_{23}}}{d} + \frac{t_{2, \alpha_{23}}}{d^2},
    \end{align}
    where $t_{1, \alpha_{23}}$ is contributed to by $s_{\phi}$ and $c_{\phi}$ terms in \eqref{appeq: alpha23}, and $t_{2, \alpha_{23}}$ is contributed to by $s^2_{\phi},$ $c^2_{\phi},$ and $s_{\phi} c_{\phi}$ terms in \eqref{appeq: alpha23}. These terms can be obtained as
    \begin{align*}
        t_{1, \alpha_{23}} &= 2 r^2 \sqrt{1 - r^2} s_{\beta} d, \\
        t_{2, \alpha_{23}} &= \left(2 r^2 - 1 \right) \sqrt{1 - r^2} s_{\beta} d^2.
    \end{align*}
    Substituting the obtained expressions for $t_{1, \alpha_{23}}$ and $t_{2, \alpha_{23}}$ in \eqref{eq: expression_alpha_23_CCCCC_path_soln} and simplifying, it follows that
    \begin{align*}
    	\alpha_{23} &= \left(4 r^2 - 1 \right) \sqrt{1 - r^2} s_{\beta} = \beta_{23}.
    \end{align*}
\end{proof}

\textbf{Claim 5:} $\alpha_{33} = \beta_{33}$ for $s_{\phi}$ and $c_{\phi}$ as given in \eqref{eq: expression_sphi_alt_path_CCCCC} and \eqref{eq: expression_cphi_alt_path_CCCCC}, respectively.

\begin{proof}
    Substituting the considered expression for $s_{\phi}$ and $c_{\phi}$ in the expression for $\alpha_{33}$ in \eqref{appeq: alpha33} and simplifying,
    \begin{align} \label{eq: expression_alpha_33_CCCCC_path_soln}
        \alpha_{33} &= r^2 \left(2 r^2 - 1 \right)^2 + 4 r^4 \left(r^2 - 1 \right) c_{\beta} + \frac{t_{2, \alpha_{33}}}{d} + \frac{t_{3, \alpha_{33}}}{d^2},
    \end{align}
    where $t_{2, \alpha_{33}}$ is contributed to by $s_{\phi}$ and $c_{\phi}$ terms in \eqref{appeq: alpha33}, and $t_{3, \alpha_{33}}$ is contributed to by $s^2_{\phi},$ $c^2_{\phi},$ and $s_{\phi} c_{\phi}$ terms in \eqref{appeq: alpha33}. These terms can be obtained as
    \begin{align*}
        t_{2, \alpha_{33}} &= 4 r^2 \left(1 - r^2 \right) \left(1 - 2 r^2 \right) \left(1 + c_{\beta} \right) d, \\
        t_{3, \alpha_{33}} &= \left(1 - r^2 \right) \left(4 r^2 \left(1 - r^2 \right) - \left(2 r^2 - 1 \right)^2 c_{\beta} \right) d^2.
    \end{align*}
    Substituting the obtained expressions for $t_{2, \alpha_{33}}$ and $t_{3, \alpha_{33}}$ in \eqref{eq: expression_alpha_33_CCCCC_path_soln} and simplifying, it follows that
    \begin{align*}
        \alpha_{33} &= r^2 \left(4 r^2 - 3 \right)^2 - \left(1 - r^2 \right) \left(4 r^2 - 1 \right)^2 c_{\beta} = \beta_{33}.
    \end{align*}
\end{proof}

\subsection{Computations for non-optimality of $RLRLRL$ path for $r \leq \frac{\sqrt{3}}{2}$  in Lemma~\ref{lemma: claim2}} \label{appsubsect: CCCCCC_non_optimality_proof}

The equation to be satisfied by the alternately constructed $LRLR$ path to connect the same initial and final configurations as the considered $L_\pi R_{\pi + \beta} L_{\pi + \beta} R_\pi$ subpath of the $RLRLRL$ path is given in \eqref{eq: matrix_equation_CCCCCC_nonoptimality}. The expression for the net rotation matrix corresponding to the $L_\pi R_{\pi + \beta} L_{\pi + \beta} R_\pi$ is given by
\begin{align} \label{eq: CCCCCC_nonoptimality_LRLR_net_matrix}
\begin{split}
    &\mathbf{R}_L (\pi) \mathbf{R}_R (\pi + \beta) \mathbf{R}_L (\pi + \beta) \mathbf{R}_R (\pi) \\
    &= \begin{pmatrix}
        \bar{\beta}_{11} & \bar{\beta}_{12} & \bar{\beta}_{13} \\
        -\bar{\beta}_{12} & \bar{\beta}_{22} & \bar{\beta}_{23} \\
        -\bar{\beta}_{13} & \bar{\beta}_{23} & \bar{\beta}_{33}
    \end{pmatrix},
\end{split}
\end{align}
where 
{\allowdisplaybreaks
\begin{align*}
    \bar{\beta}_{11} &= 1 - 20 r^2 + 75 r^4 - 104 r^6 + 48 r^8 \\
    & \quad\, + 4 r^2 \left(16 r^6 - 32 r^4 + 19 r^2 - 3 \right) c_{\beta} \\
    & \quad\, + r^4 \left(3 - 4 r^2 \right)^2 \cos{\left(2 \beta \right)}, 
    \\
    \bar{\beta}_{12} &= -2 r \left(1 - 5 r^2 + 4 r^4 + r^2 \left(-3 + 4 r^2 \right) c_{\beta} \right) s_{\beta}, \\
    \bar{\beta}_{13} &= r \sqrt{1 - r^2} \Big(-6 + 41 r^2 - 80 r^4 + 48 r^6 \\
    & \hspace{2cm} + \left(-2 + 36 r^2 - 96 r^4 + 64 r^6 \right) c_{\beta} \\
    & \hspace{2cm} + r^2 \left(3 - 16 r^2 + 16 r^4 \right) \cos{(2 \beta)} \Big), \\
    \bar{\beta}_{22} &= 1 - r^2 + r^2 \cos{\left(2 \beta \right)}, \\
    \bar{\beta}_{23} &= 2 r^2 \sqrt{1 - r^2} \left(4 r^2 - 3 + \left(4 r^2 - 1 \right) c_{\beta} \right) s_{\beta}, \\
    \bar{\beta}_{33} &= 1 - 19 r^2 + 75 r^4 - 104 r^6 + 48 r^8 \\
    & \quad\, + 4 r^2 \left(-3 + 19 r^2 - 32 r^4 + 16 r^6 \right) c_{\beta} \\
    & \quad\, + r^2 \left(1 - 4 r^2 \right)^2 \left(r^2 - 1 \right) \cos{(2 \beta)}.
\end{align*}}

The expression for the net rotation matrix corresponding to the $L_\phi R_{\pi - \beta} L_{\pi - \beta} R_\phi$ path is given by
\begin{align*}
    &\mathbf{R}_L (\phi) \mathbf{R}_R (\pi - \beta) \mathbf{R}_L (\pi - \beta) \mathbf{R}_R (\phi) \\
    &= \begin{pmatrix}
        \eta_{11} & \eta_{12} & \eta_{13} \\
        -\eta_{12} & \eta_{22} & \eta_{23} \\
        -\eta_{13} & \eta_{23} & \eta_{33}
    \end{pmatrix},
\end{align*}
where
\begin{align}
\begin{split} \label{appeq: eta11}
    \eta_{11} &= 1 - 12 r^2 + 35 r^4 - 42 r^6 + 18 r^8 \\
    & \quad\, + 4 r^2 \left(-2 + 9 r^2 - 13 r^4 + 6 r^6 \right) c_{\beta} \\
    & \quad\, + 2 r^4 \left(3 r^2 - 2 \right) \left(r^2 - 1 \right) \cos{(2 \beta)} \\
    & \quad\, - 4 r^2 \left(r^2 - 1 \right) s_{\beta} \left[\left(2 r^2 - 1 \right) + 2 r^2 c_{\beta} \right] s_{\phi} \\
    & \quad\, - 4 r^2 \left(r^2 - 1 \right) \Big[\left(2 r^2 - 1 \right) \left(\left(3 r^2 - 2 \right) + r^2 \cos{(2 \beta)} \right) \\
    & \hspace{2.8cm} + \left(8 r^4 - 8 r^2 + 1 \right) c_{\beta} \Big] c_{\phi} \\
    & \quad\, + r^4 \Big[2 \left(3 r^2 - 2 \right) \left(r^2 - 1 \right) + 4 \left(2 r^2 - 1 \right) \left(r^2 - 1 \right) c_{\beta} \\ 
    & \hspace{1.2cm} + \left(1 - 2 r^2 + 2 r^4 \right) \cos{(2 \beta)} \Big] \left(c^2_{\phi} - s^2_{\phi} \right) \\
    & \quad\, + 4 r^4 s_{\beta} \left[2 \left(r^2 - 1 \right) + \left(2 r^2 - 1 \right) c_{\beta} \right] s_{\phi} c_{\phi},
\end{split}
\end{align}
\begin{align}
\begin{split} \label{appeq: eta12}
    \eta_{12} &= 2 r \Big(-2 + 9 r^2 - 13 r^4 + 6 r^6 \\
    & \hspace{0.9cm} + \left(-1 + 9 r^2 - 16 r^4 + 8 r^6 \right) c_{\beta} \\
    & \hspace{0.9cm} + r^2 \left(1 - 3 r^2 + 2 r^4 \right) \cos{(2 \beta)} \Big) s_{\phi} \\
    & \quad\, + 2 r \left(-1 + 3 r^2 - 2 r^4 + 2 r^2 \left(1 - r^2 \right) c_{\beta} \right) s_{\beta} c_{\phi} \\
    & \quad\, + 2 r^3 s_{\beta} \left(2 \left(r^2 - 1 \right) + \left(2 r^2 - 1 \right) c_{\beta} \right) \left(c^2_{\phi} - s^2_{\phi} \right) \\
    & \quad\, + 2 r^3 \Big(-4 + 10 r^2 - 6 r^4 + \left(- 4 + 12 r^2 - 8 r^4 \right) c_{\beta} \\
    & \hspace{1.5cm} + \left(-1 + 2 r^2 - 2 r^4 \right) \cos{(2 \beta)} \Big) s_{\phi} c_{\phi},
\end{split}
\end{align}
\begin{align}
\begin{split} \label{appeq: eta13}
    \eta_{13} &= r \sqrt{1 - r^2} \Big(-2 + 15 r^2 - 30 r^4 + 18 r^6 \\
    & \hspace{2cm} + 12 r^2 \left(1 - 3 r^2 + 2 r^4 \right) c_{\beta} \\
    & \hspace{2cm} + 6 r^4 \left(r^2 - 1 \right) \cos{(2 \beta)} \Big) \\
    & \quad\, + 2 r \sqrt{1 - r^2} \Big(-1 + 4 r^2 - 4 r^4 \\
    & \hspace{2.6cm} + 2 r^2 \left(1 - 2 r^2 \right) c_{\beta} \Big) s_{\beta} s_{\phi} \\
    & \quad\, + 2 r \sqrt{1 - r^2} \Big(2 - 11 r^2 + 20 r^4 - 12 r^6 \\
    & \hspace{2.6cm} + \left(1 - 10 r^2 + 24 r^4 - 16 r^6 \right) c_{\beta} \\
    & \hspace{2.6cm} + r^2 \left(-1 + 4 r^2 - 4 r^4 \right) \cos{(2 \beta)} \Big) c_{\phi} \\
    & \quad\, + r^3 \sqrt{1 - r^2} \cos{(2 \phi)} \Big(4 - 10 r^2 + 6 r^4 \\
    & \hspace{3.8cm} + 4 \left(1 - 3 r^2 + 2 r^4 \right) c_{\beta} \\
    & \hspace{3.8cm} + \left(1 - 2 r^2 + 2 r^4 \right) \cos{(2 \beta)} \Big) \\
    & \quad\, + 4 r^3 \sqrt{1 - r^2} \left(2 \left(r^2 - 1 \right) + \left(2 r^2 - 1 \right) c_{\beta} \right) s_{\beta} s_{\phi} c_{\phi},
\end{split}
\end{align}
\begin{align}
\begin{split} \label{appeq: eta22}
    \eta_{22} &= 1 - 5 r^2 + 10 r^4 - 6 r^6 - 4 r^2 \left(2 r^2 - 1 \right) \left(r^2 - 1 \right) c_{\beta} \\
    & \quad\, + 2 r^4 (1 - r^2) \cos{2 \beta} + 4 r^4 \left(r^2 - 1 \right) \cos{\left(\beta - 2 \phi \right)} \\
    & \quad\, + 2 r^2 \left(3 r^2 - 2 \right) \left(r^2 - 1 \right) \cos{(2 \phi)} + r^6 \cos{(2 \beta - 2 \phi)} \\
    & \quad\, + r^2 \left(r^2 - 1 \right)^2 \cos{\left(2 \beta + 2 \phi \right)} \\
    & \quad\, + 4 r^2 \left(r^2 - 1 \right)^2 \cos{\left(\beta + 2 \phi \right)},
\end{split}
\end{align}
\begin{align}
\begin{split} \label{appeq: eta23}
    \eta_{23} &= 2 r^2 \sqrt{1 - r^2} \left(2 r^2 - 1 + 2 r^2 c_{\beta} \right) s_{\beta} c_{\phi} \\
    & \quad\, + 2 r^2 \sqrt{1 - r^2} s_{\phi} \Big(-2 + 7 r^2 - 6 r^4 \\
    & \hspace{3cm} + \left(-1 + 8 r^2 - 8 r^4 \right) c_{\beta} \\
    & \hspace{3cm} + r^2 \left(1 - 2 r^2 \right) \cos{(2 \beta)} \Big) \\
    & \quad\, + 2 r^2 \sqrt{1 - r^2} \left(c^2_{\phi} - s^2_{\phi} \right) s_{\beta} \Big(2 \left(1 - r^2 \right) \\
    & \hspace{4.8cm} + \left(1 - 2 r^2 \right) c_{\beta} \Big) \\
    & \quad\, + 2 r^2 \sqrt{1 - r^2} s_{\phi} c_{\phi} \Big(4 - 10 r^2 + 6 r^4 \\
    & \hspace{3.5cm} + \left(4 - 12 r^2 + 8 r^4 \right) c_{\beta} \\
    & \hspace{3.5cm} + \left(1 - 2 r^2 + 2 r^4 \right) \cos{(2 \beta)} \Big),
\end{split}
\end{align}
\begin{align}
\begin{split} \label{appeq: eta33}
    \eta_{33} &= 1 - 7 r^2 + 25 r^4 - 36 r^6 + 18 r^8 \\
    & \quad\, + 4 r^2 \left(-1 + 6 r^2 - 11 r^4 + 6 r^6 \right) c_{\beta} \\
    & \quad\, + 2 r^4 \left(1 - 4 r^2 + 3 r^4 \right) \cos{(2 \beta)} \\
    & \quad\, + 4 r^2 \left(-1 + 3 r^2 - 2 r^4 + 2 r^2 \left(1 - r^2 \right) c_{\beta} \right) s_{\beta} s_{\phi} \\
    & \quad\, + 4 r^2 \Big(2 - 9 r^2 + 13 r^4 - 6 r^6 \\
    & \hspace{1.5cm} + \left(1 - 9 r^2 + 16 r^4 - 8 r^6 \right) c_{\beta} \\
    & \hspace{1.5cm} + r^2 \left(-1 + 3 r^2 - 2 r^4 \right) \cos{(2 \beta)} \Big) c_{\phi} \\
    & \quad\, + r^2 \bigg(-4 + 14 r^2 - 16 r^4 + 6 r^6 \\
    & \hspace{1.3cm} + 4 \left(-1 + 4 r^2 - 5 r^4 + 2 r^6 \right) c_{\beta} \\
    & \hspace{1.3cm} + \left(-1 + 3 r^2 - 4 r^4 + 2 r^6 \right) \cos{(2 \beta)} \bigg) \cos{(2 \phi)} \\
    & \quad\, + 4 r^2 \left(2 - 4 r^2 + 2 r^4 + \left(1 - 3 r^2 + 2 r^4 \right) c_{\beta} \right) s_{\beta} s_{\phi} c_{\phi}.
\end{split}
\end{align}

Now, consider the solution for $\phi$ obtained by equating $\eta_{22}$ and $\bar{\beta}_{22},$ whose expression is given in \eqref{eq: solution_phi_nonoptimality_CCCCCC_path}. It is desired to verify that the other terms in the matrix equation in \eqref{eq: matrix_equation_CCCCCC_nonoptimality} match. Therefore, the following five relations must be shown using the obtained solution for $\phi$ to show that the two net rotation matrices in \eqref{eq: matrix_equation_CCCCCC_nonoptimality} are equal:
\begin{enumerate}
    \item $\eta_{11} = \bar{\beta}_{11}$
    \item $\eta_{12} = \bar{\beta}_{12}$
    \item $\eta_{13} = \bar{\beta}_{13}$
    \item $\eta_{23} = \bar{\beta}_{23}$
    \item $\eta_{33} = \bar{\beta}_{33}$
\end{enumerate}
To show the above claims, the expressions for $s_{\phi}$ and $c_{\phi}$ are required. As the solution for $\phi$ from \eqref{eq: solution_phi_nonoptimality_CCCCCC_path} is given by $\phi = \delta - \pi,$ it follows that $s_{\phi} = - s_{\delta}, c_{\phi} = - c_{\delta}$. First, the expression for $C^2 + D^2$ is obtained as
\begin{align*}
    C^2 + D^2 &= \Big(5 - 10 r^2 + 6 r^4 + 4 \left(2 r^2 - 1 \right) \left(r^2 - 1 \right) \cos{\beta} \\
    & \quad\,\,\, - 2 r^2 \left(1 - r^2 \right) \cos{2 \beta} \Big)^2 := \left(g (r, \beta) \right)^2.
\end{align*}
To obtain the expression for the positive square root of $C^2 + D^2,$ the sign of the continuous function $g (r, \beta)$ must be determined for $r \in \left(\frac{1}{\sqrt{2}}, \frac{\sqrt{3}}{2} \right], \beta \in \left(0, \pi \right).$ It should be recalled that $D < 0$ over the considered intervals for $r$ and $\beta$; hence, $C^2 + D^2 = \left(g (r, \beta) \right)^2 > 0$, which implies that $g (r, \beta) \neq 0$ as it is a continuous function. Hence, it is sufficient to obtain the sign of $g (r, \beta)$ at some point $(r, \beta) \in \left(\frac{1}{\sqrt{2}}, \frac{\sqrt{3}}{2} \right], \beta \in (0, \pi).$ Choosing $r = 0.8, \beta = \frac{\pi}{2},$ $g \left(0.8, \frac{\pi}{2} \right) = 1.5184 > 0.$
Therefore, $\sqrt{C^2 + D^2} = 5 - 10 r^2 + 6 r^4 + 4 \left(2 r^2 - 1 \right) \left(r^2 - 1 \right) \cos{\beta} - 2 r^2 \left(1 - r^2 \right) \cos{2 \beta}.$ Hence,
{\allowdisplaybreaks
\begin{align}
\begin{split} \label{eq: expression_sphi_alt_path_CCCCCC}
    s_{\phi} = - s_{\delta} &= - \frac{D}{\sqrt{C^2 + D^2}} \\
    &= - \frac{2 \sin{\beta} \left(2 \left(r^2 - 1 \right) + \left(2 r^2 - 1 \right) \cos{\beta} \right)}{g},
\end{split} \\
\begin{split} \label{eq: expression_cphi_alt_path_CCCCCC}
    c_{\phi} = - c_{\delta} &= - \frac{C}{\sqrt{C^2 + D^2}} \\
    &= - \frac{1}{g} \Big(2 \left(3 r^2 - 2 \right) \left(r^2 - 1 \right) \\
    & \hspace{1.1cm} + 4 \left(2 r^2 - 1 \right) \left(r^2 - 1 \right) \cos{\beta} \\
    & \hspace{1.1cm} + \left(2 r^4 - 2 r^2 + 1 \right) \cos{2 \beta} \Big).
\end{split}
\end{align}}
Using the obtained expressions for $s_{\phi}$ and $c_{\phi}$, the other terms in \eqref{eq: matrix_equation_CCCCCC_nonoptimality} are verified in the claims that follow.

\textbf{Claim 6:} $\eta_{11} = \bar{\beta}_{11}$ for $s_{\phi}$ and $c_{\phi}$ as given in \eqref{eq: expression_sphi_alt_path_CCCCCC} and \eqref{eq: expression_cphi_alt_path_CCCCCC}, respectively.

\begin{proof}
Substituting the considered expression for $s_{\phi}$ and $c_{\phi}$ in the expression for $\eta_{11}$ in \eqref{appeq: eta11} and simplifying,
\begin{align} \label{appeq: eta11_expanded}
\begin{split}
    \eta_{11} &= 1 - 12 r^2 + 35 r^4 - 42 r^6 + 18 r^8 \\
    & \quad\, + 4 r^2 \left(-2 + 9 r^2 - 13 r^4 + 6 r^6 \right) \cos{(\beta)} \\
    & \quad\, + 2 r^4 \left(3 r^2 - 2 \right) \left(r^2 - 1 \right) \cos{(2 \beta)} + \frac{t_{2, \eta_{11}}}{g} + \frac{t_{3, \eta_{11}}}{g^2},
\end{split}
\end{align}
where $t_{2, \eta_{11}}$ is contributed to by the terms corresponding to $s_{\phi}$ and $c_{\phi}$ in $\eta_{11}$ in \eqref{appeq: eta11}, which is simplified as
\begin{align*}
    t_{2, \eta_{11}} &= 4 r^2 \left(r^2 - 1 \right) \Big[\left(8 r^4 - 8 r^2 + 1 \right) c_{\beta} \\
    & \hspace{2.2cm} + \left(2 r^2 - 1 \right) \left(3 r^2 - 2 + r^2 \cos{(2 \beta)} \right) \Big] g,
\end{align*}
and $t_{3, \eta_{11}}$ is contributed to by the terms corresponding to $s^2_{\phi}, c^2_{\phi},$ and $s_{\phi} c_{\phi}$ in \eqref{appeq: eta11}, which is simplified as
\begin{align*}
    t_{3, \eta_{11}} &= r^4 g^2 \Big[4 - 10 r^2 + 6 r^4 + 4 \left(2 r^2 - 1 \right) \left(r^2 - 1 \right) c_{\beta} \\
    & \hspace{1.2cm} + \left(1 - 2 r^2 + 2 r^4 \right) \cos{(2 \beta)} \Big].
\end{align*}
Substituting the obtained expressions for $t_{2, \eta_{11}}$ and $t_{3, \eta_{11}}$ in the expression for $\eta_{11}$ in \eqref{appeq: eta11_expanded}, $\eta_{11}$ can be simplified as
\begin{align*}
    \eta_{11} &= 1 - 20 r^2 + 75 r^4 - 104 r^6 + 48 r^8 \\
    & \quad\, + 4 r^2 \left(16 r^6 - 32 r^4 + 19 r^2 - 3 \right) c_{\beta} \\
    & \quad\, + r^4 \left(3 - 4 r^2 \right)^2 \cos{\left(2 \beta \right)} = \bar{\beta}_{11}.
\end{align*}
\end{proof}

\textbf{Claim 7:} $\eta_{12} = \bar{\beta}_{12}$ for $s_{\phi}$ and $c_{\phi}$ as given in \eqref{eq: expression_sphi_alt_path_CCCCCC} and \eqref{eq: expression_cphi_alt_path_CCCCCC}, respectively.
\begin{proof}
Substituting the considered expression for $s_{\phi}$ and $c_{\phi}$ in the expression for $\eta_{12}$ in \eqref{appeq: eta12} and simplifying,
\begin{align} \label{appeq: eta12_expanded}
    \eta_{12} &= \frac{t_{1, \eta_{12}}}{g} + \frac{t_{2, \eta_{12}}}{g^2},
\end{align}
where $t_{1, \eta_{12}}$ is contributed to by the terms corresponding to $s_{\phi}$ and $c_{\phi}$ in $\eta_{12}$ in \eqref{appeq: eta12}, which is simplified as
\begin{align*}
    t_{1, \eta_{12}} &= -2 r \left(r^2 - 1 \right) \left(-1 + 2 r^2 + 2 r^2 c_{\beta} \right) s_{\beta} g,
\end{align*}
and $t_{2, \eta_{12}}$ is contributed to by the terms corresponding to $s^2_{\phi}, c^2_{\phi},$ and $s_{\phi} c_{\phi}$ in \eqref{appeq: eta12}, which is simplified as
\begin{align*}
    t_{2, \eta_{12}} &= -2 r^3 \left(2 \left(r^2 - 1 \right) + \left(2 r^2 - 1 \right) c_{\beta} \right) s_{\beta} g^2.
\end{align*}
Substituting the obtained expressions for $t_{1, \eta_{12}}$ and $t_{2, \eta_{12}}$ in the expression for $\eta_{12}$ in \eqref{appeq: eta12_expanded}, $\eta_{12}$ can be simplified as
\begin{align*}
    \eta_{12} &= -2 r \left(1 - 5 r^2 + 4 r^4 + r^2 \left(-3 + 4 r^2 \right) c_{\beta} \right) s_{\beta} = \bar{\beta}_{12}.
\end{align*}
\end{proof}

\textbf{Claim 8:} $\eta_{13} = \bar{\beta}_{13}$ for $s_{\phi}$ and $c_{\phi}$ as given in \eqref{eq: expression_sphi_alt_path_CCCCCC} and \eqref{eq: expression_cphi_alt_path_CCCCCC}, respectively.

\begin{proof}
Substituting the considered expression for $s_{\phi}$ and $c_{\phi}$ in the expression for $\eta_{13}$ in \eqref{appeq: eta13} and simplifying,
\begin{align} \label{appeq: eta13_expanded}
\begin{split}
    \eta_{13} &= r \sqrt{1 - r^2} \Big(-2 + 15 r^2 - 30 r^4 + 18 r^6 \\
    & \hspace{2cm} + 12 r^2 \left(1 - 3 r^2 + 2 r^4 \right) c_{\beta} \\
    & \hspace{2cm} + 6 r^4 \left(r^2 - 1 \right) \cos{(2 \beta)} \Big) + \frac{t_{2, \eta_{13}}}{g} + \frac{t_{3, \eta_{13}}}{g^2},
\end{split}
\end{align}
where $t_{2, \eta_{13}}$ is contributed to by the terms corresponding to $s_{\phi}$ and $c_{\phi}$ in $\eta_{13}$ in \eqref{appeq: eta13}, which is simplified as
\begin{align*}
    t_{2, \eta_{13}} &= 2 r \sqrt{1 - r^2} \left(2 r^2 - 1 \right) \left(1 - 8 r^2 + 8 r^4 \right) c_{\beta} g \\
    & \,\,\, + 2 r \sqrt{1 - r^2} \left(2 r^2 - 1 \right)^2 \left(-2 + 3 r^2 + r^2 \cos{(2 \beta)} \right) g,
\end{align*}
and $t_{3, \eta_{13}}$ is contributed to by the terms corresponding to $\cos{(2 \phi)}$ (the expression for which can be obtained using $s^2_{\phi}$ and $c^2_{\phi}$) and $s_{\phi} c_{\phi}$ in \eqref{appeq: eta13}, which is simplified as
\begin{align*}
    t_{3, \eta_{13}} &= r^3 \sqrt{1 - r^2} \Big(4 - 10 r^2 + 6 r^4 + 4 \left(1 - 3 r^2 + 2 r^4 \right) c_{\beta} \\
    & \hspace{2.2cm} + \left(1 - 2 r^2 + 2 r^4 \right) \cos{(2 \beta)} \Big) g^2.
\end{align*}
Substituting the obtained expressions for $t_{2, \eta_{13}}$ and $t_{3, \eta_{13}}$ in the expression for $\eta_{13},$ $\eta_{13}$ can be simplified as
\begin{align*}
    \eta_{13} &= r \sqrt{1 - r^2} \Big(-6 + 41 r^2 - 80 r^4 + 48 r^6 \\
    & \hspace{2.1cm} + \left(-2 + 36 r^2 - 96 r^4 + 64 r^6 \right) c_{\beta} \\
    & \hspace{2.1cm} + r^2 \left(3 - 16 r^2 + 16 r^4 \right) \cos{(2 \beta)} \Big) = \bar{\beta}_{13}.
\end{align*}
\end{proof}

\textbf{Claim 9:} $\eta_{23} = \bar{\beta}_{23}$ for $s_{\phi}$ and $c_{\phi}$ as given in \eqref{eq: expression_sphi_alt_path_CCCCCC} and \eqref{eq: expression_cphi_alt_path_CCCCCC}, respectively.
\begin{proof}
Substituting the considered expression for $s_{\phi}$ and $c_{\phi}$ in the expression for $\eta_{23}$ in \eqref{appeq: eta23} and simplifying,
\begin{align} \label{appeq: eta23_expanded}
    \eta_{23} &= \frac{t_{1, \eta_{23}}}{g} + \frac{t_{2, \eta_{23}}}{g^2},
\end{align}
where $t_{1, \eta_{23}}$ is contributed to by the terms corresponding to $s_{\phi}$ and $c_{\phi}$ in $\eta_{23}$ in \eqref{appeq: eta23}, which is simplified as
\begin{align*}
    t_{1, \eta_{23}} &= 2 r^2 \sqrt{1 - r^2} \left(2 r^2 - 1 + 2 r^2 c_{\beta} \right) s_{\beta} g,
\end{align*}
and $t_{2, \eta_{23}}$ is contributed to by the terms corresponding to $s^2_{\phi}, c^2_{\phi},$ and $s_{\phi} c_{\phi}$ in \eqref{appeq: eta23}, which is simplified as
\begin{align*}
    t_{2, \eta_{23}} &= 2 r^2 \sqrt{1 - r^2} \left(2 \left(r^2 - 1 \right) + \left(2 r^2 - 1 \right) c_{\beta} \right) s_{\beta} g^2.
\end{align*}
Substituting the obtained expressions for $t_{1, \eta_{23}}$ and $t_{2, \eta_{23}}$ in the expression for $\eta_{23}$ in \eqref{appeq: eta23_expanded}, $\eta_{23}$ can be simplified as
\begin{align*}
    \eta_{23} &= 2 r^2 \sqrt{1 - r^2} \left(4 r^2 - 3 + \left(4 r^2 - 1 \right) c_{\beta} \right) s_{\beta} = \bar{\beta}_{23}.
\end{align*}
\end{proof}

\textbf{Claim 10:} $\eta_{33} = \bar{\beta}_{33}$ for $s_{\phi}$ and $c_{\phi}$ as given in \eqref{eq: expression_sphi_alt_path_CCCCCC} and \eqref{eq: expression_cphi_alt_path_CCCCCC}, respectively.
\begin{proof}
    Substituting the considered expression for $s_{\phi}$ and $c_{\phi}$ in the expression for $\eta_{33}$ in \eqref{appeq: eta33} and simplifying,
    \begin{align} \label{appeq: eta33_expanded}
    \begin{split}
        \eta_{33} &= 1 - 7 r^2 + 25 r^4 - 36 r^6 + 18 r^8 \\
        & \quad\, + 4 r^2 \left(-1 + 6 r^2 - 11 r^4 + 6 r^6 \right) c_{\beta} \\
        & \quad\, + 2 r^4 \left(1 - 4 r^2 + 3 r^4 \right) \cos{(2 \beta)} + \frac{t_{2, \eta_{33}}}{g} + \frac{t_{3, \eta_{33}}}{g^2},
    \end{split}
    \end{align}
    where $t_{2, \eta_{33}}$ is contributed to by the terms corresponding to $s_{\phi}$ and $c_{\phi}$ in $\eta_{33}$ in \eqref{appeq: eta33}, which is simplified as
    \begin{align*}
        t_{2, \eta_{33}} &= 4 r^2 \left(r^2 - 1 \right) \left(1 - 8 r^2 + 8 r^4 \right) c_{\beta} g \\
        & \quad\, + 4 r^2 \left(r^2 - 1 \right) \left(2 r^2 - 1 \right) \left(3 r^2 - 2 + r^2 \cos{(2 \beta)} \right) g,
    \end{align*}
    and $t_{3, \eta_{33}}$ is contributed to by the terms corresponding to $\cos{(2 \phi)}$ (the expression for which can be obtained using $s^2_{\phi}$ and $c^2_{\phi}$) and $s_{\phi} c_{\phi}$ in \eqref{appeq: eta33}, which is simplified as
    \begin{align*}
        t_{3, \eta_{33}} &= r^2 \left(r^2 - 1 \right) \Big(4 - 10 r^2 + 6 r^4 + 4 \left(1 - 3 r^2 + 2 r^4 \right) c_{\beta} \\
        & \hspace{2.3cm} + \left(1 - 2 r^2 + 2 r^4 \right) \cos{(2 \beta)} \Big) g^2.
    \end{align*}
    Substituting the obtained expressions for $t_{2, \eta_{33}}$ and $t_{3, \eta_{33}}$ in the expression for $\eta_{33}$ in \eqref{appeq: eta33_expanded}, $\eta_{33}$ can be simplified as
    \begin{align*}
        \eta_{33} &= 1 - 19 r^2 + 75 r^4 - 104 r^6 + 48 r^8 \\
        & \quad\, + 4 r^2 \left(-3 + 19 r^2 - 32 r^4 + 16 r^6 \right) c_{\beta} \\
        & \quad\, + r^2 \left(1 - 4 r^2 \right)^2 \left(r^2 - 1 \right) \cos{(2 \beta)} = \bar{\beta}_{33}.
    \end{align*}
\end{proof}



    

    

\subsection{3D model} \label{appsubsect: 3D_model_details}

We would like to begin with the standard derivation of Rotation Minimizing Frame where roll is eliminated along the curve (but the arc length in both frames is the same). 

\subsubsection{Eliminating roll to obtain Bishop/RMF kinematics}

Let $\hat{\mathbf{R}}(s)=[\mathbf{T}(s)\ \mathbf{U}_1(s)\ \mathbf{U}_2(s)]\in SO(3)$ be an adapted frame along an arc-length parameterized curve $\mathbf{X}(s)$, with $\mathbf{T}=\mathbf{X}'(s)$ and kinematics
\begin{align*}
    \mathbf{T}' = w_z \mathbf{U}_1 - w_y \mathbf{U}_2,\\ \mathbf{U}_1' = w_x \mathbf{U}_2 - w_z \mathbf{T},\\ \mathbf{U}_2' = -w_x \mathbf{U}_1 + w_y \mathbf{T},    
\end{align*}
where $(\cdot)'=d(\cdot)/ds$ and $w_x,w_y,w_z$ are scalar functions of $s$ defining the angular velocity of an object traversing the arc at unit speed along $\mathbf{T}, \mathbf{U}_1, \mathbf{U}_2$ respectively. Define a rotated normal basis $(\mathbf{Y},\mathbf{U})$ by a roll angle $\phi(s)$:\[\mathbf{Y}=\cos\phi\,\mathbf{U}_1+\sin\phi\,\mathbf{U}_2,\qquad \mathbf{U}=-\sin\phi\,\mathbf{U}_1+\cos\phi\,\mathbf{U}_2,\]and let $\mathbf{R}(s)=[\mathbf{T}(s)\ \mathbf{Y}(s)\ \mathbf{U}(s)]\in SO(3)$. Then\[\mathbf{T}' = k_1 \mathbf{Y} + k_2 \mathbf{U},\]where\[k_1 = w_z\cos\phi - w_y\sin\phi,\qquad k_2 = -(w_z\sin\phi + w_y\cos\phi).\]Moreover,\[\mathbf{Y}' = -k_1 \mathbf{T} + (w_x-\phi')\mathbf{U},\qquad \mathbf{U}' = -k_2 \mathbf{T} - (w_x-\phi')\mathbf{Y}.\]Consequently, choosing $\phi'(s)=w_x(s)$ eliminates the tangent-spin (roll) term and yields the rotation-minimizing/Bishop-frame kinematics\[\mathbf{T}' = k_1 \mathbf{Y} + k_2 \mathbf{U},\qquad \mathbf{Y}'=-k_1 \mathbf{T},\qquad \mathbf{U}'=-k_2 \mathbf{T}.\]

\subsubsection{Why two controls are sufficient to specify a 3D curve (up to rigid motion)}

A standard result in differential geometry—the fundamental theorem of space curves—states that a sufficiently regular space curve is determined uniquely, up to a Euclidean rigid motion, by two intrinsic scalar functions (curvature $\kappa(s)$ and torsion $\tau(s)$) specified as functions of arc length $s$ \cite{MathWorldFTSC}. Equivalently (and more convenient for our control formulation), one may describe the curve using a rotation-minimizing (RMF) / Bishop frame, in which the curve is characterized by two scalar “Bishop curvatures” $(k_1(s), k_2(s))$ \cite{Bishop1975,FaroukiSpringerRMF}. In Bishop’s construction, an adapted orthonormal frame $(\mathbf{T}, \mathbf{Y}, \mathbf{U})$ evolves as
\[
\frac{d\mathbf{T}}{ds} = k_1 \mathbf{Y} + k_2 \mathbf{U},\quad
\frac{d\mathbf{Y}}{ds} = -k_1 \mathbf{T},\quad
\frac{d\mathbf{U}}{ds} = -k_2 \mathbf{T},
\]
so that the curvature magnitude is $\kappa = \lVert d\mathbf{T}/ds \rVert = \sqrt{k_1^2 + k_2^2}$ \cite{Bishop1975,FaroukiSpringerRMF}. Moreover, Bishop shows that the curve is characterized (up to congruence) by its “normal development” $(k_1(s), k_2(s))$ in the normal plane \cite{Bishop1975}. This explains why our model naturally involves two control inputs: they are precisely the two components of the curvature vector in an orthonormal basis of the normal plane \cite{Bishop1975,FaroukiSpringerRMF}.

As in the classical 2D Dubins path planning problem, our focus is solely on the \emph{kinematics} of the robot rather than its full attitude dynamics. The simplest and most natural representation of the curve geometry is provided by the Bishop (rotation-minimizing) frame, which we adopt throughout this work.

\subsubsection{What “no roll” means here: it is a frame (gauge) choice for the curve geometry}

The usual convention of “roll” corresponds to rotation about the tangent direction. In moving-frame terms, this is the component of the frame’s angular velocity along the tangent $\mathbf{T}$. A rotation-minimizing adapted frame is defined by requiring that the instantaneous angular velocity have zero component along the tangent, i.e., there is no instantaneous spin about $\mathbf{T}$ as one traverses the curve \cite{Bishop1975,FaroukiSpringerRMF}. Thus, in our formulation, “no roll” should be interpreted as: we work in the Bishop/RMF gauge, in which the tangent spin component is eliminated by construction, leaving only the two normal-plane curvature components \cite{Bishop1975,FaroukiSpringerRMF}. Importantly, allowing an additional “roll control” (spin about $\mathbf{T}$) changes the attitude evolution but does not change the centerline curve in $\mathbb{R}^3$; the curve geometry is governed by the normal-plane curvature vector (two degrees of freedom), while roll is an additional attitude degree of freedom about $\mathbf{T}$. This separation is consistent with the fact that the geometry of a space curve requires only two functions to specify the curve up to rigid motion \cite{MathWorldFTSC,Bishop1975}.

\subsubsection{Arc-length minimization is frame-invariant.}

Arc length is purely geometric: $L = \int_0^L ds$ depends only on the curve and not on the chosen frame. Consequently, if for every spatial curve with a specified bound on curvature (or on the norm of the curvature vector in the normal plane) there is an equivalent curve with corresponding bounds on the Bishop curvatures, then \emph{minimizing arc length is invariant} under the choice of representation—Frenet $(\kappa,\tau)$ or Bishop $(k_1,k_2)$. In particular, whether one considers three angular-velocity components (including spin about $\mathbf{T}$) or uses the two Bishop curvatures, the optimal curve (and its length) is unchanged \cite{Bishop1975,FaroukiSpringerRMF,MathWorldFTSC}. 

An intuitive way to think about it is that the tangent vector (a unit vector along a small piece of the arc length) is perpendicular to the normal plane; including the roll rate in the control only changes the orientation in the roll plane, but does not affect the arc length. Since the objective is about arc-length minimization (and not a general function of the orientation of the body), the choice of not including roll rate is intuitively justified.


\subsubsection{Physical interpretation for robots: curvature $\leftrightarrow$ angular-rate components}

If a vehicle traverses the curve at speed $v$, then $d\mathbf{T}/dt = v\, (d\mathbf{T}/ds)$, so the curvature components correspond to angular-rate components scaled by speed. In moving-frame language, the frame derivative encodes the angular velocity decomposition; standard adapted-frame formulations express the associated skew-symmetric “connection” in terms of these components \cite{Bishop1975,FaroukiSpringerRMF}. In the RMF/Bishop gauge, the angular velocity has no tangent (roll) component, leaving only the two normal-plane components—exactly our two controls \cite{Bishop1975,FaroukiSpringerRMF}.

\vfill

\end{document}